\documentclass[10pt,a4paper,leqno]{article}
\usepackage[english]{babel}
\usepackage[tbtags]{amsmath}
\usepackage{amsthm}
\usepackage{amsfonts}
\usepackage{amssymb}
\usepackage{eucal}
\usepackage{upref}
\usepackage[text={16cm,21cm},centering]{geometry}
\usepackage{varioref}
\newcommand{\abs}[1]{\left\lvert #1\right\rvert}

\renewcommand{\le}{\leqslant}
\renewcommand{\ge}{\geqslant}
\DeclareMathOperator{\RE}{Re}
\renewcommand{\Re}{\RE}
\DeclareMathOperator{\IM}{Im}
\renewcommand{\Im}{\IM}
\labelformat{enumi}{\Alph{enumi}}
\labelformat{enumii}{\alph{enumii}}
\DeclareMathOperator{\Ai}{Ai}
\DeclareMathOperator{\Bi}{Bi}

\DeclareMathOperator{\Arg}{Arg}
\renewcommand{\mid}{\;:\;}
\theoremstyle{definition}
\newtheorem{theorem}{Theorem}
\newtheorem{lemma}[theorem]{Lemma}
\newtheorem{proposition}[theorem]{Proposition}

\newtheorem{definition}[theorem]{Definition}
\theoremstyle{remark}
\newtheorem*{remark}{Remark}
\newcounter{Myenum}
\newenvironment{Mylist}
{\begin{list}{(\Alph{Myenum})}{\usecounter{Myenum}
\setlength{\leftmargin}{0pt}
\setlength{\labelwidth}{0pt}
\setlength{\itemindent}{.5em}
\setlength{\listparindent}{\parindent}
\setlength{\parsep}{0pt}}}
{\end{list}
}
\labelformat{Myenum}{\Alph{Myenum}}
\allowdisplaybreaks

\begin{document}
\begin{center}
\LARGE \bfseries
Global  regularity
\\
 of second order twisted   differential operators
\end{center}

\smallskip
\begin{center}
Ernesto Buzano\,\footnote{\;Dipartimento di Matematica, Universit\`a di Torino (Retired.)} and
Alessandro Oliaro\,\footnote{\;Dipartimento di Matematica, Universit\`a di Torino.}
\end{center}

\begin{abstract}
In this paper  we characterize global regularity in the sense of Shubin of  twisted partial differential operators of second order in dimension $2$. These operators form  a class   containing  the twisted Laplacian, and in bi-unique correspondence  with second order ordinary differential operators with polynomial coefficients and   symbol of degree $2$. This correspondence is established by  a transformation of Wigner type. In this way the global regularity of twisted partial differential operators turns out to be equivalent to global regularity and injectivity of the corresponding ordinary differential operators, which can be completely characterized in terms of the asymptotic behavior of the Weyl symbol.   
In conclusion we observe that we have obtained a new class of globally regular  partial differential operators which is  disjoint from the class of hypo-elliptic operators in the sense of Shubin.
 \\[0.2cm]
\end{abstract}

\section{Introduction}

In this paper we deal with the problem of global regularity for non
hypo-elliptic partial differential operators with polynomial
coefficients. An operator $A:\mathcal{S}^\prime(\mathbb R^n)\rightarrow
\mathcal{S}^\prime(\mathbb R^n)$ is globally regular if
\begin{equation}\label{eqn:1003}
u\in\mathcal{S}(\mathbb R^n) \quad\text{\upshape whenever}\quad Au\in\mathcal{S}(\mathbb R^n).
\end{equation}
It is well known that
hypo-elliptic partial differential operators in the sense of
Definition 25.2 of \cite{Shubin} are globally regular. On the other
hand, the
problem of finding necessary and sufficient conditions for
the
global regularity of a differential operator with polynomial coefficients is still  open.
In the case of
ordinary differential equations, in \cite{Nicola-Rodino:2} necessary and
sufficient
conditions for global regularity are found under additional hypotheses. For partial
differential equations the problem is much more
complicated.

In this paper we study twisted differential operators of second order in $\mathbb{R}^2$, that is, partial differential operators of the kind
\begin{equation}\label{eqn:1000}
A=\sum_{j+k\le 2}(-1)^{j+k}a_{kj} (\alpha D_y-\beta M_x)^j (\gamma D_x-\delta M_y)^k
\end{equation}
with complex coefficients $a_{kj}$, where $D_x=-\mathrm i\partial_x$, $D_y=-\mathrm i\partial_y$, $M_x$ and $M_y$ are the multiplication operators by the corresponding variables $x$ and $y$, and $\alpha,\beta,\gamma,\delta\in\mathbb{R}$ are such that
\begin{equation}\label{eqn:1001}
\alpha\delta-\beta\gamma=1\quad\text{\upshape and}\quad \beta\delta\ne 0.
\end{equation}
An important example is the twisted Laplacian
\begin{equation}\label{eqn:1002}
L=\left( D_x+\frac{1}{2}M_y\right)^2+\left( D_y-\frac{1}{2}M_x\right)^2,
\end{equation}
that can be viewed as a Schr\"odinger operator with magnetic potential. It is well-known that $L$ has a discrete spectrum, consisting of the set of positive odd numbers, and that each of the corresponding eigenspaces is infinite-dimensional. The literature on operators of the kind of \eqref{eqn:1002} is wide. For general results on the twisted Laplacian and its relations with the sublaplacian on the Heisenberg group and the Harmonic Oscillator see for instance \cite{Thangavelu}. In \cite{Koch-Ricci} the eigenspaces of the twisted Laplacian are described and the spectral projections $P_\lambda$ are studied, finding the optimal exponent $\rho(p)$ such that $\Vert P_\lambda u\Vert_{L^p}\le \lambda^{\rho(p)}\Vert u\Vert_{L^2}$, for $p\in[2,\infty]$. Dispersive estimates of the wave flow for the twisted Laplacian (and the Harmonic Oscillator) are investigated in \cite{Dancona-Pierfelice-Ricci}. Moreover, problems related to regularity of the solution of the twisted Laplacian are studied in different frames. In particular, in \cite{Li-Parmeggiani} analytic and Gevrey regularity is analyzed, whereas in \cite{Wong} the global regularity in the sense of \eqref{eqn:1003} is proved, by explicit computation of the heat kernel and Green function. Here we follow a new approach, related to transformations of Wigner type, to characterize global regularity of second order twisted
operators. The approach consists in applying a Wigner-like
transform to a general differential equation. This idea is
already present in some works related to engineering applications,
see \cite{Galleani-Cohen-1}, \cite{Galleani-Cohen-2}. In these
papers some equations are analyzed, looking for the Wigner transform
of the solution. Instead of finding  first a solution $u$, and then
computing its Wigner transform $\mathcal Tu$, the equation itself is
Wigner-transformed  obtaining an equation in $\mathcal Tu$. In this way
it is possible to find,  in some cases, the exact expression of
$\mathcal Tu$.

In this paper, by using the approach of \cite{Galleani-Cohen-1}, \cite{Galleani-Cohen-2} (see also \cite{Cohen},) we establish a link between twisted operators \eqref{eqn:1000} and general second order ordinary differential operators with polynomial coefficients of the form
\begin{equation}\label{eqn:1004}
B=\sum_{j+k\le 2} a_{kj}M^j D^k.
\end{equation}
We call $B$ the \emph{source} of $A$. We prove in Theorem \ref{thm:21} that \eqref{eqn:1000} is globally regular in the sense of \eqref{eqn:1003} if and only if \eqref{eqn:1004} is globally regular and one-to-one as an operator from $\mathcal{S}'(\mathbb{R})$ into $\mathcal{S}'(\mathbb{R})$. In Proposition \ref{pro:47} we give a complete characterization of all operators  \eqref{eqn:1004} that are  globally regular, in terms of the behavior of the complex roots of its Weyl symbol. In particular we avoid the additional hypotheses required in \cite{Nicola-Rodino:2}. Among the operators \eqref{eqn:1004} that are globally regular we then characterize those that are also one-to-one (see Theorem \ref{thm:32}.) This is done through a careful analysis of the asymptotic behavior of the solutions of $Bu=0$. As a consequence we characterize  all the operators \eqref{eqn:1000} that are globally regular. Then we recover as a particular case the global regularity of the twisted Laplacian, (already proved in \cite{Wong},) since the source of the twisted Laplacian is the Harmonic Oscillator, that is globally regular and one-to-one.

As already observed, hypo-elliptic differential operators in the sense of Definition 25.2 of \cite{Shubin} are globally regular. Then, starting from an hypo-elliptic and one-to-one source, the corresponding twisted operator is globally regular. It is worthwhile to stress that twisted operators \eqref{eqn:1000} are never hypo-elliptic, as shown in Proposition \ref{pro:18}, so the class of twisted globally regular operators that we find is completely disjoint from the class of hypo-elliptic operators. Moreover, we observe that there are globally regular twisted operators that have an hypo-elliptic source, as the twisted Laplacian, but not all twisted globally regular operators have an hypo-elliptic source. For example the operator with constant coefficients
\begin{equation*}
B_1=a_{20}D_x^2+a_{10}D_x+a_{00}
\end{equation*}
is globally regular and one-to-one if and only if the polynomial
\begin{equation}\label{eqn:1005}
a_{20}\xi^2+a_{10}\xi+a_{00}
\end{equation}
never vanishes. This is consequence of Theorem \ref{thm:32} below, but it can be easily proved directly since $B_1$, on the Fourier transform side, is the multiplication by \eqref{eqn:1005}. The corresponding twisted operator is
\begin{equation*}
A_1=a_{20}\gamma^2(D_x-\mu M_y)^2-a_{10}\gamma(D_x-\mu M_y)+a_{00},
\end{equation*}
with $\gamma,\mu\in\mathbb{R}$, $\mu\ne 0$. If \eqref{eqn:1005} never vanishes, $A_1$ is then globally regular but its source $B_1$ is never hypo-elliptic. We can find examples of this kind also in the case of sources with variable coefficients. Consider for example the twisted operator
\begin{equation*}
A_2=(\gamma D_x-\delta M_y)^2-\mathrm i(\gamma D_x-\delta M_y)-(\alpha D_y-\beta M_x)^2
\end{equation*}
with source
\begin{equation*}
B_2=D_x^2+\mathrm i D_x-M_x^2.
\end{equation*}
In view of the results of the present paper, for $\alpha,\beta,\gamma,\delta\in\mathbb{R}$ satisfying \eqref{eqn:1001}, both $A_2$ and $B_2$ are globally regular, and $B_2$ is one-to-one, but both $B_2$ and $A_2$ are not hypo-elliptic.

In this paper we only treat the case of second order operators in dimension $2$. Our results can  be probably generalized to dimension greater than $2$,  but this depends on how to  extend  the Definition \ref{def:2} to higher dimensions. On the other hand, the extension of Theorem \ref{thm:16} to operators of order greater than $2$  looks very difficult because already a complete characterization of globally regular ordinary differential operators of  order  greater than $2$ and with polynomial coefficients is an open problem. 

Lastly, since the technique used in this paper to link a source to the corresponding twisted operator recaptures well-known connections between the Harmonic Oscillator and the twisted Laplacian, we think that it can be fruitfully used to prove that results holding for the twisted Laplacian  (see for example \cite{Koch-Ricci}, or \cite{Li-Parmeggiani}) hold in fact for larger classes of operators.

The paper is organized as follows. After  some  basic results in Section  \ref{sec:2}, we study properties of twisted operators and the relations with their sources in Section \ref{sec:3}. The main results on global regularity are proved in Section \ref{sec:4}. As already observed, we need a careful analysis of the asymptotic behavior of the solutions of second order ordinary differential 
equations. As a consequence we then need precise asymptotic expansions of special functions, as well as of their linear combinations, in suitable sectors of the complex plane. Since we have not found in the literature  all the results  in the form we need, for the sake of completeness we prove them in Sections \ref{sec:5} and \ref{sec:6}.

\smallskip We end 
this introduction with some notations and definitions.

Given a subset $S$ of the complex numbers $\mathbb C$, we set $S^\ast=S\setminus\{0\}$.
If $S\subset\mathbb R$, we set $S_+=\{x\in S: x\ge 0\}$, and $S_-=\{x\in S: x\le 0\}$.
Thus in particular $\mathbb Z^\ast_+=\{1,2,\ldots \}$.

To avoid ambiguity due to polar representation of complex numbers we define the  \emph{principal branch of the argument of $z\in\mathbb C^\ast$} as
\begin{equation}\label{eqn:448}
\Arg z=\begin{cases}
2\arctan\frac{\Im z}{\Re z+\abs z},&\text{\upshape if $\Im z\ne 0$ or $\Im z=0$ and $\Re z>0$},\\
\pi,&\text{\upshape if $\Im z=0$ and $\Re z<0$}.\end{cases}.
\end{equation}

Observe that \eqref{eqn:448} implies
\begin{equation*}
\Arg(-z)=\Arg z+\sigma(z)\pi,
\end{equation*}
where
\begin{equation*}
\sigma(z)=\begin{cases}1,&\text{\upshape if $\Arg z\le 0$},\\ -1,&\text{\upshape if $0<\Arg z$}.\end{cases}
\end{equation*}

Given a complex number $\lambda$ we \emph{define}
\begin{equation*}
z^\lambda=\mathrm e^{\lambda\log\abs{z}+\mathrm i \lambda\Arg z},\qquad \text{\upshape for $z\in\mathbb C^\ast$}.
\end{equation*}
With this definition we have
\begin{equation*}
\Arg (z^\lambda)=\Im \lambda\log\abs{z}+\Re \lambda\Arg z\iff -\pi<\Im \lambda\log\abs{z}+(\Re\lambda) \Arg z\le \pi.
\end{equation*}
In particular, given a \emph{real} number $p$ such that  $\abs{p}< 1$,
we have
$\Arg(z^p)=p\Arg z$,
and therefore
$(z^p)^\lambda=z^{\lambda p}$,  for all $\lambda\in \mathbb C$.

\section{Globally regular operators}\label{sec:2}

\begin{definition}
A linear operator $A$ on $\mathcal S'(\mathbb R^n)$ is
\emph{globally  regular}  if
\begin{equation*}
Au\in  \mathcal S(\mathbb R^{n}) \implies u\in \mathcal S(\mathbb R^n),\qquad \text{\upshape for all $u\in\mathcal S'(\mathbb R^n)$}.
\end{equation*}
\end{definition}

We employ standard multi-index notation. In particular, a linear differential operator  $A$ has \emph{symbol}
\begin{equation}\label{eqn:51}
a(x,\xi)=\sum_{\abs{\alpha}\le m} a_\alpha(x)\xi^\alpha,
\end{equation}
if
\begin{equation}\label{eqn:52}
A=\sum_{\abs{\alpha}\le m} a_\alpha(x) D^\alpha,
\end{equation}
with
\begin{equation*}
D_j=-\mathrm i\partial_j,\qquad \text{\upshape for $1\le j\le n$}
\end{equation*}
and $\mathrm i^2=-1$.

\begin{definition}[See \makebox{\cite[Definition 1.3.2]{Nicola-Rodino:1}}]\label{def:1}
A linear differential operator on $\mathcal S'(\mathbb R^n)$, \emph{with polynomial symbol}:
\begin{equation*}
a(x,\xi)=\sum_{\abs{\alpha+\beta}\le m} a_{\alpha,\beta} x^\alpha \xi^\beta
\end{equation*}
is \emph{globally hypo-elliptic}  if $a(x,\xi)$ does not vanish outside a compact set and
\begin{equation}\label{eqn:47}
\lim_{\abs{x}+\abs{\xi}\to \infty} \frac
{\partial_x^\alpha\partial_\xi^\beta a(x,\xi)}
{a(x,\xi)}=0,\qquad \text{\upshape for $\abs{\alpha}+\abs{\beta}=1$}.
\end{equation}
\end{definition}

 \begin{theorem}\label{thm:10}
Assumption \eqref{eqn:47} implies that
\begin{equation}\label{eqn:48}
\lim_{\abs{x}+\abs{\xi}\to\infty}\frac{\partial^\alpha_x\partial^\beta_\xi a(x,\xi)}{a(x,\xi)}=0,\qquad \text{\upshape for $\abs{\alpha}+\abs{\beta}\ge 1$},
\end{equation}
and that there exists $0<m_0\le m$ such that
\begin{equation}\label{eqn:49}
\inf_{(x,\xi)\in\mathbb R^n\times\mathbb R^n}\frac{1+\abs{a(x,\xi)}}{\left(1+\abs{x}+\abs{\xi}\right)^{m_0}}>0.
\end{equation}
\end{theorem}
\begin{proof}
Statement \eqref{eqn:48} follows from Propositions 2.4.1 and 2.4.4 of \cite{Nicola-Rodino:1}.
\end{proof}

\begin{theorem}\label{thm:4}
A globally hypo-elliptic linear differential operator   with polynomial symbol
is globally regular.
\end{theorem}
\begin{proof}
Thanks to  Theorem \ref{thm:10}  the symbol  satisfies the hypothesis of Theorem 25.3 of  \cite{Shubin}.
\end{proof}

\section{Twisted differential operators}\label{sec:3}

Define the \emph{multiplication operators}
\begin{equation*}
M_1u(x,y)=M_xu(x,y)=xu(x,y),\qquad
M_2u(x,y)=M_yu(x,y)=yu(x,y),
\end{equation*}
where $u\in\mathcal S'(\mathbb R^2)$.

The \emph{twisted Laplacian}
\begin{equation}\label{eqn:392}
\left( D_x+\frac 1 2 M_y\right)^{\!\!2}+\left( D_y-\frac 1 2 M_x\right)^{\!\!2}
\end{equation}
 is  an important example of an operator which is globally regular
but
 not  globally hypo-elliptic (see \cite{Wong}.)

\begin{definition}\label{def:2}
A  \emph{twisted differential operator of order $m$}  is a linear differential operator on $\mathbb R^2$
of the kind
\begin{equation}\label{eqn:302}
A=\sum_{j+k\le m}(-1)^{j+k}a_{kj} (\alpha D_y-\beta M_x)^j (\gamma D_x-\delta M_y)^k,
\end{equation}
 where  $\alpha,\beta,\gamma,\delta$ are real numbers such that
\begin{equation}\label{eqn:301}
\alpha\delta-\beta\gamma=1\quad\text{\upshape and}\quad \beta\delta\ne 0
\end{equation}
and the coefficients $a_{k,j}$ are complex numbers such that $\sum_{j+k= m}\abs{a_{kj}}\ne 0$.
\end{definition}
For example, if we set
\begin{equation*}
m=2,\quad a_{20}=a_{02}=1,\quad  a_{jk}=0,\quad\text{\upshape for $j,k\le 1$},
\end{equation*}
and
\begin{equation*}
\alpha=-1,\quad \beta=-\frac 12,\quad \gamma=1,\quad \delta=-\frac 12,
\end{equation*}
the operator  \eqref{eqn:302} becomes the twisted Laplacian \eqref{eqn:392}.

\bigskip
The class of twisted differential operators is  completely disjoint from the class of  globally hypo-elliptic operators.
\begin{proposition}\label{pro:18}
Twisted differential operators are \emph{never}  globally hypo-elliptic.
\end{proposition}
\begin{proof}
By Theorem 3.4 of
\cite{Shubin} we have that the symbol of the operator \eqref{eqn:302}
is given by
\begin{align*}
a(x,y;\xi,\eta) &=\sum_{j+k\le m} (-1)^{j+k}a_{kj}
\sum_{n\in\mathbb{Z}_+} \frac{(-\mathrm i)^n}{n!}
\partial^n_\eta (\alpha \eta-\beta x)^j \partial^n_y
(\gamma \xi-\delta y)^k \\
&= \sum_{j+k\le m} (-1)^{j+k}a_{kj} \sum_{n\le \min\{j,k\}}
(\mathrm i\alpha\delta)^n  \binom{j}{n}
\binom{k}{n}n! (\alpha \eta-\beta x)^{j-n}
(\gamma\xi-\delta y)^{k-n}.
\end{align*}
Since $a$ is constant along the plane
\begin{equation*}
\begin{cases}\alpha\eta-\beta x=0,\\
\gamma\xi-\delta y=0,\end{cases}
\end{equation*}
we have that the operator \eqref{eqn:302} cannot be  globally hypo-elliptic.
\end{proof}

\bigskip
Given four real numbers $\alpha,\beta\,\gamma,\delta$ satisfying \eqref{eqn:301}, define the integral transform of a function
$u\in\mathcal S(\mathbb R^2)$:
\begin{equation*}
\mathcal Tu(x,y)=
(2\pi)^{-\frac 12}\int_{\mathbb R}  \mathrm e^{- \mathrm i z y}u(\beta x+\alpha z,\beta x+\beta\gamma\delta^{-1}z)\,\mathrm d z.
\end{equation*}
A simple computation shows that $\mathcal T$ is an isomorphism on $\mathcal S(\mathbb R^2)$ with inverse given by
\begin{equation*}
\mathcal T^{-1}v(x,y)=(2\pi)^{-\frac 12}\int_{\mathbb R}  \mathrm e^{ \mathrm i t \delta(x-y)}v(\alpha\delta\beta^{-1} y-\gamma x,t)\, \mathrm d t.
\end{equation*}

Since $\mathcal T$ and  its inverse extend  to  $\mathcal S'(\mathbb R^2)$, we may define the transform  of an operator $A$ on $\mathcal S'(\mathbb R^2)$  as
\begin{equation*}
\mathcal T[A]=\mathcal T A \mathcal T^{-1}.
\end{equation*}
Of course this transformation is invertible, with inverse given by
\begin{equation*}
\mathcal T^{-1}[B]=\mathcal T^{-1} B \mathcal T.
\end{equation*}

Since $\mathcal T$ is an isomorphism on $\mathcal S(\mathbb R^2)$ and on $\mathcal S'(\mathbb R^2)$, we have that
\begin{equation}\label{eqn:303}
\text{\upshape $A$ is globally regular $\iff$ $\mathcal T[A]$ is globally regular.}
\end{equation}

Compute
\begin{align}
& \label{eqn:221}\begin{aligned}[t]
&D_1\mathcal Tu(x,y)=(2\pi)^{-\frac 12}\int_{\mathbb R} \mathrm e^{-\mathrm iz y} D_x\left(u(\beta x+\alpha z,\beta x+\beta\gamma\delta^{-1}z)\right)\,\mathrm d z
\\
&\qquad = \beta\mathcal TD_1 u(x,y)+\beta\mathcal TD_2 u(x,y),
\end{aligned}
\\ &\label{eqn:222}
\begin{aligned}[t]
&D_2\mathcal Tu(x,y)=(2\pi)^{-\frac 12}\int_{\mathbb R} -z \mathrm e^{-\mathrm iz y} u(\beta x+\alpha z,\beta x+\beta\gamma\delta^{-1}z)\,\mathrm d z
\\
&\qquad=-(2\pi)^{-\frac 12}\int_{\mathbb R}  \mathrm e^{-\mathrm iz y} \delta\left(\beta x+\alpha z\right)
u(\beta x+\alpha z,\beta x+\beta\gamma\delta^{-1}z)\,\mathrm d z
\\
&\qquad\quad+(2\pi)^{-\frac 12}\int_{\mathbb R}  \mathrm e^{-\mathrm iz y} \delta\left(\beta x+\beta\gamma\delta^{-1} z\right)
u(\beta x+\alpha z,\beta x+\beta\gamma\delta^{-1}z)\,\mathrm d z
\\
 &\qquad=-\delta\mathcal TM_1u(x,y)+\delta\mathcal TM_2u(x,y),
\end{aligned}
\\ &\label{eqn:223}
\begin{aligned}[t]
&M_1\mathcal Tu(x,y)=(2\pi)^{-\frac 12}\int_{\mathbb R} \mathrm e^{-\mathrm iz y} x u(\beta x+\alpha z,\beta x+\beta\gamma\delta^{-1}z)\,\mathrm d z
\\
&\qquad=-(2\pi)^{-\frac 12}\int_{\mathbb R}  \mathrm e^{-\mathrm iz y} \gamma(\beta x+\alpha z)
u(\beta x+\alpha z,\beta x+\beta\gamma\delta^{-1}z)\,\mathrm d z
\\
&\qquad\quad+(2\pi)^{-\frac 12}\int_{\mathbb R}  \mathrm e^{-\mathrm iz y} \alpha\delta\beta^{-1}(\beta x+\beta\gamma\delta^{-1}z)
u(\beta x+\alpha z,\beta x+\beta\gamma\delta^{-1}z)\,\mathrm d z
\\
&\qquad=  -\gamma \mathcal TM_1 u(x,y)+\alpha\delta\beta^{-1}\mathcal TM_2 u(x,y),
\end{aligned}
\\ &\label{eqn:224}
\begin{aligned}[t]
&M_2\mathcal Tu(x,y)=(2\pi)^{-\frac 12}\int_{\mathbb R} \left(-D_z\mathrm e^{-\mathrm iz y}\right) u(\beta x+\alpha z,\beta x+\beta\gamma\delta^{-1}z)\,\mathrm d z
\\
&\qquad=(2\pi)^{-\frac 12} \int_{\mathbb R}  \mathrm e^{-\mathrm iz y}
\alpha D_1 u(\beta x+\alpha z,\beta x+\beta\gamma\delta^{-1} z)\,\mathrm d z
\\
&\qquad\quad
+(2\pi)^{-\frac 12} \int_{\mathbb R}  \mathrm e^{-\mathrm iz y}
\beta\gamma\delta^{-1} D_2 u(\beta x+\alpha z,\beta x+\beta\gamma\delta^{-1} z)\,\mathrm d z
\\
&\qquad=  \alpha \mathcal TD_1 u(x,y)+\beta\gamma\delta^{-1} \mathcal TD_2 u(x,y).
\end{aligned}
\end{align}
It follows that
\begin{alignat*}{2}
&\mathcal T[M_x]=
-\alpha D_y+\beta M_x,
\\
&\mathcal T[D_x]
= -\gamma D_x+\delta M_y,
\end{alignat*}
and more generally
the twisted differential operator \eqref{eqn:302}
can be written as
\begin{equation}\label{eqn:310}
A=\mathcal T[\check A]
\end{equation}
where
\begin{equation}\label{eqn:311}
\check A=\sum_{j+k\le m}a_{kj} M_x^j D_x^k.
\end{equation}
Observe that $\check A$ is an operator on $\mathbb R^2$, acting only on the first variable:
\begin{equation*}
\check Au(x,y)=\sum_{j+k\le m}a_{kj} x^j D_x^ku(x,y).
\end{equation*}

 Recall now that  $\mathcal S(\mathbb R^2)$ is the tensor product of $\mathcal S(\mathbb R)$ by $\mathcal S(\mathbb R)$.
This means that
$\mathcal S(\mathbb R^2)$ is the completion $\mathcal S(\mathbb R)\widehat\otimes\mathcal S(\mathbb R)$ of  the space $\mathcal S(\mathbb R)\otimes\mathcal S(\mathbb R)$
of linear combinations of products
\begin{equation*}
(f\otimes g)(x,y)=f(x)g(y).
\end{equation*}
The same is true for temperate distributions:
\begin{equation*}\mathcal S'(\mathbb R^2)=\left(\mathcal S(\mathbb R)\widehat\otimes\mathcal S(\mathbb R)\right)'=\mathcal S'(\mathbb R)\widehat\otimes\mathcal S'(\mathbb R).
\end{equation*}

Given two continuous linear operators $A_1$ and $A_2$ on $\mathcal S'(\mathbb R)$, there exists a
unique continuous linear operator
$A_1\widehat \otimes A_2$ on $\mathcal S'(\mathbb R)\widehat\otimes\mathcal S'(\mathbb R)=\mathcal S'(\mathbb R^2)$
such that
\begin{equation*}
(A_1\widehat\otimes A_2)(u_1\otimes u_2)=A_1 u_1\otimes A_2 u_2,\quad  (u_1,u_2)\in  \mathcal S'(\mathbb R)\times \mathcal S'(\mathbb R).
\end{equation*}
If $A_1$ and $A_2$ are continuous on $\mathcal S(\mathbb R)$,
the tensor product $A_1\widehat \otimes A_2$ is continuous on
$\mathcal S(\mathbb R)\widehat\otimes\mathcal S(\mathbb R)=\mathcal S(\mathbb R^2)$.

Define  the operators on $\mathcal S'(\mathbb R)$:
\begin{equation*}
Du(x)=-\mathrm i u'(x),\qquad Mu(x)=xu(x),\qquad  Iu(x)=u(x).
\end{equation*}
then we have
\begin{equation*}
M_x=M\widehat \otimes I,\qquad D_x=D\widehat\otimes I,
\end{equation*}
and more generally
\begin{equation*}
\sum_{j+k\le m}a_{kj}  M_x^j D_x^k=\left(\sum_{j+k\le m}a_{kj}  M^j D^k\right)\widehat\otimes I.
\end{equation*}
In other words, if we keep into account \eqref{eqn:302}, \eqref{eqn:310} and \eqref{eqn:311}, we  obtain  the following identity:
\begin{equation*}
A=\mathcal T[\tilde A\widehat \otimes I]
\end{equation*}
where $A$ is the operator \eqref{eqn:302} and
\begin{equation}\label{eqn:314}
\tilde A=\sum_{j+k\le m}a_{kj}  M^j D^k.
\end{equation}
\begin{definition}
The ordinary differential operator $\tilde A$ defined in \eqref{eqn:314} is the \emph{source} of the twisted differential operator $A$ given by \eqref{eqn:302}.
\end{definition}

We always consider the kernel of the source $\tilde A$ in the sense of temperate distributions:
\begin{equation*}
\ker \tilde A=\{u\in \mathcal S'(\mathbb R)\mid \tilde Au=0\}.
\end{equation*}
Observe that $\ker \tilde A\subset \mathcal S(\mathbb R)$, if $\tilde A$ is globally regular.

\bigskip
From \eqref{eqn:303}, we obtain the following proposition.
\begin{proposition}\label{pro:23}
A twisted differential operator $A$
is globally regular if and only if
 $\tilde A\widehat \otimes I$
is globally regular.\end{proposition}

\begin{proposition}\label{pro:22}
The source of a globally regular twisted differential operator
is   globally regular and one-to-one.

In particular a globally regular twisted differential operator
is one-to-one.
\end{proposition}
\begin{proof}
Let $A$ be the twisted operator. We know from Proposition \ref{pro:23} that $\tilde A\widehat \otimes I$ is globally regular.

Consider $u\in \mathcal S'(\mathbb R)$ such that $\tilde Au\in\mathcal
S(\mathbb R)$. Then $(\tilde A\widehat \otimes I) (u\otimes v)=(\tilde A u)\otimes
v\in \mathcal S(\mathbb R^2)$ for all $v\in \mathcal S(\mathbb R)$.
Since $\tilde A\widehat\otimes I$ is globally regular, $u\otimes v$ must
belong to $\mathcal S(\mathbb R^2)$ for all $v\in \mathcal S(\mathbb
R)$. But this is impossible, unless $u$ belongs to $\mathcal
S(\mathbb R)$. In fact , given $v\in \mathcal S(\mathbb R)$ such that $v(0)=1$, let $(\psi_n)$ be a sequence in $\mathcal S(\mathbb R)$ converging to the Dirac distribution $\delta$. Then for all $\phi\in\mathcal S(\mathbb R)$, we have
\begin{equation*}\begin{aligned}
\int_{\mathbb R}(u\otimes v)(x,0)\phi(x)\,\mathrm d x
&=
\lim_{n\to+\infty} \int_{\mathbb R}\left\{\int_{\mathbb R} (u\otimes v)(x,y)\phi(x)\,\mathrm d x\right\}\psi_n(y)\,\mathrm dy
\\&=\lim_{n\to+\infty} \int_{\mathbb R^2} (u\otimes v)(x,y)\,(\phi\otimes \psi_n)(x,y)\,\mathrm d x\mathrm dy
\\&=\langle u\,\vert\, \phi\rangle \lim_{n\to+\infty} \int_{\mathbb R} v(y) \psi_n(y)\,\mathrm dy= \langle u\,\vert\, \phi\rangle.
\end{aligned}\end{equation*}
But this means that $u(x)=(u\otimes v)(x,0)\in\mathcal S(\mathbb R)$.

Now we show that $\tilde A$ is one-to-one. Assume there exists $\phi\in\mathcal S(\mathbb R)\setminus
\{0\}$ such that $\tilde A\phi=0$. Then $\phi\otimes \delta$ belongs to the kernel of
$\tilde A\widehat\otimes I$, but not to $\mathcal S(\mathbb R^2)$, in
contradiction with the global regularity of $\tilde A\widehat\otimes I$.

If $A$ is the globally regular twisted differential operator with source $\tilde A$, we have that $\ker \tilde A=0$. Then $\ker A=\mathcal T\left((\ker \tilde A)\widehat \otimes I\right)=0$, that is $A$ is one-to-one.
\end{proof}

Denote by $(\tilde A)'$ the transpose of the  source \eqref{eqn:314}:
\begin{equation*}
(\tilde A)'=\sum_{j+k\le m}(-1)^k a_{kj} D^k M^j.
\end{equation*}
Observe that $\tilde A$ and $(\tilde A)'$ are dual to each other, that is $(\tilde A)''=\tilde A$. In other words, we have
\begin{equation*}
\begin{cases}\langle (\tilde A)'u\,\vert\,\phi\rangle=\langle u\,\vert\, \tilde A\phi\rangle,\\
\langle \tilde A u\,\vert\,\phi\rangle=\langle u\,\vert\, (\tilde A)'\phi\rangle,\end{cases}\qquad \text{\upshape for all $u\in\mathcal S'(\mathbb R)$, and $\phi\in\mathcal S(\mathbb R)$.}
\end{equation*}

Recall now the following Theorem of \cite{Nacinovich}.
\begin{theorem}\label{thm:22}
An ordinary differential operator with polynomial coefficients, has closed range in   $\mathcal S(\mathbb R)$ and $\mathcal S'(\mathbb R)$.
\end{theorem}

Thanks to Theorem \ref{thm:22},   the images $\tilde A\left(\mathcal S(\mathbb R)\right)$
and $\tilde A\left(\mathcal S'(\mathbb R)\right)$ are closed  subspaces of  $\mathcal S(\mathbb R)$ and $\mathcal S'(\mathbb R)$, respectively.
Then by Closed Range Theorem  \cite[Theorem 1.2]{Browder}, it follows that
\begin{equation*}
\tilde A\left(\mathcal S(\mathbb R)\right)=\left\{f\in \mathcal S(\mathbb R)\mid \langle \phi\,\vert\, f\rangle=0,\,\forall \phi\in\ker (\tilde A)'\right\}
\end{equation*}
and
\begin{equation*}
\tilde A\left(\mathcal S'(\mathbb R)\right)=\left\{f\in \mathcal S'(\mathbb R)\mid \langle f\,\vert\, \phi\rangle=0,\,\forall \phi\in\ker (\tilde A)'\cap \mathcal S(\mathbb R)\right\}.
\end{equation*}
Since $\ker (\tilde A)'$ is finite-dimensional, both $\tilde A\left(\mathcal S(\mathbb R)\right)$ and $\tilde A\left(\mathcal S'(\mathbb R)\right)$ have a topological supplementary, we can choose as follows.
Fix a basis $\phi_1,\ldots,\phi_n$ of $\ker (\tilde A)'$, and let $\psi_1,\ldots,\psi_n$ be functions in $\mathcal S(\mathbb R)$ such that $\langle\phi_j\,\vert\,\psi_k\rangle=\delta_{jk}$ for $j,k\in\{1,\ldots,n\}$. Let $\mathcal N((\tilde A)')$ be the subspace of $\mathcal S(\mathbb R)$ generated by $\psi_1,\ldots,\psi_n$. Then
\begin{equation}\label{eqn:385}
\mathcal S(\mathbb R)=\tilde A\left(\mathcal S(\mathbb R)\right)\oplus \mathcal N((\tilde A)').
\end{equation}
Without loss of generality, we can assume that $\ker (\tilde A)'\cap \mathcal S(\mathbb R)$ either equals $0$ or it is generated by $\phi_1,\ldots,\phi_m$, with $m\le n$. Then
\begin{equation}\label{eqn:386}
\mathcal S'(\mathbb R)=\tilde A\left(\mathcal S'(\mathbb R)\right)\oplus \mathcal M((\tilde A)'),
\end{equation}
where $\mathcal M((\tilde A)')$ is either $0$ or the subspace of $\mathcal N((\tilde A)')$ generated by $\psi_1,\ldots,\psi_m$.

Moreover, by Propositions 43.7 and 43.9 of \cite{Treves}, it follows from \eqref{eqn:385} and \eqref{eqn:386} that
\begin{equation}
\label{eqn:226}
\mathcal S(\mathbb R^2)
=
(\tilde A\widehat \otimes I)\mathcal S(\mathbb R^2)\oplus \mathcal N((\tilde A)')\widehat\otimes\mathcal S(\mathbb R)
\end{equation}
and
\begin{equation}
\label{eqn:227}
\mathcal S'(\mathbb R^2)=
(\tilde A\widehat\otimes I)\mathcal S'(\mathbb R^2)\oplus \mathcal M((\tilde A)')\widehat \otimes\mathcal S'(\mathbb R).
\end{equation}

\begin{proposition}\label{pro:48}
Given a twisted differential operator $A$, the images
$A\left(\mathcal S(\mathbb R^2)\right)$ and $A\left(\mathcal S'(\mathbb R^2)\right)$ are closed subspaces of $\mathcal S(\mathbb R^2)$ and $\mathcal S'(\mathbb R^2)$ respectively.
\end{proposition}
\begin{proof}
Let $\tilde A$ be the source of $A$. Then $A=\mathcal T[\tilde A\widehat \otimes I]$.  Since $\mathcal T$ is an automorphism of $\mathcal S(\mathbb R^2)$ and  of $\mathcal S'(\mathbb R^2)$, the closure of the images follows from
 \eqref{eqn:226} and \eqref{eqn:227}.
\end{proof}

\begin{proposition}\label{pro:40}
Given a twisted differential operator $A$ the following conditions are equivalent.
\begin{enumerate}
\item\label{itm:11}
$\ker \tilde A\subset\mathcal S(\mathbb R)$  and $\ker (\tilde A)'\subset \mathcal S(\mathbb R)$.
\item\label{itm:13}
$\tilde A$ and $(\tilde A)'$ are globally regular.
\end{enumerate}
\end{proposition}
\begin{proof}
It is clear that \eqref{itm:13} $\implies$ \eqref{itm:11}.

Let us prove that \eqref{itm:11} implies that $(\tilde A)'$ is globally regular.
Consider  $u\in \mathcal S'(\mathbb R)$ such that $f=(\tilde A)'u\in \mathcal S(\mathbb R)$.
By the dual to \eqref{eqn:385}, there exist $v\in\mathcal S(\mathbb R)$ and $h\in\mathcal N(\tilde A)$ such that $f=(\tilde A)' v+h$. Since $\ker \tilde A\subset \mathcal S(\mathbb R)$, we have $\mathcal N(\tilde A)=\mathcal M(\tilde A)$. Then the dual to \eqref{eqn:386} implies that $h=0$, that is that $u-v\in\ker (\tilde A)'\subset \mathcal S(\mathbb R)$. Since $v\in\mathcal S(\mathbb R)$ also $u\in\mathcal S(\mathbb R)$.

The proof that  \eqref{itm:11} implies that $\tilde A$ is globally regular is very similar and is left to the reader.
\end{proof}
\begin{theorem}\label{thm:19}
Consider a twisted differential operator $A$.
If $\ker \tilde A=0$ and $\ker (\tilde A)'\subset \mathcal S(\mathbb R)$, the operator $A$ is
 globally regular.
\end{theorem}
\begin{proof}
Thanks to Proposition \ref{pro:23} it is sufficient to prove that $\tilde A\widehat \otimes I$ is globally regular.

Consider  $u\in\mathcal S'(\mathbb R^2)$ such that $f=(\tilde A\widehat \otimes I) u\in\mathcal S(\mathbb R^2)$. Thanks to Proposition \ref{pro:40}, $(\tilde A)'$ is globally regular.  Since $f$ belongs to $\mathcal S(\mathbb R^2)$, by \eqref{eqn:226} there exist $v\in \mathcal S(\mathbb R^2)$ and $h\in  \mathcal N((\tilde A)')\widehat\otimes \mathcal S(\mathbb R)$
such that
$(\tilde A\widehat\otimes I)u=(\tilde A\widehat\otimes I)v+h$.
Since  $\ker (\tilde A)'\subset \mathcal S(\mathbb R)$, we have $\mathcal M((\tilde A)')=\mathcal N((\tilde A)')$ and
identity \eqref{eqn:227} implies that $h=0$.  Then $u=v\in\mathcal S(\mathbb R^2)$, because $\ker (\tilde A\widehat \otimes I)=(\ker \tilde A)\widehat\otimes \mathcal S(\mathbb R)=0$.
\end{proof}

\section{Global regularity of second order twisted differential operators}\label{sec:4}
\subsection{Statement of the results}

Global regularity of second order twisted differential operators can be characterized in a rather complete way.
We state two theorems, which are the main results of the paper. We prove these theorems  in Subsections \ref{subsec:1}, and \ref{subsec:2}.

\bigskip
Consider the second order twisted differential operator
\begin{equation*}
A=\sum_{j+k\le 2}(-1)^{j+k}a_{kj} (\alpha D_y-\beta M_x)^j (\gamma D_x-\delta M_y)^k,
 \end{equation*}
with source
\begin{equation*}
\tilde A=\sum_{j+k\le 2}a_{kj} M^jD^k.
 \end{equation*}

\begin{theorem}\label{thm:21}The following statements are equivalent.
\begin{enumerate}
\item\label{itm:17}
$A$ is globally regular.
\item\label{itm:18}
$\ker \tilde A=0$, and $\tilde A$ is  globally regular.
\item\label{itm:19}
$\ker \tilde A=0$, and $(\tilde A)'$ is  globally regular.
\item\label{itm:20}
$\ker \tilde A=0$, and $\ker (\tilde A)'\subset \mathcal S(\mathbb R)$.
\end{enumerate}
\end{theorem}

\begin{definition}
Two polynomials $p(x,\xi)$ and $q(x,\xi)$ are symplectically equivalent  if there exists a symplectic transformation\,\footnote{\
In dimension $2$ a symplectic transformation is a linear map with determinant equal to $1$.}
$\chi$ such that  $q=p\circ \chi$.
\end{definition}
\begin{lemma}\label{lem:11}
For any polynomial
\begin{equation*}
p(x,\xi)=\sum_{j+k\le 2} p_{kj}x^j\xi^k,
\end{equation*}
such that $\abs{p_{20}}+\abs{p_{11}}+\abs{p_{02}}>0$,
there is an infinite number of polynomials
\begin{equation*}
q(x,\xi)=\sum_{j+k\le 2} q_{kj}x^j\xi^k,
\end{equation*}
symplectically equivalent to $p$ and such that
$q_{20}\ne 0$.
\end{lemma}
\begin{proof}
It is sufficient to consider $\chi(x,\xi)=(x+\theta\xi,\xi)$, where $\theta\in\mathbb R$ is such that $p_{20}+\theta p_{11}+\theta^2 p_{02}\ne 0$.
\end{proof}
Recall that  the  \emph{Weyl symbol} (see \cite[Definition 23.5]{Shubin}) of a differential operator
\begin{equation*}
P=p_{20}D^2+p_{11}MD+p_{02}M^2+p_{10}D+p_{01}M+p_{00}I
\end{equation*}
is given by
\begin{equation*}
p(x,\xi)=p_{20}\xi^2+p_{11}x\xi+p_{02}x^2+p_{10}\xi+p_{01}x+p_{00}+\frac {\mathrm i}2\, p_{11}.
\end{equation*}
Denote by $\mathcal B$ the set of polynomials
\begin{equation*}
b(x,\xi)=b_{20}\xi^2+b_{11}x\xi+b_{02}x^2+b_{10}\xi+b_{01}x+b_{00}+\frac {\mathrm i}2\, b_{11},
\end{equation*}
 with $b_{20}\ne 0$, and symplectically equivalent to the Weyl symbol of $\tilde A$.

Since the order of $A$ is $2$, we have $\abs{a_{20}}+\abs{a_{11}}+\abs{a_{02}}> 0$. Then Lemma \ref{lem:11} implies that
$\mathcal B\ne \emptyset$.

For all $b\in\mathcal B$, set
\begin{equation}\label{eqn:801}
\begin{cases}
\Delta_2=b_{11}^2-4b_{20}b_{02}, \\
\Delta_1=2b_{11}b_{10}-4b_{20}b_{01}, \\
\Delta_0=b_{10}^2-4b_{20}b_{00}-2\mathrm i b_{20}b_{11},
\end{cases}
\end{equation}
\begin{equation*}
\lambda=\frac 18\left(-\frac {\Delta_2}{b_{20}^2}\right)^{\!\!-\frac 32}\frac {\Delta_1^2-4\Delta_2\Delta_0}{b_{20}^4},
\end{equation*}
and
\begin{equation}\label{eqn:818}
\Xi_\pm(x)=\begin{cases}\displaystyle
-\frac 12 \left\{\frac {b_{11}}{b_{20}}x+\frac {b_{10}}{b_{20}}\pm
\sigma\!\left(\frac {\Delta_2}{b_{20}^2}\right)\left(\frac {\Delta_2}{b_{20}^2}\right)^{\!\!\frac 12}
x\left(1+\frac {\Delta_1}{\Delta_2x}+\frac {\Delta_0}{\Delta_2 x^2}\right)^{\!\!\frac 12}\right\},
&\text{\upshape if $\Delta_2\ne 0$},
\\[10pt] \displaystyle
-\frac 12\left\{\frac {b_{11}}{b_{20}}\,x+\frac {b_{10}}{b_{20}}
\pm\sigma\!\left(\frac{\Delta_1}{b_{20}^2}\right)\left(\frac{\Delta_1}{b_{20}^2}\right)^{\!\!\frac 12}
x^{\frac 12}\left(1+\frac{\Delta_0}{\Delta_1x}\right)^{\!\!\frac 12}\right\},
&\text{\upshape if $\Delta_2= 0\ne \Delta_1$},
\\[10pt] \displaystyle
-\frac 12\left\{\frac {b_{11}}{b_{20}}\,x+\frac {b_{10}}{b_{20}}
\pm \sigma\!\left(\frac {\Delta_0}{b_{20}^2}\right)\left(\frac {\Delta_0}{b_{20}^2}\right)^{\!\!\frac 12}\right\},
&\text{\upshape if $\Delta_2=\Delta_1= 0$}.\end{cases}
\end{equation}
$\xi=\Xi_\pm$ are  the complex roots of the Weyl symbol of $B$:
\begin{equation*}
b(x,\xi)=b_{20}\xi^2+(b_{11}x+b_{10})\xi+b_{02}x^2+b_{01}x+b_{00}+\frac {\mathrm i}2 \,b_{11}.
\end{equation*}

\begin{theorem}\label{thm:16}
The following conditions are equivalent.
\begin{enumerate}
\item
$A$ is globally regular.
\item
There exists
 $b\in\mathcal B$ such that
\begin{equation*}
\mathrm e^{\mathrm i x\Xi_\pm}\notin \mathcal S',
\end{equation*}
 or
\begin{equation*}
\mathrm e^{\mathrm i x\Xi_-}\notin \mathcal S',\quad
\mathrm e^{\mathrm i x\Xi_+}\in \mathcal S,\quad \Delta_2\ne 0, \quad\lambda\notin\{1+2n:n\in\mathbb Z_+\},
\end{equation*}
 or
\begin{equation*}
\mathrm e^{\mathrm i x\Xi_-}\notin \mathcal S',\quad
\mathrm e^{\mathrm i x\Xi_+}\in \mathcal S,\quad \Delta_2= 0.
\end{equation*}
\item
For all
 $b\in\mathcal B$ we have
\begin{equation*}
\mathrm e^{\mathrm i x\Xi_\pm}\notin \mathcal S',
\end{equation*}
 or
\begin{equation*}
\mathrm e^{\mathrm i x\Xi_-}\notin \mathcal S',\quad
\mathrm e^{\mathrm i x\Xi_+}\in \mathcal S,\quad \Delta_2\ne 0, \quad\lambda\notin\{1+2n:n\in\mathbb Z_+\},
\end{equation*}
 or
\begin{equation*}
\mathrm e^{\mathrm i x\Xi_-}\notin \mathcal S',\quad
\mathrm e^{\mathrm i x\Xi_+}\in \mathcal S,\quad \Delta_2= 0.
\end{equation*}
\end{enumerate}
\end{theorem}

\subsection{Proof of Theorem \ref{thm:21}}\label{subsec:1}
Let
\begin{equation}\label{eqn:800}
B=b_{20}D^2+b_{11}MD+b_{02}M^2+b_{10}D+b_{01}M+b_{00}I
\end{equation}
be a differential operator with Weyl symbol $b\in \mathcal B$.

As for the source of a twisted differential operator, also the kernel of $B$ is considered in the sense of temperate distributions:
\begin{equation*}
\ker B=\{u\in \mathcal S'(\mathbb R)\mid B u=0\}.
\end{equation*}

\begin{proposition}\label{pro:47}
The following conditions are equivalent.
\begin{enumerate}
\item\label{itm:16}
$\lim\limits_{\abs{x}\to \infty}\abs{ x\Im \Xi_\pm(x)}=\infty$.
\item\label{itm:45}
$\mathrm e^{\mathrm i x\Xi_\pm(x)}\in\mathcal S\cup(\mathcal C^\infty\setminus\mathcal S')$.
\item\label{itm:22}
$B$ is globally regular.
\end{enumerate}
\end{proposition}
\begin{proof}
It is obvious that \eqref{itm:16}$\iff$\eqref{itm:45}. Let us prove \eqref{itm:16}$\iff$\eqref{itm:22}.

Assume $\Delta_2=\Im \frac {b_{11}}{b_{20}}=0$. Then it is easy to verify that the following conditions are equivalent.
\begin{enumerate}
\item[(a)]
There exists $\epsilon>0$ such that
\begin{equation*}
\max\left\{\abs{\Xi_+(x)+\frac{b_{11}}{2b_{20}}x},\,\abs{\Xi_-(x)+\frac{b_{11}}{2b_{20}}x},\,\abs{x}^{\epsilon-1}\right\}
=\mathcal O\left(\abs{\Xi_+(x)-\Xi_-(x)}\right),\quad\text{\upshape for $\abs{x}\to\infty$}.
\end{equation*}
\item[(b)]
$\Delta_1 x+\Delta_0$ does not vanish identically.
\end{enumerate}
If $\Delta_2 x^2+\Delta_1 x+\Delta_0$ does not vanish identically, it follows that we can apply Theorem 1.2 of \cite{Nicola-Rodino:2}, obtaining that  \eqref{itm:16} is equivalent to  \eqref{itm:22}.

\smallskip
If $\Delta_2=\Delta_1=\Delta_0=0$, the equation $Bu=f$ can be solved explicitly:
\begin{equation}\label{eqn:496}
u(x)=-\mathrm e^{-\frac{\mathrm i}{4b_{20}}(b_{11} x^2+2b_{10}x)}\left\{\frac{1}{b_{20}}\int_0^x(x-t)\,\mathrm e^{\frac{\mathrm i}{4b_{20}}(b_{11} t^2+2b_{10}t)}f(t)\,\mathrm d t+
c_0x+c_1\right\},
\end{equation}
where $c_0$ and $c_1$ are arbitrary constants.

Since  $x\Xi_{\pm}=-\frac 1{2b_{20}}(b_{11}x^2+b_{10}x)$,
 we have to show that
\begin{equation*}
\left(\Im \frac {b_{11}}{b_{20}}\right)^2+\left(\Im \frac {b_{10}}{b_{20}}\right)^2>0 \iff \text{\upshape $B$ is globally regular}.
\end{equation*}

Assume $\left(\Im \frac {b_{11}}{b_{20}}\right)^2+\left(\Im \frac {b_{10}}{b_{20}}\right)^2>0$, and $f \in \mathcal S$. Then we have to prove that
$u$
belongs to $\mathcal S\cup (\mathcal C^\infty\setminus \mathcal S')$.

If $\Im \frac {b_{11}}{b_{20}}<0$, set
\begin{equation*}
v(x)=\mathrm e^{-h(x)}\int_0^x(x-t)\mathrm e^{h(t)}f(t)\,\mathrm d t,
\end{equation*}
with
\begin{equation*}
h(x)=\frac{\mathrm i}{4b_{20}}(b_{11} x^2+2 b_{10}x).
\end{equation*}
If we show that $v\in \mathcal S$, we have that $u\in\mathcal S$.

It is clear that for all $n\in\mathbb Z_+$ there exist polynomials $P_n(x)$ and $Q_n(x)$ of degree $n$ such that\,\footnote{\
Definition \eqref{eqn:463} is equivalent to define by induction
\begin{equation*}
P_{n}=\begin{cases}1,&\text{\upshape if $n=0$},\\ P_{n-1}'-P_{n-1} h',&\text{\upshape if $n\ge 1$},\end{cases}, \qquad
Q_{n}=\begin{cases}1,&\text{\upshape if $n=0$},\\ Q_{n-1}'+Q_{n-1} h',&\text{\upshape if $n\ge 1$}.\end{cases}\end{equation*}}
\begin{equation}\label{eqn:463}
\frac{\mathrm d^n}{\mathrm d x^n}\, \mathrm e^{-h(x)}=P_n(x)\mathrm e^{-h(x)},\quad
\frac{\mathrm d^n}{\mathrm d x^n}\, \mathrm e^{h(x)}=Q_n(x)\mathrm e^{h(x)},\qquad \text{\upshape for $n\ge 0$}.
\end{equation}
Then we have
\begin{equation*}
v'(x)=-h'(x)\mathrm e^{-h(x)}\int_0^x(x-t)\mathrm e^{h(t)}f(t)\,\mathrm d t+\mathrm e^{-h(x)}\int_0^x \mathrm e^{h(t)}f(t)\,\mathrm d t,
\end{equation*}
and
\begin{equation*}\begin{aligned}
v^{(n)}(x)&=P_n(x)\mathrm e^{-h(x)}\int_0^x(x-t)\mathrm e^{h(t)}f(t)\,\mathrm d t+
n P_{n-1}(x)\mathrm e^{-h(x)}\int_0^x \mathrm e^{h(t)}f(t)\,\mathrm d t+
\\&+\sum_{k=2}^n \binom {n}{k} P_{n-k}(x) \sum_{j=0}^{k-2} \binom {k-2}{j} Q_{k-2-j}(x) f^{(j)}(x),\qquad\text{\upshape for $n\ge 2$}.
\end{aligned} \end{equation*}
Since $f\in \mathcal S$, we have
\begin{equation*}
\lim_{\abs{x}\to\infty}x^m\sum_{k=2}^n \binom {n}{k} P_{n-k}(x) \sum_{j=0}^{k-2} \binom {k-2}{j} Q_{k-2-j}(x) f^{(j)}(x)=0,\qquad\forall m\in\mathbb Z_+.
\end{equation*}
On the other side, since $\Re\left(\frac{\mathrm i}{4b_{20}}b_{11} (t^2-x^2)\right)=\Im \frac {b_{11}}{4b_{20}}(x^2-t^2)<0$, for $x>t$, we have
\begin{equation*}\begin{aligned}
&\lim_{\abs{x}\to\infty} x^m\Bigl\{P_n(x)\mathrm e^{-h(x)}\int_0^x(x-t)\mathrm e^{h(t)}f(t)\,\mathrm d t+
n P_{n-1}(x)\mathrm e^{-h(x)}\int_0^x \mathrm e^{h(t)}f(t)\,\mathrm d t\Bigr\}=
\\&\qquad=\lim_{\abs{x}\to\infty} \int_0^x x^m\Bigl\{x P_n(x)+
n P_{n-1}(x)-P_n(x)t\Bigr\}\mathrm e^{\frac{\mathrm i}{4b_{20}}[b_{11} (t^2-x^2)+2 b_{10}(t-x)]}f(t)\,\mathrm d t=0,
\end{aligned}\end{equation*}
by Dominated Convergence Theorem.
Then we have shown that $\lim_{\abs{x}\to 0}x^m v^{(n)}(x)=0$ for all $m,n\in\mathbb Z_+$. It follows that $u\in \mathcal S$, that is that $B$ is globally regular.

\bigskip
If $\Im \frac {b_{11}}{b_{20}}>0$,
$\mathrm e^{\frac{\mathrm i}{4b_{20}}(b_{11} t^2+2b_{10}t)}f(t)$ belongs to $\mathcal S$. Then
\begin{equation*}
\int_0^{\pm \infty} \mathrm e^{\frac{\mathrm i}{4b_{20}}(b_{11} t^2+2b_{10}t)}f(t)\,\mathrm d t\quad\text{\upshape and}\quad
\int_0^{\pm \infty} t\,\mathrm e^{\frac{\mathrm i}{4b_{20}}((b_{11} t^2+2b_{10}t)}f(t)\,\mathrm d t
\end{equation*}
are convergent, so $u$ grows at infinity as $(1+\abs{x})\mathrm e^{\Im\frac {b_{11}}{4b_{20}} x^2}$ and  cannot belong to $\mathcal S'$.

\bigskip
If $\Im \frac {b_{11}}{b_{20}}=0$ and $\Im \frac {b_{10}}{b_{20}}> 0$,
\begin{equation*}
\int_0^{+\infty} \mathrm e^{\frac{\mathrm i}{4b_{20}}(b_{11} t^2+2b_{10}t)}f(t)\,\mathrm d t\quad\text{\upshape and}\quad
\int_0^{+\infty} t\,\mathrm e^{\frac{\mathrm i}{4b_{20}}(b_{11} t^2+2b_{10}t)}f(t)\,\mathrm d t
\end{equation*}
are convergent, so  $u$ grows    as $(1+x)\mathrm e^{\Im \frac{b_{10}}{2b_{20}} x}$ for $x\to +\infty$ and  cannot belong to $\mathcal S'$.

If $\Im \frac {b_{10}}{b_{20}}<0$, $u$ grows    as $(1-x)\mathrm e^{\Im \frac {b_{10}}{2b_{20}} x}$ for $x\to -\infty$ and  again cannot belong to $\mathcal S'$.

On the contrary, if $B$ is globally regular, from \eqref{eqn:496} with $f=0$, $c_0=0$, and $c_1=1$, we get that
\begin{equation*}
\mathrm e^{-\frac{\mathrm i}{b_{20}}(b_{11} x^2+2b_{10}x)}\in \mathcal S\cup (\mathcal C^\infty\setminus \mathcal S'),
\end{equation*}
which in turn implies $\left(\Im \frac {b_{11}}{b_{20}}\right)^2+\left(\Im \frac {b_{10}}{b_{20}}\right)^2>0$.
\end{proof}

\begin{proposition}\label{pro:53}
$B$ is globally regular  if and only if $B'$ is globally regular.
\end{proposition}
\begin{proof}
Consider  the \emph{formal adjoint} $B^\ast=\overline{B'}$.
Since $B'=f$ is equivalent to $B^\ast=\bar f$, $B'$ is globally regular if and only if $B^\ast$ is globally regular.

A simple computation shows that
the Weyl symbol of $B^\ast$ is the complex conjugate of the Weyl symbol of $B$.
Then, since $\abs{ x\Im \Xi_\pm(x)}=\abs{ x\Im \overline{\Xi_\pm(x)}}$,
the statement follows from Proposition \ref{pro:47}.
\end{proof}

\begin{proposition}\label{pro:49}
We have
\begin{gather}\label{eqn:414}
\ker \tilde A=0 \iff \ker B=0,
\\\label{eqn:697}
\text{\upshape $\tilde A$ is globally regular if and only if $B$ is globally regular},
\intertext{\upshape and}\label{eqn:442}
\text{\upshape $(\tilde A)'$ is globally regular if and only if $B'$ is globally regular}.
\end{gather}
\end{proposition}
\begin{proof}
Thanks to  \cite[Theorem 18.5.9]{Hormander}, there exists a unitary operator $U$ on $L^2(\mathbb R)$, which is an automorphism of
$\mathcal S(\mathbb R)$ and  $\mathcal S'(\mathbb R)$, such that
$B=U^{-1}\tilde A U$. Since the dual is globally regular if and only if the formal adjoint is globally regular, this implies the result.
\end{proof}

\noindent\textbf{Proof of Theorem \ref{thm:21}}
\eqref{itm:17}$\implies$\eqref{itm:18}: follows from Proposition \ref{pro:22}.

\noindent\eqref{itm:18}$\implies$\eqref{itm:19}: follows from Propositions \ref{pro:53}, and \ref{pro:49}.

\noindent\eqref{itm:19}$\implies$\eqref{itm:20}: obvious.

\noindent\eqref{itm:20}$\implies$\eqref{itm:17}: follows from Theorem \ref{thm:19}.
\qed

\subsection{Proof of Theorem \ref{thm:16}}

\subsubsection{Asymptotic behavior of the general solution to equation $Bu=0$}

Consider the operator $B$ given by \eqref{eqn:800} with $b_{20}\ne0$.

Define
\begin{equation}\label{eqn:465}
\Sigma_\pm(x)=\begin{cases}\displaystyle
-\frac 14\left\{\frac {b_{11}}{b_{20}}x^2+2\frac {b_{10}}{b_{20}}x\pm
\sigma\!\left(\frac {\Delta_2}{b_{20}^2}\right)\left(\frac {\Delta_2}{b_{20}^2}\right)^{\!\!\frac 12}
x^2\left(1+\frac {\Delta_1}{2\Delta_2x}\right)^{\!\!2}\right\},
&\text{\upshape if $\Delta_2\ne 0$},
\\[10pt] \displaystyle
-\frac 14\left\{\frac { b_{11}}{b_{20}}\,x^2+2\frac {b_{10}}{b_{20}}\,x
\pm\frac 43 \sigma\!\left(\frac{\Delta_1}{b_{20}^2}\right)\left(\frac{\Delta_1}{b_{20}^2}\right)^{\!\!\frac 12}
x^{\frac 32}\left(1+\frac{\Delta_0}{\Delta_1x}\right)^{\!\!\frac 32}\right\},
&\text{\upshape if $\Delta_2= 0\ne \Delta_1$},
\\[10pt] \displaystyle
-\frac 14\left\{\frac { b_{11}}{b_{20}}\,x^2+2\frac {b_{10}}{b_{20}}\,x
\pm 2\sigma\!\left(\frac {\Delta_0}{b_{20}^2}\right)\left(\frac {\Delta_0}{b_{20}^2}\right)^{\!\!\frac 12}x\right\},
&\text{\upshape if $\Delta_2=\Delta_1= 0$},\end{cases}
\end{equation}
where $\Delta_ 0$, $\Delta_ 1$, and $\Delta_ 2$ are given by \eqref{eqn:801}.

\paragraph{Assume $\Delta_2\ne 0$.}
The \emph{confluent hypergeometric function of the first kind,
of parameters $p\in\mathbb C$ and $q\in\mathbb C\setminus\mathbb Z_-$},
is the solution to the differential equation in the complex domain
\begin{equation*}
z u''+(q-z)u'-p u=0,
\end{equation*}
given by the entire analytic function (see \cite[(9.9.1)]{Lebedev})
\begin{equation}\label{eqn:362}
\Phi(p,q;z)=
\sum_{k=0}^{\infty}
\frac{(p)_k}{k!(q)_k}\,z^k,
\end{equation}
where
\begin{equation}\label{eqn:495}
(p)_k=\frac{\Gamma(p+k)}{\Gamma(p)}=\begin{cases}1,& \text{\upshape if $k=0$},\\ p(p+1)\cdots (p+k-1),&\text{\upshape if $k\ge 1$}.
\end{cases}
\end{equation}
and $\Gamma$ is the Euler Gamma Function.

\begin{proposition}\label{pro:56}
Consider  a complex number  $\lambda$. The  \emph{Hermite-Weber equation} (in the complex domain)
\begin{equation}\label{eqn:502}
w''(z)-(z^2-\lambda)w(z)=0
\end{equation}
has two linearly independent solutions given  by
\begin{equation}\label{eqn:520}
w_1(z)=\mathrm e^{-\frac 12\,z^2}\Phi\!\left(\frac{1-\lambda}4,\frac 12;z^2\right), \qquad
w_2(z)=\mathrm e^{-\frac 12\,z^2}z\Phi\!\left(\frac{3-\lambda}4,\frac 32;z^2\right).
\end{equation}
\end{proposition}
\begin{proof}
A straightforward computation shows that $w_1$ and $w_2$ given by \eqref{eqn:520} solve \eqref{eqn:502}.

Now we show that $w_1$ and $w_2$ are linearly independent. Since the Wronskian $\mathcal W$ of $w_1$ and $w_2$ is constant, it suffices to compute it at the origin, where
we have
\begin{equation*}
\mathcal W(0)=\begin{vmatrix}w_1(0) & w_2(0) \\ w_1'(0) & w_2'(0)\end{vmatrix}=
\Phi\!\left(\frac{1-\lambda}4,\frac 12;0\right)\Phi\!\left(\frac{3-\lambda}4,\frac 32;0\right)=1.   \qedhere
\end{equation*}
\end{proof}

\begin{proposition}\label{pro:55}
The equation $Bu=0$ has two  linearly independent analytic solutions $u_1$ and $u_2$ given by
\begin{equation}\label{eqn:188}
u_j(x)=\mathrm e^{-\frac{\mathrm i}{4b_{20}}\left(b_{11}x^2+2b_{10}x\right)}v_j(x),
\end{equation}
where $j\in\{1,2\}$,
\begin{equation}\label{eqn:187}
v_j(x)=w_j\!\left(\left(-\frac{\Delta_2}{4b_{20}^2}\right)^{\!\!\frac 14}\left(x+\frac{\Delta_1}{2\Delta_2}\right)\right),
\end{equation}
and $w_1$, and $w_2$ are given by \eqref{eqn:520}, with
\begin{equation*}
\lambda=\frac 18\left(-\frac {\Delta_2}{b_{20}^2}\right)^{\!\!-\frac 32}\frac {\Delta_1^2-4\Delta_2\Delta_0}{b_{20}^4}.
\end{equation*}
\end{proposition}
\begin{proof}
Set
\begin{equation*}
v(x)=\mathrm e^{\frac{\mathrm i}{4b_{20}}\,\left(b_{11} x^2+2b_{10} x\right)}u(x).
\end{equation*}
A simple computation shows that $Bu=0$ if and only if
\begin{equation}\label{eqn:494}
v''(x)+\frac {1}{4b_{20}^2}\left(\Delta_2 x^2+\Delta_1 x+\Delta_0\right)v(x)=0.
\end{equation}
 Define
\begin{equation*}
w(z)=v\!\left(\left(-\frac{\Delta_2}{4b_{20}^2}\right)^{\!\!-\frac 14} z-\frac {\Delta_1}{2\Delta_2}\right).
\end{equation*}
Then $v$ satisfies equation \eqref{eqn:494} if and only if $w$ is a solution to equation \eqref{eqn:502}.

It follows that Proposition \ref{pro:55} is a consequence of Proposition \ref{pro:56}.
\end{proof}
\begin{proposition}\label{pro:50}
Let $u_1$ and $u_2$ be as in Proposition \ref{pro:55} and assume $\Arg \frac{\Delta_2}{b_{20}^2}\ne 0$.
For all $c_1,c_2\in\mathbb C$ we have
 the following asymptotic expansions, with $\Sigma_\pm$ defined by \eqref{eqn:465}.
\begin{Mylist}
\item
If
$\frac {c_1}{\Gamma\left(\frac {1-\lambda}4\right)}\pm \frac {c_2}{2\Gamma\left(\frac {3-\lambda}4\right)}\ne 0$, we have
\begin{align*}
&c_1u_1(x)+ c_2u_2(x)=\sqrt \pi \left(-\frac {\Delta_2}{4b_{20}^2}\right)^{-\frac {1+\lambda}8}\left(\frac {c_1}{\Gamma\!\left(\frac {1-\lambda}4\right)}+ \frac {c_2}{2\Gamma\!\left(\frac {3-\lambda}4\right)}\right)\mathrm e^{\mathrm i\Sigma_-(x)} \abs{x}^{-\frac {1+\lambda}2}
\left\{1+\mathcal O\left(\abs{x}^{-1}\right)\right\},
\\&\hspace{8.5cm}\text{\upshape for $x\to+\infty$},
\\
&c_1u_1(x)+ c_2u_2(x)=
\sqrt \pi \left(-\frac {\Delta_2}{4b_{20}^2}\right)^{-\frac {1+\lambda}8}\left(\frac {c_1}{\Gamma\!\left(\frac {1-\lambda}4\right)}- \frac {c_2}{2\Gamma\!\left(\frac {3-\lambda}4\right)}\right)
\mathrm e^{\mathrm i\Sigma_-(x)}\abs{x}^{-\frac {1+\lambda}2}
\left\{1+\mathcal O\left(\abs{x}^{-1}\right)\right\},
\\&\hspace{8.5cm}\text{\upshape for $x\to-\infty$}.
\end{align*}
\item
If $c_1=\frac {c}{\Gamma\!\left(\frac {3-\lambda}4\right)}$,
$c_2=-\frac {2c}{\Gamma\!\left(\frac {1-\lambda}4\right)}$,
with $c\ne 0$, and $\lambda\notin\{1+2n:n\in\mathbb Z_+\}$,
  we have
\begin{align*}
&c_1u_1(x)+ c_2u_2(x)=
\frac {c}{\sqrt \pi}\left(-\frac {\Delta_2}{4b_{20}^2}\right)^{-\frac {1-\lambda}8}\mathrm e^{\mathrm i\Sigma_+(x)}  \abs{x}^{-\frac {1-\lambda}2}
\left\{1+\mathcal O\left(\abs{x}^{-1}\right)\right\},\qquad
\text{\upshape for $x\to+\infty$},
\\
&c_1u_1(x)+ c_2u_2(x)=
\frac {2\sqrt \pi c}{\Gamma\!\left(\frac {1-\lambda}4\right)\Gamma\!\left(\frac {3-\lambda}4\right)}
\left(-\frac {\Delta_2}{4b_{20}^2}\right)^{-\frac {1+\lambda}8}
\mathrm e^{\mathrm i\Sigma_-(x)}\abs{x}^{-\frac {1+\lambda}2}
\left\{1+\mathcal O\left(\abs{x}^{-1}\right)\right\},
\\&\hspace{8.5cm}\text{\upshape for $x\to-\infty$},
\end{align*}
\item
If $c_1=\frac {c}{\Gamma\!\left(\frac {3-\lambda}4\right)}$,
$c_2=\frac {2c}{\Gamma\!\left(\frac {1-\lambda}4\right)}$,
with $c\ne 0$, and  $\lambda\notin\{1+2n:n\in\mathbb Z_+\}$,  we have
\begin{align*}
&c_1u_1(x)+ c_2u_2(x)=
\frac {2\sqrt \pi c}{\Gamma\!\left(\frac {1-\lambda}4\right)\Gamma\!\left(\frac {3-\lambda}4\right)}
\left(-\frac {\Delta_2}{4b_{20}^2}\right)^{-\frac {1+\lambda}8}
\mathrm e^{\mathrm i\Sigma_-(x)} \abs{x}^{-\frac {1+\lambda}2}
\left\{1+\mathcal O\left(\abs{x}^{-1}\right)\right\},
\\&\hspace{8.5cm}\text{\upshape for $x\to+\infty$},
\\
&c_1u_1(x)+ c_2u_2(x)=
\frac {c}{\sqrt \pi}
\left(-\frac {\Delta_2}{4b_{20}^2}\right)^{-\frac {1-\lambda}8}
\mathrm e^{\mathrm i\Sigma_+(x)} \abs{x}^{-\frac {1-\lambda}2}
\left\{1+\mathcal O\left(\abs{x}^{-1}\right)\right\},
\\&\hspace{8.5cm}\text{\upshape for $x\to-\infty$}.
\end{align*}
\item
If  $c_1=\frac {c}{\Gamma\!\left(\frac {3-\lambda}4\right)}$,
$c_2=\mp\frac {2c}{\Gamma\!\left(\frac {1-\lambda}4\right)}$,
with $c\ne 0$, and
$\lambda=1+4n$, with $n\in\mathbb Z_+$, we have
\begin{equation*}
c_1u_1(x)+ c_2u_2(x)=\frac {c}{\sqrt\pi}
\left(-\frac {\Delta_2}{4b_{20}^2}\right)^{\frac n2}
\mathrm e^{\mathrm i\Sigma_+(x)}x^{2n}
\left\{1+\mathcal O\left(\abs{x}^{-1}\right)\right\},\qquad\text{\upshape for  $\abs{x}\to \infty$}.
\end{equation*}
\item
If $c_1=\frac {c}{\Gamma\!\left(\frac {3-\lambda}4\right)}$,
$c_2=\mp\frac {2c}{\Gamma\!\left(\frac {1-\lambda}4\right)}$,
with $c\ne 0$, and $\lambda=3+4n$, with $n\in\mathbb Z_+$,
we have
\begin{equation*}
c_1u_1(x)+ c_2u_2(x)=\pm\frac {c}{\sqrt\pi}
\left(-\frac {\Delta_2}{4b_{20}^2}\right)^{\frac 14+\frac n2}
\mathrm e^{\mathrm i\Sigma_+(x)}x^{2n+1}\left\{1+\mathcal O\left(\abs{x}^{-1}\right)\right\},\qquad\text{\upshape for  $\abs{x}\to \infty$}.
\end{equation*}
\end{Mylist}
\end{proposition}
\begin{proof}
Set
\begin{equation*}
z=\left(-\frac {\Delta_2}{4b_{20}^2}\right)^{\!\!\frac 14}\left(x+\frac {\Delta_1}{2\Delta_2}\right).
\end{equation*}
From \eqref{eqn:520}, \eqref{eqn:188}, and \eqref{eqn:187}, it follows that
\begin{equation}\label{eqn:96}
c_1u_1(x)+c_2 u_2(x)=\mathrm e^{-\frac {\mathrm i}4\left(\frac {b_{11}}{b_{20}}x^2+2\frac {b_{10}}{b_{20}}x\right)}
\Bigl(c_1 w_1(z)+c_2w_2(z)\Bigr).
\end{equation}
On the other side, since
\begin{multline}\label{eqn:817}
\left(-\frac {\Delta_2}{b_{20}^2}\right)^{\!\!\frac 12}=
\abs{\frac {\Delta_2}{b_{20}^2}}^{\frac 12}\mathrm e^{\frac {\mathrm i}2\Arg\left(-\frac {\Delta_2}{b_{20}^2}\right)}=
\abs{\frac {\Delta_2}{b_{20}^2}}^{\frac 12}
\mathrm e^{\frac {\mathrm i}2\left(\Arg \frac {\Delta_2}{b_{20}^2}+\sigma\!\left(\frac {\Delta_2}{b_{20}^2}\right)\pi\right)}
=\\=
\mathrm e^{\frac {\mathrm i}2 \sigma\!\left(\frac {\Delta_2}{b_{20}^2}\right)\pi}\left(\frac {\Delta_2}{b_{20}^2}\right)^{\!\!\frac 12}
=\mathrm i\sigma\!\left(\frac {\Delta_2}{b_{20}^2}\right)\left(\frac {\Delta_2}{b_{20}^2}\right)^{\!\!\frac 12},
\end{multline}
we have
\begin{equation}\label{eqn:97}
-\frac {\mathrm i}4\left(\frac {b_{11}}{b_{20}}x^2+2\frac {b_{10}}{b_{20}}x\right)\pm \frac 12\, z^2
=-\frac {\mathrm i}4\left(\frac {b_{11}}{b_{20}}x^2+2\frac {b_{10}}{b_{20}}x\right)
\pm \frac 12\left(-\frac {\Delta_2}{4b_{20}^2}\right)^{\!\!\frac 12}
\left(x+\frac {\Delta_1}{2\Delta_2}\right)^{\!\!2}
=\mathrm i \Sigma_{\mp}(x).
\end{equation}
Moreover, since \begin{equation*}
\abs{\Arg\left(\left(-\frac {\Delta_2}{b_{20}^2}\right)^{\!\!\frac 14}\right)}<\frac \pi 4,
\end{equation*}
 and
\begin{equation*}
\lim_{\abs{x}\to\infty}\Arg\left(1+\frac {\Delta_1}{2\Delta_2 x}\right)=0,
\end{equation*}
there exists $0<\epsilon<\frac \pi 4$, such that
\begin{equation}\label{eqn:99}
\abs{\Arg (\pm z)}\le \frac \pi 4-\epsilon,\qquad\text{\upshape for $x\to\pm \infty$}.
\end{equation}

In particular
\begin{equation}\label{eqn:98}
\pm z=\left(-\frac {\Delta_2}{4b_{20}^2}\right)^{\!\!\frac 14}\abs{x}\left(1+\mathcal O\left(\abs{x}^{-1}\right)\right),
\qquad \text{\upshape for $x\to\pm\infty$}.
\end{equation}

 In conclusion the statement follows from \eqref{eqn:96}, \eqref{eqn:97}, \eqref{eqn:99}, \eqref{eqn:98}, and Proposition \ref{pro:59}.
\end{proof}

\begin{proposition}\label{pro:61}
Let $u_1$ and $u_2$ be as in Proposition \ref{pro:55} and assume $\Arg \frac {\Delta_2}{b_{20}^2}= 0$.
For all $c_1,c_2\in\mathbb C$ we have
 the following asymptotic expansions.
\begin{Mylist}
\item\label{itm:46}
If $\left(\frac {\mathrm i c_1}{\Gamma\left(\frac {1+\lambda}4\right)}\mp \frac {c_2}{2\Gamma\left(\frac {3+\lambda}4\right)}\right)
\left(\frac {c_1}{\Gamma\left(\frac {1-\lambda}4\right)}\pm \frac {c_2}{2\Gamma\left(\frac {3-\lambda}4\right)}\right)\ne 0$, we have
\begin{align*}
&c_1u_1(x)+ c_2u_2(x)=
\sqrt \pi \mathrm e ^{-\mathrm i \frac {1+\lambda}{4} \pi}
\left(-\frac {\Delta_2}{4b_{20}^2}\right)^{-\frac {1-\lambda}8}
\left(\frac {\mathrm i c_1}{\Gamma\!\left(\frac {1+\lambda}4\right)}- \frac {c_2}{2\Gamma\!\left(\frac {3+\lambda}4\right)}\right)
\mathrm e^{\mathrm i\Sigma_+(x)} \abs{x}^{-\frac {1-\lambda}2}\left\{1+\mathcal O\left(\abs{x}^{-1}\right)\right\}+
\\
&\quad+
\sqrt \pi
\left(-\frac {\Delta_2}{4b_{20}^2}\right)^{-\frac {1+\lambda}8}
\left(\frac {c_1}{\Gamma\!\left(\frac {1-\lambda}4\right)}+ \frac {c_2}{2\Gamma\!\left(\frac {3-\lambda}4\right)}\right)
\mathrm e^{\mathrm i\Sigma_-(x)}\abs{x}^{-\frac {1+\lambda}2}\left\{1+\mathcal O\left(\abs{x}^{-1}\right)\right\},
\\ &\hspace{8cm}\text{\upshape for $x\to+\infty$},
\\ &c_1u_1(x)+ c_2u_2(x)=
\sqrt \pi \mathrm e ^{-\mathrm i \frac {1+\lambda}{4} \pi}
\left(-\frac {\Delta_2}{4b_{20}^2}\right)^{-\frac {1-\lambda}8}
\left(\frac {\mathrm i c_1}{\Gamma\!\left(\frac {1+\lambda}4\right)}+ \frac {c_2}{2\Gamma\!\left(\frac {3+\lambda}4\right)}\right)
\mathrm e^{\mathrm i\Sigma_+(x)}\abs{x}^{-\frac {1-\lambda}2}
\left\{1+\mathcal O\left(\abs{x}^{-1}\right)\right\}
\\&\quad
+\sqrt \pi \left(-\frac {\Delta_2}{4b_{20}^2}\right)^{-\frac {1+\lambda}8}
\left(\frac {c_1}{\Gamma\!\left(\frac {1-\lambda}4\right)}- \frac {c_2}{2\Gamma\!\left(\frac {3-\lambda}4\right)}\right)
\mathrm e^{\mathrm i\Sigma_-(x)}\abs{x}^{-\frac {1+\lambda}2}\left\{1+
\mathcal O\left(\abs{x}^{-1}\right)\right\},
\\ &\hspace{8cm}\text{\upshape for $x\to-\infty$},
\end{align*}
\item
If $c_1=\frac {c}{\Gamma\!\left(\frac {3+\lambda}4\right)}$,
$c_2=\frac {2\mathrm i c}{\Gamma\!\left(\frac {1+\lambda}4\right)}$,
with $c\ne 0$, and $\lambda\notin\{-(1+2n):n\in\mathbb Z_+\}$ we have
\begin{align*}
&c_1u_1(x)+ c_2u_2(x)=
\frac {c}{\sqrt \pi}\,
\mathrm e^{\mathrm i\Sigma_-(x)}
\mathrm e^{\mathrm i\frac {1+\lambda}{4}\pi}
\left(-\frac {\Delta_2}{4b_{20}^2}\right)^{-\frac {1+\lambda}8}
\abs{x}^{-\frac {1+\lambda}2}
\left\{1+\mathcal O\left(\abs{x}^{-1}\right)\right\},
\\ &\hspace{8cm}\text{\upshape for $x\to+\infty$},
\\ &c_1u_1(x)+ c_2u_2(x)=
\frac {\mathrm 2 \mathrm i\sqrt \pi \mathrm e ^{-\mathrm i \frac {1+\lambda}{4} \pi}
 c}{\Gamma\!\left(\frac {1+\lambda}4\right)\Gamma\!\left(\frac {3+\lambda}4\right)}
\left(-\frac {\Delta_2}{4b_{20}^2}\right)^{-\frac {1-\lambda}8}
\mathrm e^{\mathrm i\Sigma_+(x)}\abs{x}^{-\frac {1-\lambda}2}
\left\{1+\mathcal O\left(\abs{x}^{-1}\right)\right\}+
\\&\quad
+\frac {c\,\mathrm e^{-\mathrm i \frac {1+\lambda}{4}\pi}}{\sqrt \pi}
\left(-\frac {\Delta_2}{4b_{20}^2}\right)^{-\frac {1+\lambda}8}
\mathrm e^{\mathrm i\Sigma_-(x)}\abs{x}^{-\frac {1+\lambda}2}
\left\{1+\mathcal O\left(\abs{x}^{-1}\right)\right\},
\qquad\text{\upshape for $x\to-\infty$}.
\end{align*}
\item
If $c_1=\frac {c}{\Gamma\!\left(\frac {3+\lambda}4\right)}$,
$c_2=-\frac {2\mathrm i c}{\Gamma\!\left(\frac {1+\lambda}4\right)}$,
with $c\ne 0$, and $\lambda\notin\{-(1+2n):n\in\mathbb Z_+\}$, we have
\begin{align*}
&c_1u_1(x)+ c_2u_2(x)=
\frac {2\mathrm i \sqrt \pi \mathrm e ^{-\mathrm i \frac {1+\lambda}{4} \pi} c}
{\Gamma\!\left(\frac {1+\lambda}4\right)\Gamma\!\left(\frac {3+\lambda}4\right)}\,
\left(-\frac {\Delta_2}{4b_{20}^2}\right)^{-\frac {1-\lambda}8}
\mathrm e^{\mathrm i\Sigma_+(x)} \abs{x}^{-\frac {1-\lambda}2}
\left\{1+\mathcal O\left(\abs{x}^{-1}\right)\right\}+
\\&\quad
+\frac {c\,\mathrm e^{-\mathrm i\frac {1+\lambda}{4}\pi}}{\sqrt \pi}
\left(-\frac {\Delta_2}{4b_{20}^2}\right)^{-\frac {1+\lambda}8}
\mathrm e^{\mathrm i\Sigma_-(x)}\abs{x}^{-\frac {1+\lambda}2}
\left\{1+\mathcal O\left(\abs{x}^{-1}\right)\right\},
\qquad\text{\upshape for $x\to+\infty$},
\\ &c_1u_1(x)+ c_2u_2(x)=
\frac {c}{\sqrt \pi}\mathrm e^{\mathrm i\frac {1+\lambda}{4}\pi}
\left(-\frac {\Delta_2}{4b_{20}^2}\right)^{-\frac {1+\lambda}8}
\mathrm e^{\mathrm i\Sigma_-(x)}\abs{x}^{-\frac {1+\lambda}2}
\left\{1+\mathcal O\left(\abs{x}^{-1}\right)\right\},
\qquad\text{\upshape for $x\to-\infty$}.
\end{align*}
\item
If $c_1=\frac {c}{\Gamma\!\left(\frac {3+\lambda}4\right)}$,
$c_2=\pm\frac {2\mathrm i c}{\Gamma\!\left(\frac {1+\lambda}4\right)}$,
with $c\ne 0$, and   $\lambda=-(1+4n)$, with $n\in\mathbb Z_+$,
  we have
\begin{equation*}
c_1u_1(x)+c_2u_2(x)=
\frac {(-1)^n c}{\sqrt \pi}
\left(-\frac {\Delta_2}{4b_{20}^2}\right)^{\frac n2}
\mathrm e^{\mathrm i\Sigma_-(x)}x^{2n}\left\{1+\mathcal O\left(\abs{x}^{-1}\right)\right\},
\qquad\text{\upshape for  $\abs{x}\to\infty$}.
\end{equation*}
\item
If $c_1=\frac {c}{\Gamma\left(\frac {3+\lambda}4\right)}$,
$c_2=\pm\frac {2\mathrm i c}{\Gamma\!\left(\frac {1+\lambda}4\right)}$,
with $c\ne 0$, and    $\lambda=-(3+4n)$, with $n\in\mathbb Z_+$,
  we have
\begin{equation*}
c_1u_1(x)+c_2u_2(x)=
\mp\mathrm i \,\frac {(-1)^n c}{\sqrt \pi}
\left(-\frac {\Delta_2}{4b_{20}^2}\right)^{\frac 14+\frac n2}\mathrm e^{\mathrm i\Sigma_-(x)}
 x^{2n+1}\left\{1+\mathcal O\left(\abs{x}^{-1}\right)\right\},
\qquad\text{\upshape for  $\abs{x}\to\infty$}.
\end{equation*}
\item
If $c_1=\frac {c}{\Gamma\!\left(\frac {3-\lambda}4\right)}$,
$c_2=-\frac {2c}{\Gamma\!\left(\frac {1-\lambda}4\right)}$,
with $c\ne 0$, and  $\lambda\notin\{1+2n:n\in\mathbb Z_+\}$,  we have
\begin{align*}
&c_1u_1(x)+ c_2u_2(x)=
\frac {c}{\sqrt \pi}
\left(-\frac {\Delta_2}{4b_{20}^2}\right)^{-\frac {1-\lambda}8}
\mathrm e^{\mathrm i\Sigma_+(x)} \abs{x}^{-\frac {1-\lambda}{2}}
\left\{1+\mathcal O\left(\abs{x}^{-1}\right)\right\},
\qquad\text{\upshape for $x\to+\infty$},
\\ &c_1u_1(x)+ c_2u_2(x)=
-\frac {c \mathrm e^{-\mathrm i\frac {1+\lambda}{2}\pi}}{\sqrt \pi}
\left(-\frac {\Delta_2}{4b_{20}^2}\right)^{-\frac {1-\lambda}8}
\mathrm e^{\mathrm i\Sigma_+(x)}\abs{x}^{-\frac {1-\lambda}2}
\left\{1+\mathcal O\left(\abs{x}^{-1}\right)\right\}+
\\&\quad
+\frac {2\sqrt \pi c}{\Gamma\!\left(\frac {1-\lambda}4\right)\Gamma\!\left(\frac {3-\lambda}4\right)}\,
\left(-\frac {\Delta_2}{4b_{20}^2}\right)^{-\frac {1+\lambda}8}
\mathrm e^{\mathrm i\Sigma_-(x)}\abs{x}^{-\frac {1+\lambda}2}
\left\{1+\mathcal O\left(\abs{x}^{-1}\right)\right\},
\qquad\text{\upshape for $x\to-\infty$},
\end{align*}
\item
If $c_1=\frac {c}{\Gamma\!\left(\frac {3-\lambda}4\right)}$,
$c_2=\frac {2c}{\Gamma\!\left(\frac {1-\lambda}4\right)}$,
with $c\ne 0$, and  $\lambda\notin\{1+2n:n\in\mathbb Z_+\}$,
  we have
\begin{align*}
&c_1u_1(x)+ c_2u_2(x)=
-\frac {c\,\mathrm e^{-\mathrm i\frac {1+\lambda}{2}\pi}}{\sqrt \pi}
\left(-\frac {\Delta_2}{4b_{20}^2}\right)^{-\frac {1-\lambda}8}
\mathrm e^{\mathrm i\Sigma_+(x)} \abs{x}^{-\frac {1-\lambda}2}
\left\{1+\mathcal O\left(\abs{x}^{-1}\right)\right\}+
\\&\quad+
\frac {2\sqrt \pi  c}{\Gamma\!\left(\frac {1-\lambda}4\right)\Gamma\!\left(\frac {3-\lambda}4\right)}\,
\left(-\frac {\Delta_2}{4b_{20}^2}\right)^{-\frac {1+\lambda}8}
\mathrm e^{\mathrm i\Sigma_-(x)}\abs{x}^{-\frac {1+\lambda}2}
\left\{1+\mathcal O\left(\abs{x}^{-1}\right)\right\},
\qquad\text{\upshape for $x\to+\infty$},
\\ &c_1u_1(x)+ c_2u_2(x)=
\frac {c}{\sqrt \pi}
\left(-\frac {\Delta_2}{4b_{20}^2}\right)^{-\frac {1-\lambda}8}
\mathrm e^{\mathrm i\Sigma_+(x)}\abs{x}^{-\frac {1-\lambda}2}
\left\{1+\mathcal O\left(\abs{x}^{-1}\right)\right\},
\qquad\text{\upshape for $x\to-\infty$},
\end{align*}
\item
If $c_1=\frac {c}{\Gamma\!\left(\frac {3-\lambda}4\right)}$,
$c_2=\mp\frac {2c}{\Gamma\!\left(\frac {1-\lambda}4\right)}$,
with $c\ne 0$, and $\lambda=1+4n$, with $n\in\mathbb Z_+$,
we have
\begin{equation*}
c_1u_1(x)+ c_2u_2(x)=\frac {c}{\sqrt\pi}
\left(-\frac {\Delta_2}{4b_{20}^2}\right)^{\frac n2}\mathrm e^{\mathrm i \Sigma_+(x)}x^{2n}
\left\{1+\mathcal O\left(\abs{x}^{-1}\right)\right\},
\qquad\text{\upshape for  $\abs{x}\to\infty$}.
\end{equation*}
\item
If $c_1=\frac {c}{\Gamma\!\left(\frac {3-\lambda}4\right)}$,
$c_2=\mp\frac {2c}{\Gamma\!\left(\frac {1-\lambda}4\right)}$,
with $c\ne 0$, and  $\lambda=3+4n$, with $n\in\mathbb Z_+$,
we have
\begin{equation*}
c_1u_1(x)+ c_2u_2(x)=\pm\frac {c}{\sqrt\pi}\left(-\frac {\Delta_2}{4b_{20}^2}\right)^{\frac 14+\frac n2}
\mathrm e^{\mathrm i \Sigma_+(x)}x^{2n+1}\left\{1+\mathcal O\left(\abs{x}^{-1}\right)\right\},
\qquad\text{\upshape for  $\abs{x}\to\infty$}.
\end{equation*}
\end{Mylist}
\end{proposition}
\begin{proof}
Set
\begin{equation*}
z=\left(-\frac {\Delta_2}{4b_{20}^2}\right)^{\!\!\frac 14}\left(x+\frac {\Delta_1}{2\Delta_2}\right).
\end{equation*}
From \eqref{eqn:520}, \eqref{eqn:188}, and \eqref{eqn:187}, it follows that
\begin{equation}\label{eqn:806}
c_1u_1(x)+c_2 u_2(x)=\mathrm e^{-\frac {\mathrm i}4\left(\frac {b_{11}}{b_{20}}x^2+2\frac {b_{10}}{b_{20}}x\right)}
\Bigl(c_1 w_1(z)+c_2w_2(z)\Bigr).
\end{equation}
On the other side we have (see \eqref{eqn:97})
\begin{equation}\label{eqn:807}
-\frac {\mathrm i}4\left(\frac {b_{11}}{b_{20}}x^2+2\frac {b_{10}}{b_{20}}x\right)\pm \frac 12\, z^2
=\mathrm i \Sigma_{\mp}(x).
\end{equation}
Moreover, since \begin{equation*}
\Arg\left(\left(-\frac {\Delta_2}{4b_{20}^2}\right)^{\!\!\frac 14}\right)=\frac \pi 4,
\end{equation*}
 and
\begin{equation*}
\lim_{\abs{x}\to\infty}\Arg\left(1+\frac {\Delta_1}{2\Delta_2 x}\right)=0,
\end{equation*}
given $0<\epsilon<\frac \pi 4$, we have
\begin{equation}\label{eqn:808}
\abs{\Arg (\pm z)-\frac \pi 4}\le \epsilon,\qquad\text{\upshape for $x\to\pm \infty$}.
\end{equation}

In particular
\begin{equation}\label{eqn:809}
\pm z=\left(-\frac {\Delta_2}{4b_{20}^2}\right)^{\!\!\frac 14}\abs{x}\left(1+\mathcal O\left(\abs{x}^{-1}\right)\right),
\qquad \text{\upshape for $x\to\pm\infty$},
\end{equation}
and
\begin{equation}\label{eqn:810}
\mp \mathrm i z=\mathrm e^{-\mathrm i\frac \pi 2}\left(-\frac {\Delta_2}{4b_{20}^2}\right)^{\!\!\frac 14}
\abs{x}\left(1+\mathcal O\left(\abs{x}^{-1}\right)\right),
\qquad \text{\upshape for $x\to\pm\infty$}.
\end{equation}

 In conclusion the statement follows from \eqref{eqn:806}, \eqref{eqn:807}, \eqref{eqn:808}, \eqref{eqn:809}, \eqref{eqn:810}, and Proposition \ref{pro:60}.
\end{proof}

\paragraph{Assume $\Delta_2=0$ and $\Delta_1\ne 0$.}
The \emph{Airy functions}  are  two linearly independent solutions to the differential equation in the complex domain
\begin{equation*}
u''(z)-z u(z)=0,
\end{equation*}
given by the
entire analytic functions  (see \cite[(5.17.3)]{Lebedev})
\begin{equation*}
\Ai(z)=\sum_{k=0}^{\infty} \frac{z^{3k}}{3^{2k+\frac 23}k!\Gamma\!\left(k+\frac 23\right)}
-\sum_{k=0}^{\infty} \frac{z^{3k+1}}{3^{2k+\frac 43}k!\Gamma\!\left(k+\frac 43\right)},
\end{equation*}
and
\begin{equation*}
\Bi(z)=3^{\frac 12}\:\sum_{k=0}^{\infty} \frac{z^{3k}}{3^{2k+\frac 23}k!\Gamma\!\left(k+\frac 23\right)}
+3^{\frac 12}\:\sum_{k=0}^{\infty} \frac{z^{3k+1}}{3^{2k+\frac 43}k!\Gamma\!\left(k+\frac 43\right)}.
\end{equation*}

\begin{proposition}\label{pro:65}
The equation $Bu=0$ has two  linearly independent analytic solutions $u_1$ and $u_2$ given by
\begin{equation}\label{eqn:815}
u_j(x)=\mathrm e^{-\frac{\mathrm i}{4b_{20}}\left(b_{11}x^2+2b_{10}x\right)}v_j(x),
\end{equation}
where $j\in\{1,2\}$, and
\begin{equation}\label{eqn:816}
v_1(x)=
\Ai\left(\left(-\frac{\Delta_1}{4b_{20}^2}\right)^{\!\!\frac 13}\left(x+\frac{\Delta_0}{\Delta_1}\right)\right),\qquad
v_2(x)=\Bi\left(\left(-\frac{\Delta_1}{4b_{20}^2}\right)^{\!\!\frac 13}\left(x+\frac{\Delta_0}{\Delta_1}\right)\right).
\end{equation}
\end{proposition}
\begin{proof}
Set
\begin{equation*}
v(x)=\mathrm e^{\frac{\mathrm i}{4b_{20}}\,\left(b_{11} x^2+2b_{10} x\right)}u(x),
\end{equation*}
and
\begin{equation*}
w(z)=v\left(\left(-\frac{\Delta_1}{4b_{20}^2}\right)^{\!\!-\frac 13}z-\frac {\Delta_0}{\Delta_1}\right).
\end{equation*}
Then a simple computation shows that $Bu=0$ if and only if
$w$ solves
the Airy equation
\begin{equation*}
w''(z)-z w(z)=0. \qedhere
\end{equation*}
\end{proof}
  \begin{proposition}\label{pro:51}
Let $u_1$ and $u_2$ be as in Proposition \ref{pro:65}. Then we have the following asymptotic expansions.
\begin{align}\label{eqn:813}
&\begin{aligned}[t]
&c_1 u_1(x)+c_2 u_2(x)=
\frac {1}{2\sqrt{\pi}}
\left(-\frac{\Delta_1}{4b_{20}^2}\right)^{\!\!-\frac 1{12}}\,\abs{x}^{-\frac 14}\cdot
\\&\quad \cdot\left\{2c_2\mathrm e^{\mathrm i \Sigma_-(x)}
\left(1+\mathcal O\left(\abs{x}^{-1}\right)\right)-
\left(c_1+\mathrm i c_2\right)\mathrm e^{\mathrm i \Sigma_+(x)}
\left(1+\mathcal O\left(\abs{x}^{-1}\right)\right)\right\},
\qquad
\text{\upshape for $x\to+\infty$},
\end{aligned}
\\\label{eqn:814}
&\begin{aligned}[t]
&c_1 u_1(x)+c_2 u_2(x)=
\frac{1}{2\sqrt{2\pi}}\,
\left(-\frac{\Delta_1}{4b_{20}^2}\right)^{\!\!-\frac 1{12}}\,\abs{x}^{-\frac 14}\cdot
\\&\qquad\cdot\left\{\bigl((1-\mathrm i)c_1+(1+\mathrm i)c_2\bigr)\,\mathrm e^{\mathrm i \Sigma_+(x)}
\left(1+\mathcal O\left(\abs{x}^{-1}\right)\right)\right.+
\\ & \hspace{4cm} +
\left.\bigl((1+\mathrm i)c_1+(1-\mathrm i)c_2\bigl)\, \mathrm e^{\mathrm i \Sigma_-(x)}
\left(1+\mathcal O\left(\abs{x}^{-1}\right)\right)\right\},
\qquad\text{\upshape for $x\to-\infty$}.
\end{aligned}
\end{align}
\end{proposition}
\begin{proof}
First we  prove the  following asymptotic expansions.
\begin{align}
\label{eqn:481}
&\begin{aligned}[t]
&v_1(x)=
\frac{1}{2\sqrt\pi}\left(-\frac{\Delta_1}{4b_{20}^2}\right)^{\!\!-\frac 1{12}}\,\abs{x}^{-\frac 14}
\,\mathrm e^{-\frac 23 \left(-\frac{\Delta_1}{4b_{20}^2}\right)^{\!\!\frac 12}\abs{x}^{\frac 32}\left(1+\frac{\Delta_0}{\Delta_1x}\right)^{\!\!\frac 32}}
\left(1+\mathcal O\left(\abs{x}^{-1}\right)\right),
\end{aligned}
\\\label{eqn:482}
&\begin{aligned}[t]
&v_2(x)=
\frac {1}{2\sqrt{\pi}}
\left(-\frac{\Delta_1}{4b_{20}^2}\right)^{\!\!-\frac 1{12}}\,\abs{x}^{-\frac 14}
\,\left\{2\mathrm e^{\frac 23 \left(-\frac{\Delta_1}{4b_{20}^2}\right)^{\!\!\frac 12}
\abs{x}^{\frac 32}\left(1+\frac{\Delta_0}{\Delta_1x}\right)^{\!\!\frac 32}}
\left(1+\mathcal O\left(\abs{x}^{-1}\right)\right)+\right.
\\&\qquad\qquad\left.+\mathrm i\mathrm e^{-\frac 23 \left(-\frac{\Delta_1}{4b_{20}^2}\right)^{\!\!\frac 12}
\abs{x}^{\frac 32}\left(1+\frac{\Delta_0}{\Delta_1x}\right)^{\!\!\frac 32}}
\left(1+\mathcal O\left(\abs{x}^{-1}\right)\right)\right\},
\end{aligned}
\end{align}
for $x\to+\infty$, and
\begin{align}
\label{eqn:483}
&\begin{aligned}[t]
v_1(x)&=
\frac{1}{2\sqrt{2\pi}}\,
\left(-\frac{\Delta_1}{4b_{20}^2}\right)^{\!\!-\frac 1{12}}\,\abs{x}^{-\frac 14}
\left\{(1-\mathrm i)\,\mathrm e^{\mathrm i\frac 23\,\left(-\frac{\Delta_1}{4b_{20}^2}\right)^{\!\!\frac 12}
\abs{x}^{\frac 32}\left(1+\frac{\Delta_0}{\Delta_1x}\right)^{\!\!\frac 32}}
\left(1+\mathcal O\left(\abs{x}^{-1}\right)\right)+\right.
\\&\qquad\qquad
\left.+(1+\mathrm i)\, \mathrm e^{-\mathrm i\frac 23\,\left(-\frac{\Delta_1}{4b_{20}^2}\right)^{\!\!\frac 12}
\abs{x}^{\frac 32}\left(1+\frac{\Delta_0}{\Delta_1x}\right)^{\!\!\frac 32}}
\left(1+\mathcal O\left(\abs{x}^{-1}\right)\right)\right\},
\end{aligned}
\\
\label{eqn:485}
&\begin{aligned}[t]
v_2(x)&=
\frac{1}{2\sqrt{2\pi}}\,
\left(-\frac{\Delta_1}{4b_{20}^2}\right)^{\!\!-\frac 1{12}}\,\abs{x}^{-\frac 14}
\left\{(1+\mathrm i)\,\mathrm e^{\mathrm i\frac 23\,\left(-\frac{\Delta_1}{4b_{20}^2}\right)^{\!\!\frac 12}
\abs{x}^{\frac 32}\left(1+\frac{\Delta_0}{\Delta_1x}\right)^{\!\!\frac 32}}
\left(1+\mathcal O\left(\abs{x}^{-1}\right)\right)+\right.
\\&\qquad\qquad
\left.+(1-\mathrm i)\, \mathrm e^{-\mathrm i\frac 23\,\left(-\frac{\Delta_1}{4b_{20}^2}\right)^{\!\!\frac 12}
\abs{x}^{\frac 32}\left(1+\frac{\Delta_0}{\Delta_1x}\right)^{\!\!\frac 32}}
\left(1+\mathcal O\left(\abs{x}^{-1}\right)\right)\right\},
\end{aligned}
\end{align}
for $x\to-\infty$.

Let  $0<\epsilon<\pi/3$.  Airy functions have the following asymptotic expansions for $\abs{z}\to\infty$, see
\cite[10.4.59,  and 10.4.65]{Abramowitz-Stegun}:
\begin{align}
\label{eqn:486}
&\Ai(z)=\frac{z^{-\frac 14}}{2\sqrt\pi}\,\mathrm e^{-\frac 23 z^{\frac 32}}\left(1
+\mathcal O\left(\abs{z}^{-\frac 32}\right)\right),\qquad\text{\upshape for $\abs{\Arg z}\le \pi-\epsilon$}.
\\
\label{eqn:490}
&\begin{aligned}[t]
&\Bi (z)=
\sqrt{\frac 2 \pi}\,\mathrm e^{\mathrm i\frac \pi 6 }\left(\mathrm e^{-\mathrm i\frac \pi 3} z\right)^{-\frac14}\cdot
\left\{\sin\left(\frac 23\,\left(\mathrm e^{-\mathrm i\frac \pi 3} z\right)^{\frac32}+\frac \pi 4-\frac{\log 2}2\,\mathrm i\right)
\left(1+\mathcal O\left(\abs{z}^{-3}\right)\right)-\right.
\\
&\qquad\quad\left.-\cos\left(\frac 23\,\left(\mathrm e^{-\mathrm i\frac \pi 3} z\right)^{\frac32}+\frac \pi 4-\frac{\log 2}2\,\mathrm i\right)
\cdot\mathcal O\left(\abs{z}^{-\frac 32}\right)\right\}
\\
&\quad=
\frac {\mathrm e^{\mathrm i\frac \pi 4 }z^{-\frac14}}{\sqrt{2\pi}}
\left\{(1-\mathrm i)\,\mathrm e^{\frac 23 z^{\frac 32}}\left(1+\mathcal O\left(\abs{z}^{-\frac 32}\right)\right)+
\frac {1+\mathrm i}2\,\mathrm e^{-\frac 23 z^{\frac 32}}\left(1+\mathcal O\left(\abs{z}^{-\frac 32}\right)\right)\right\},
\\
&\hspace{10cm}\text{\upshape for $-\frac \pi 3+\epsilon\le\Arg z\le \frac \pi 3+\epsilon$},
\end{aligned}
\end{align}
and, see \cite[10.4.60, and 10.4.64]{Abramowitz-Stegun}:
\begin{align}
\label{eqn:487}
&\begin{aligned}[t]
&\Ai (z)=\frac{(-z)^{-\frac 14}}{\sqrt\pi}\,\left\{\sin\left(\frac 23\,(-z)^{\frac 32}+\frac \pi 4\right)
\left(1
+\mathcal O\left(\abs{z}^{-3}\right)\right)-\cos\left(\frac 23\,(-z)^{\frac 32}+\frac \pi 4\right)
\cdot\mathcal O\left(\abs{z}^{-\frac 32}\right)\right\}
\\
&\quad=\frac{(-z)^{-\frac 14}}{2\sqrt{2\pi}}\left\{(1-\mathrm i)\,\mathrm e^{\mathrm i\frac 23\,(-z)^{\frac 32}}
\left(1+\mathcal O\left(\abs{z}^{-\frac 32}\right)\right)+
(1+\mathrm i)\, \mathrm e^{-\mathrm i\frac 23\,(-z)^{\frac 32}}
\left(1+\mathcal O\left(\abs{z}^{-\frac 32}\right)\right)
\right\},
\\
&\hspace{10cm}\text{\upshape for $\abs{\Arg (-z)}\le \frac{2\pi}3-\epsilon$},
\end{aligned}
\\
\label{eqn:489}
&\begin{aligned}[t]
&\Bi (z)=\frac{(-z)^{-\frac14}}{\sqrt\pi}\left\{\cos\left(\frac 23\,(-z)^{\frac32}+\frac \pi 4\right)
\left(1
+\mathcal O\left(\abs{z}^{-3}\right)\right)+\sin\left(\frac 23\,(-z)^{\frac32}+\frac \pi 4\right)
\cdot\mathcal O\left(\abs{z}^{-\frac 32}\right)\right\}
\\
&\quad=\frac {(-z)^{-\frac 14}}{2\sqrt{2\pi}}\left\{(1+\mathrm i)\,\mathrm e^{\mathrm i\frac 23\,(-z)^{\frac 32}}
\left(1+\mathcal O\left(\abs{z}^{-\frac 32}\right)\right)+
(1-\mathrm i)\, \mathrm e^{-\mathrm i\frac 23\,(-z)^{\frac 32}}
\left(1+\mathcal O\left(\abs{z}^{-\frac 32}\right)\right)
\right\},
\\
&\hspace{10cm}\text{\upshape for $\abs{\Arg (-z)}\le \frac{2\pi}3-\epsilon$}.
\end{aligned}
\end{align}

Let
\begin{equation}\label{eqn:812}
z=\left(-\frac{\Delta_1}{4b_{20}^2}\right)^{\!\!\frac 13}\left(x+\frac{\Delta_0}{\Delta_1}\right),
\end{equation}
and
\begin{equation*}
0<\epsilon<\frac \pi 6.
\end{equation*}
Since
\begin{gather*}
\pm z=\left(-\frac{\Delta_1}{4b_{20}^2}\right)^{\!\!\frac 13}\abs{x}\left(1+\frac{\Delta_0}{\Delta_1x}\right),
\qquad\text{\upshape for $x\to\pm \infty$},
\\
-\frac \pi 3<\Arg\left(\left(-\frac{\Delta_1}{4b_{20}^2}\right)^{\!\!\frac 13}\right)\le \frac \pi 3,
\intertext{and}
\lim_{\abs{x}\to\infty}\abs{\Arg\left(1+\frac{\Delta_0}{\Delta_1 x}\right)}= 0,
\end{gather*}
we have for $x\to\pm \infty$:
\begin{equation*}
\Arg(\pm z)=\Arg\left(\left(-\frac{\Delta_1}{4b_{20}^2}\right)^{\!\!\frac 13}\right)+\Arg\left(1+\frac{\Delta_0}{\Delta_1 x}\right)\le
\frac \pi 3+\epsilon\le \frac {2\pi}3-\epsilon,
\end{equation*}
for
\begin{equation*}\abs{\Arg\left(1+\frac{\Delta_0}{\Delta_1 x}\right)}\le\epsilon,
\end{equation*}
and
\begin{equation*}
\Arg(\pm z)=\Arg\left(\left(-\frac{\Delta_1}{4b_{20}^2}\right)^{\!\!\frac 13}\right)+\Arg\left(1+\frac{\Delta_0}{\Delta_1 x}\right)\ge
-\frac \pi 3+\epsilon\ge -\frac {2\pi}3+\epsilon,
\end{equation*}
for
\begin{equation*}\abs{\Arg\left(1+\frac{\Delta_0}{\Delta_1 x}\right)}\le
\frac 12\left[\frac \pi 3+\Arg\left(\left(-\frac{\Delta_1}{4b_{20}^2}\right)^{\!\!\frac 13}\right)\right],
\end{equation*}
and
\begin{equation*}
\epsilon\le\frac 12\left[\frac \pi 3+\Arg\left(\left(-\frac{\Delta_1}{4b_{20}^2}\right)^{\!\!\frac 13}\right)\right].
\end{equation*}
This shows that we can make the substitution \eqref{eqn:812} into expansions \eqref{eqn:486},
 \eqref{eqn:490},  \eqref{eqn:487}, and  \eqref{eqn:489}.

Since
\begin{equation*}
(\pm z)^{-\frac 14}=\left(-\frac{\Delta_1}{4b_{20}^2}\right)^{\!\!-\frac 1{12}}\abs{x}^{-\frac 14}\left(1+\mathcal O\left(\abs{x}^{-1}\right)\right),\qquad
\text{\upshape for $x\to \pm\infty$},
\end{equation*}
thanks to \eqref{eqn:816}, we obtain \eqref{eqn:481},
\eqref{eqn:482}, \eqref{eqn:483},  and \eqref{eqn:485}.

Now observe that
\begin{equation*}
-\mathrm i (-x)^{\frac 32}=\mathrm e^{\mathrm i \frac {3\pi}2}(-x)^{\frac 32}=
\mathrm e^{\frac 32\left(\log(-x)+\mathrm i\pi\right)}=x^{\frac 32},\quad\text{\upshape for $x<0$},
\end{equation*}
and (see \eqref{eqn:817})
\begin{equation*}
\left(-\frac{\Delta_1}{4b_{20}^2}\right)^{\!\!\frac 12}=\mathrm i\sigma\!\left(\frac{\Delta_1}{4b_{20}^2}\right)
\left(\frac{\Delta_1}{4b_{20}^2}\right)^{\!\!\frac 12}.
\end{equation*}
It follows that
\begin{multline*}
-\frac{\mathrm i}{4b_{20}}\left(b_{11}x^2+2b_{10}x\right)
\pm\frac 23 \left(-\frac{\Delta_1}{4b_{20}^2}\right)^{\!\!\frac 12}
\abs{x}^{\frac 32}\left(1+\frac{\Delta_0}{\Delta_1x}\right)^{\!\!\frac 32}
= \\ =
-\frac{\mathrm i}{4}\left\{\frac { b_{11}}{b_{20}}\,x^2+2\frac {b_{10}}{b_{20}}\,x
\mp\frac 43 \sigma\!\left(\frac{\Delta_1}{b_{20}^2}\right)\left(\frac{\Delta_1}{b_{20}^2}\right)^{\!\!\frac 12}
x^{\frac 32}\left(1+\frac{\Delta_0}{\Delta_1x}\right)^{\!\!\frac 32}\right\}=\mathrm i \Sigma_\mp(x),
\qquad\text{\upshape for $x\to+\infty$},
\end{multline*}
and
\begin{multline*}
-\frac{\mathrm i}{4b_{20}}\left(b_{11}x^2+2b_{10}x\right)
\pm\mathrm i\frac 23 \left(-\frac{\Delta_1}{4b_{20}^2}\right)^{\!\!\frac 12}
\abs{x}^{\frac 32}\left(1+\frac{\Delta_0}{\Delta_1x}\right)^{\!\!\frac 32}
= \\ =
-\frac{\mathrm i}{4}\left\{\frac { b_{11}}{b_{20}}\,x^2+2\frac {b_{10}}{b_{20}}\,x
\pm\frac 43 \sigma\!\left(\frac{\Delta_1}{b_{20}^2}\right)\left(\frac{\Delta_1}{b_{20}^2}\right)^{\!\!\frac 12}
x^{\frac 32}\left(1+\frac{\Delta_0}{\Delta_1x}\right)^{\!\!\frac 32}\right\}=\mathrm i \Sigma_\pm(x),
\qquad\text{\upshape for $x\to-\infty$}.
\end{multline*}
It follows that \eqref{eqn:815}, \eqref{eqn:481},
\eqref{eqn:482}, \eqref{eqn:483},  and \eqref{eqn:485}
imply \eqref{eqn:813}, and \eqref{eqn:814}.
\end{proof}
\paragraph{Assume $\Delta_2=\Delta_1= 0$.}
In this case it is sufficient to observe that the general solution is given by  (see \eqref{eqn:817})
\begin{equation}\label{eqn:819}
cu_1(x)+c_2u_2(x)=\begin{cases}
\begin{aligned}[b]
&\mathrm e^{-\frac {\mathrm i}4 \left(\frac {b_{11}}{b_{20}}x^2+2\frac {b_{10}}{b_{20}} x\right)}
\biggl\{c_1\mathrm e^{\frac {\mathrm i}2\sigma\!\left(\frac {\Delta_0}{b_{20}^2}\right)\left(\frac {\Delta_0}{b_{20}^2}\right)^{\!\!\frac 12} x}
+c_2\mathrm e^{-\frac {\mathrm i}2\sigma\!\left(\frac {\Delta_0}{b_{20}^2}\right)\left(\frac {\Delta_0}{b_{20}^2}\right)^{\!\!\frac 12} x}\biggr\}
=\\&\qquad =c_1\mathrm e^{\mathrm i \Sigma_-(x)}+c_2\mathrm e^{\mathrm i \Sigma_+(x)},
\end{aligned}&\text{\upshape if $\Delta_0\ne 0$},
\\[8pt]
\mathrm e^{-\frac {\mathrm i}4 \left(\frac {b_{11}}{b_{20}}x^2+2\frac {b_{10}}{b_{20}} x\right)}
\left(c_1 +c_2 x\right)=
\left(c_1 +c_2 x\right)\mathrm e^{\mathrm i \Sigma_\pm(x)},&\text{\upshape if $\Delta_0=0$}.
\end{cases}\end{equation}

\subsubsection{Proof of Theorem  \ref{thm:16}}\label{subsec:2}

\begin{theorem}\label{thm:32}
 $B$ is globally regular and one-to-one if and only if
\begin{equation}\label{eqn:832}
\mathrm e^{\mathrm i x\Xi_\pm}\notin \mathcal S',
\end{equation}
 or
\begin{equation}\label{eqn:833}
\mathrm e^{\mathrm i x\Xi_-}\notin \mathcal S',\quad
\mathrm e^{\mathrm i x\Xi_+}\in \mathcal S,\quad \Delta_2\ne 0, \quad\lambda\notin\{1+2n:n\in\mathbb Z_+\},
\end{equation}
 or
\begin{equation}\label{eqn:834}
\mathrm e^{\mathrm i x\Xi_-}\notin \mathcal S',\quad
\mathrm e^{\mathrm i x\Xi_+}\in \mathcal S,\quad \Delta_2= 0.
\end{equation}
\end{theorem}
\begin{proof}
We have  the following asymptotic expansions for $\abs{x}\to\infty$.
\begin{trivlist}
\item
If $\Delta_2\ne 0$,
\begin{align}\notag
x\Xi_\pm(x)&=
-\frac 12 \left(\frac {b_{11}}{b_{20}}
\pm\sigma\!\left(\frac {\Delta_2}{b_{20}^2}\right)\left(\frac {\Delta_2}{b_{20}^2}\right)^{\!\!\frac 12}\right)x^2
-\frac 12\left(\frac {b_{10}}{b_{20}}\pm
\frac 12\,\sigma\!\left(\frac {\Delta_2}{b_{20}^2}\right)\left(\frac {\Delta_2}{b_{20}^2}\right)^{\!\!\frac 12}
\frac {\Delta_1}{\Delta_2}\right)x+
 \mathcal O(1),
\\\label{eqn:838}
\Sigma_\pm(x)&=
-\frac 14\left(\frac {b_{11}}{b_{20}}
\pm\sigma\!\left(\frac {\Delta_2}{b_{20}^2}\right)\left(\frac {\Delta_2}{b_{20}^2}\right)^{\!\!\frac 12}\right)x^2
-\frac 14\left(2\,\frac {b_{10}}{b_{20}}\pm
\sigma\!\left(\frac {\Delta_2}{b_{20}^2}\right)\left(\frac {\Delta_2}{b_{20}^2}\right)^{\!\!\frac 12}
\frac {\Delta_1}{\Delta_2}\right)x+
 \mathcal O(1).\end{align}
\item
If $\Delta_2= 0$ and $\Delta_1\ne 0$,\footnote{\;Observe that when $x<0$ we have $x\cdot x^{\frac 12}=-\mathrm e^{\mathrm i\frac \pi 2}(-x)^{\frac 32}=
\mathrm e^{\mathrm i\frac {3\pi} 2}(-x)^{\frac 32}=x^{\frac 32}$.}
\begin{align}\notag
x\Xi_\pm(x)&=
-\frac 12\frac {b_{11}}{b_{20}}\,x^2
\mp\,\frac 12\sigma\!\left(\frac{\Delta_1}{b_{20}^2}\right)\left(\frac{\Delta_1}{b_{20}^2}\right)^{\!\!\frac 12}
x^{\frac 32}
-\frac 12\frac {b_{10}}{b_{20}}\,x
\mp\,\frac 14\,\sigma\!\left(\frac{\Delta_1}{b_{20}^2}\right)\left(\frac{\Delta_1}{b_{20}^2}\right)^{\!\!\frac 12}\frac{\Delta_0}{\Delta_1}\,
x^{\frac 12}+\mathcal O\left(\abs{x}^{-\frac 12}\right),
\\\label{eqn:839}
\Sigma_\pm(x)&=
-\frac 14\frac {b_{11}}{b_{20}}\,x^2
\mp\frac 13\,\sigma\!\left(\frac{\Delta_1}{b_{20}^2}\right)\left(\frac{\Delta_1}{b_{20}^2}\right)^{\!\!\frac 12}
x^{\frac 32}
-\frac 12\frac {b_{10}}{b_{20}}\,x
\mp \frac 12\,\sigma\!\left(\frac{\Delta_1}{b_{20}^2}\right)\left(\frac{\Delta_1}{b_{20}^2}\right)^{\!\!\frac 12}\frac{\Delta_0}{\Delta_1}\,
x^{\frac 12}+\mathcal O\left(\abs{x}^{-\frac 12}\right).
\end{align}
\item
If $\Delta_2=\Delta_1=0$,
\begin{align}\notag
x\Xi_\pm(x)&=
-\frac 12\frac {b_{11}}{b_{20}}\,x^2-\frac 12\left(\frac {b_{10}}{b_{20}}
\pm \sigma\!\left(\frac {\Delta_0}{b_{20}^2}\right)\left(\frac {\Delta_0}{b_{20}^2}\right)^{\!\!\frac 12}\right)x,
\\\label{eqn:840}
\Sigma_\pm(x)&=
-\frac 14\,\frac { b_{11}}{b_{20}}\,x^2-\frac 12\left(\frac {b_{10}}{b_{20}}\,
\pm \sigma\!\left(\frac {\Delta_0}{b_{20}^2}\right)\left(\frac {\Delta_0}{b_{20}^2}\right)^{\!\!\frac 12}\right)x.
\end{align}
\end{trivlist}
From these asymptotic expansions it  follows that
\begin{itemize}
\item[(I)]
\eqref{eqn:832}, \eqref{eqn:833}, and  \eqref{eqn:834} are equivalent to
\begin{equation}\label{eqn:835}
\mathrm e^{\mathrm i \Sigma_\pm}\notin \mathcal S',
\end{equation}
 or
\begin{equation}\label{eqn:836}
\mathrm e^{\mathrm i \Sigma_-}\notin \mathcal S',\quad
\mathrm e^{\mathrm i \Sigma_+}\in \mathcal S,\quad \Delta_2\ne 0, \quad\lambda\notin\{1+2n:n\in\mathbb Z_+\},
\end{equation}
 or
\begin{equation}\label{eqn:837}
\mathrm e^{\mathrm i \Sigma_-}\notin \mathcal S',\quad
\mathrm e^{\mathrm i \Sigma_+}\in \mathcal S,\quad \Delta_2= 0.
\end{equation}
\item[(II)]
Thanks to Proposition \ref{pro:47}, global regularity is equivalent to
\begin{equation}\label{eqn:829}
\mathrm e^{\mathrm i \Sigma_\pm}\in \mathcal S\cup(\mathcal C^\infty\setminus\mathcal S').
\end{equation}
\item[(III)]
Then, if $B$ is globally regular, there are only three possible behaviors of $\mathrm e^{\mathrm i \Sigma_\pm}$:
\begin{align*}
&\mathrm e^{\mathrm i \Sigma_\pm}\notin \mathcal S',\\
&\mathrm e^{\mathrm i \Sigma_-}\notin \mathcal S',\quad \mathrm e^{\mathrm i \Sigma_+}\in \mathcal S,\\
&\mathrm e^{\mathrm i \Sigma_\pm}\in \mathcal S.
\end{align*}
\end{itemize}

Since \eqref{eqn:835}, \eqref{eqn:836}, and \eqref{eqn:837} imply \eqref{eqn:829}, we have only to
show that
\begin{enumerate}
\item\label{itm:32} $\mathrm e^{\mathrm i\Sigma_\pm}\notin\mathcal S'$ $\implies u\notin \mathcal S'$,
\item\label{itm:33} $\mathrm e^{\mathrm i\Sigma_-}\notin\mathcal S'$,
$\mathrm e^{\mathrm i\Sigma_+}\in\mathcal S$,
$\Delta_2\ne 0$, and $\lambda\notin\{1+2n:n\in\mathbb Z_+\}$ $\implies u\notin \mathcal S'$,
\item\label{itm:34} $\mathrm e^{\mathrm i\Sigma_-}\notin\mathcal S'$,
$\mathrm e^{\mathrm i\Sigma_+}\in\mathcal S$,
$\Delta_2\ne 0$, and $\lambda\in\{1+2n:n\in\mathbb Z_+\}$  $\implies u\in \mathcal S$,
\item\label{itm:47} $\mathrm e^{\mathrm i\Sigma_-}\notin\mathcal S'$,
$\mathrm e^{\mathrm i\Sigma_+}\in\mathcal S$,
$\Delta_2= 0$, $\implies u\notin \mathcal S'$,
\item\label{itm:35} $\mathrm e^{\mathrm i\Sigma_\pm}\in\mathcal S$ $\implies u\in \mathcal S$,
\end{enumerate}
where
\begin{equation*}
u=c_1u_1+c_2u_2,
\end{equation*}
$u_1$, and $u_2$ are as in Propositions \ref{pro:50}, \ref{pro:61}, and \ref{pro:51}, and formula \eqref{eqn:819}, and  $\abs{c_1}+\abs{c_2}>0$.

Since all assumptions in \eqref{itm:32}--\eqref{itm:35} imply that $B$ is globally regular, we have that
\begin{equation*}
u\in\mathcal S \iff \lim_{\abs{x}\to\infty}u(x)=0.
\end{equation*}
At last implications \eqref{itm:32}--\eqref{itm:35} follow by computing the limit of $u$ as $\abs{x}\to \infty$ by making use of the asymptotic expansions \eqref{eqn:838}, \eqref{eqn:839}, and \eqref{eqn:840}, and   Propositions \ref{pro:50}, \ref{pro:61}, and \ref{pro:51}, and formula \eqref{eqn:819}.

We leave the details to the reader.
\end{proof}

\bigskip\noindent\textbf{Proof of Theorem \ref{thm:16}}
It follows from Theorem \ref{thm:21}, Lemma \ref{lem:11},  Proposition \ref{pro:49},  and Theorem \ref{thm:32}.
\qed

\section{Asymptotic expansions of functions $\Phi$, and $\Theta$}\label{sec:5}

\subsection{Lemmas on Gamma Function.}

The \emph{Euler Gamma Function} is defined by
\begin{equation*}
\Gamma(z)=\int_0^\infty t^{z-1}\mathrm e ^{-t}\,\mathrm d t,\qquad \text{\upshape for $\Re z>0$}.
\end{equation*}
This function can be extended to a meromorphic function with simple pole at every $k\in\mathbb Z_-$, by the formula (see \cite[1.1]{Lebedev}):
\begin{equation*}
\Gamma(z)=\sum_{k=0}^\infty\frac {(-1)^k}{k!}\,\frac 1{z+k}+\int_1^\infty t^{z-1}\mathrm e^{-t}\,\mathrm d t.
\end{equation*}

\begin{lemma}\label{lem:9}
Given two complex numbers $p$ and $q$ such that $\Re p>0$, and $\Re q>0$, we have
\begin{equation}\label{eqn:6}
\mathrm e^{\mathrm i\theta p}\int_0^{\infty} \frac {t^{p-1}}{(1+\mathrm e^{\mathrm i \theta} t) ^{p+q}}\,\mathrm d t=
\frac {\Gamma(p)\Gamma(q)}{\Gamma(p+q)},\qquad \text{\upshape for all $\theta\in\mathbb R$}.
\end{equation}
\end{lemma}
\begin{proof}
Since (\cite[(1.5.3) and (1.5.6)]{Lebedev})
\begin{equation*}
\int_0^{\infty} \frac {t^{p-1}}{(1+t)^{p+q}}\,\mathrm d t=B(p,q)=\frac {\Gamma(p)\Gamma(q)}{\Gamma(p+q)},
\end{equation*}
where $B$ is the \emph{Euler Beta Function}, it suffices to show that the left-hand side of \eqref{eqn:6} is constant with respect to $\theta$.
But this follows from
\begin{multline*}
\frac {\mathrm d}{\mathrm d \theta}\left\{\mathrm e^{\mathrm i\theta p}\int_0^{\infty} \frac {t^{p-1}}{(1+\mathrm e^{\mathrm i \theta} t) ^{p+q}}\,\mathrm d t\right\}=
\\
=\mathrm i p \mathrm e^{\mathrm i\theta p}\int_0^{\infty} \frac {t^{p-1}}{(1+\mathrm e^{\mathrm i \theta} t) ^{p+q}}\,\mathrm d t
-(p+q)\mathrm i \mathrm e^{\mathrm i \theta (p+1)}\int_0^{\infty} \frac {t^{p}}{(1+\mathrm e^{\mathrm i \theta} t) ^{p+q+1}}\,\mathrm d t
=0,
\end{multline*}
because
\begin{equation*}\begin{split}
p \int_0^{\infty} \frac {t^{p-1}}{(1+\mathrm e^{\mathrm i \theta} t) ^{p+q}}\,\mathrm d t&=
\left[\frac {t^{p}}{(1+\mathrm e^{\mathrm i \theta} t) ^{p+q}}\right]_{t=0}^{t=\infty}
-\int_0^{\infty}t^{p} \frac{\mathrm d}{\mathrm d t} (1+\mathrm e^{\mathrm i \theta} t) ^{-(p+q)}\,\mathrm d t
\\&=
(p+q)\mathrm e^{\mathrm i \theta}\int_0^{\infty} \frac {t^{p}}{(1+\mathrm e^{\mathrm i \theta} t) ^{p+q+1}}\,\mathrm d t. \qedhere
\end{split}\end{equation*}
\end{proof}

\begin{lemma}\label{lem:10}
If $\Re z>0$,
and $\Re p>0$, we have
\begin{equation}\label{eqn:7}
z^p
\int_0^{\infty} t^{p-1}\mathrm e^{-t z}\,\mathrm d t=\Gamma (p).
\end{equation}
\end{lemma}
\begin{proof}
Since $\Re z>0$, the left-hand side of \eqref{eqn:7} is analytic. Let

 Differentiate the left-end side of \eqref{eqn:7}:
\begin{equation}\label{eqn:617}
\frac {\mathrm d}{\mathrm d z}\left\{z^ p
\int_0^{\infty} t^{p-1}\mathrm e^{-z t}\,\mathrm d t\right\}=
p z^{p-1}
\int_0^{\infty} t^{p-1}\mathrm e^{-z t}\,\mathrm d t
-z^p\int_0^{\infty} t^{p}\mathrm e^{-z t}\,\mathrm d t.
\end{equation}
Since $\Re p>0$, an integration by parts yields:
\begin{equation}\label{eqn:66}
p z^{p-1}
\int_0^{\infty} t^{p-1}\mathrm e^{-z t}\,\mathrm d t
=z^p\int_0^{\infty} t^{p}\mathrm e^{-z t}\,\mathrm d t.
\end{equation}
Then \eqref{eqn:617}, and \eqref{eqn:66} imply
\begin{equation*}
\frac {\mathrm d}{\mathrm d z}\left\{z^ p
\int_0^{\infty} t^{p-1}\mathrm e^{-z t}\,\mathrm d t\right\}=0,
\end{equation*}
that is that the left-end  side of \eqref{eqn:7} is constant with respect to $z$.
It follows that
\begin{equation*}
z^ p
\int_0^{\infty} t^{p-1}\mathrm e^{-z t}\,\mathrm d t
=\int_0^{\infty} t^{p-1}\mathrm e^{-t}\,\mathrm d t=\Gamma (p).  \qedhere
\end{equation*}
\end{proof}

\begin{lemma}\label{lem:8}
Let  $\Re p>0$ and $0<\epsilon<\frac \pi 2$. Then
\begin{multline}\label{eqn:13}
\int_0^1 t^{p-1}\,\mathrm e^{-t z}\,\mathrm d t=z^{-p}\left\{\Gamma(p)
+\mathcal O\left(\abs{z}^{\Re p-1}\mathrm e^{-(\sin \epsilon)\abs{z}}\right)\right\},\qquad
\text{\upshape for $\abs{z}\to \infty$, and $\abs{\Arg z}\le \frac \pi 2-\epsilon$}.
\end{multline}
\end{lemma}
\begin{proof}
From \eqref{eqn:7} it follows that
\begin{equation}\label{eqn:71}
z^{p}\int_0^1 t^{p-1}\,\mathrm e^{-t z}\,\mathrm d t-\Gamma(p)
=-z^p\int_1^{\infty} t^{p-1} \mathrm e^{-t z}\,\mathrm d t.
\end{equation}
Let $N=\min\{k\in \mathbb Z_+:\Re p-1-k\le 0\}$. Integrating by parts we get
\begin{equation}\label{eqn:69}
\int_1^{\infty} t^{p-1} \mathrm e^{-t z}\,\mathrm d t=\sum_{k=0}^{N}\frac {\gamma_k}{z^k}\,\mathrm e^{-z}+
\frac{\gamma_{N+1}}{z^{N}}\int_1^{\infty} t^{p-1-N} \mathrm e^{-t z}\,\mathrm d t,
\end{equation}
where
\begin{equation}\label{eqn:5}
\gamma_k=\begin{cases} 0,&\text{\upshape if $k=0$},\\ 1,&\text{\upshape if $k=1$},\\
(p-1)(p-2)\cdots(p-k+1),&\text{\upshape if $k>1$}.\end{cases}
\end{equation}
Since $\Re p-1-N\le 0$, and $\abs{\Arg z}\le \frac \pi 2-\epsilon$, we have
\begin{equation}\label{eqn:70}
\abs{\int_1^{\infty} t^{p-1-N} \mathrm e^{-t z}\,\mathrm d t}\le
\int_1^{\infty}\mathrm e^{-t\abs{z}\cos(\Arg z)}\,\mathrm d t=
\frac {\mathrm e^{-\abs{z}\cos(\Arg z)}}{\abs{z}\cos(\Arg z)}.
\end{equation}
Since $\abs{\Arg z}\le \frac \pi 2-\epsilon$, from \eqref{eqn:69}, \eqref{eqn:5}, and \eqref{eqn:70}, it follows that
\begin{equation*}\begin{split}
\abs{z^p\int_1^{\infty} t^{p-1} \mathrm e^{-t z}\,\mathrm d t}
&\le\sum_{k=0}^{N}\abs{\gamma_k} \abs{z^{p-k}}\abs{\mathrm e^{-z}}+
\abs{\gamma_{N+1}}\abs{z^{p-N}}\frac {\mathrm e^{-\abs{z}\cos(\Arg z)}}{\abs{z}\cos(\Arg z)},
\\&\le\left(\sum_{k=0}^{N+1}\abs{\gamma_k}\right)\mathrm e^{-\Im p \Arg z} \abs{z}^{\Re p-1}\frac {\mathrm e^{-\abs{z}\cos(\Arg z)}}{\cos(\Arg z)}
\\&\le\left(\sum_{k=1}^{N+1}\abs{\gamma_k}\right) \mathrm e^{\frac {\pi\abs{\Im p}}2}\abs{z}^{\Re p-1}\frac {\mathrm e^{-(\sin \epsilon)\abs{z}}}{\sin\epsilon},\qquad\text{\upshape for $\abs{z}\ge 1$}.
\end{split}\end{equation*}
This inequality together with \eqref{eqn:71} implies \eqref{eqn:13}.
\end{proof}

\subsection{Asymptotic behavior of $\Phi$.}

 \begin{proposition}
We have the following integral representation:
\begin{equation}
\label{eqn:94}
\Phi(p,q;z)=\frac{\Gamma(q)}{\Gamma(p)\Gamma(q-p)}\,\mathrm e^z\int_0^1 (1-t)^{p-1} t^{q-p-1}\mathrm e^{-t z}\,\mathrm d t,
\qquad\text{\upshape for $\Re q>\Re p>0$}.
\end{equation}
\end{proposition}
\begin{proof}
We have (see \cite[(1.5.2), and (1.5.6)]{Lebedev})
\begin{equation*}\begin{aligned}[t]
\frac {(p)_n}{(q)_n}&=\frac {\Gamma(p+n)\Gamma(q)}{\Gamma(p)\Gamma(q+n)}=
\frac {\Gamma(p+n)\Gamma(q-p)}{\Gamma(q+n)}\,\frac{\Gamma(q)}{\Gamma(p)\Gamma(q-p)}
\\&=\frac{\Gamma(q)}{\Gamma(p)\Gamma(q-p)}\,B(p+n,q-p)=
\frac{\Gamma(q)}{\Gamma(p)\Gamma(q-p)}\int_0^1 s^{p+n-1}(1-s)^{q-p-1}\,\mathrm d s,
\end{aligned}\end{equation*}
Thus, from \eqref{eqn:362} we obtain
\begin{equation*}\begin{split}
\Phi(p,q;z)&=\sum_{n=0}^{\infty} \frac {(p)_n}{n!(q)_n}\,z^n=
\frac{\Gamma(q)}{\Gamma(p)\Gamma(q-p)}\sum_{n=0}^{\infty} \int_0^1 s^{p-1}(1-s)^{q-p-1}\,\frac {(s z)^n}{n!}\,\mathrm d s
\\&=\frac{\Gamma(q)}{\Gamma(p)\Gamma(q-p)}\int_0^1 s^{p-1}(1-s)^{q-p-1}\,\mathrm e^{s z}\,\mathrm d s
=\frac{\Gamma(q)}{\Gamma(p)\Gamma(q-p)}\,
\mathrm e^z\int_0^1 (1-t)^{p-1} t^{q-p-1}\mathrm e^{-t z}\,\mathrm d t. \qedhere
\end{split}\end{equation*}
\end{proof}

\begin{proposition}[Kummer identity] \label{pro:58}For all $q\notin \mathbb Z_-$ we have
\begin{equation}\label{eqn:12}
\Phi(p,q;z)=\mathrm e^{z}\Phi(q-p,q;-z).
\end{equation}
\end{proposition}
\begin{proof}Assume $\Re q>\Re p>0$, and put $t=1-s$ in the right hand side of \eqref{eqn:94}. We get
\begin{equation*}
\Phi(p,q;z)=\frac{\Gamma(q)}{\Gamma(p)\Gamma(q-p)}\int_0^1 s^{p-1} (1-s)^{q-p-1}\mathrm e^{s z}\,\mathrm d s.
\end{equation*}
Then using again \eqref{eqn:94} we have
\begin{equation*}
\frac{\Gamma(q)}{\Gamma(p)\Gamma(q-p)}\int_0^1 s^{p-1} (1-s)^{q-p-1}\mathrm e^{s z}\,\mathrm d s=
\mathrm e^z \Phi(q-p,q,-z).
\end{equation*}
This proves \eqref{eqn:12} under the additional hypothesis $\Re q>\Re p>0$. However by analytic continuity with respect to $p$ and $q$, \eqref{eqn:12} is true for all $p\in \mathbb C$, and $q\in\mathbb C\setminus \mathbb Z_-$.\end{proof}

\begin{theorem}\label{thm:31}
Let  $0<\epsilon<\pi/2$, $p\in\mathbb C$, and $q \in \mathbb C\setminus \mathbb Z_-$. For all   $N\in\mathbb Z_+$, we have the following asymptotic expansions for $\abs{z}\to\infty$.
\begin{equation}\label{eqn:10}
\Phi(p,q;z)=\mathrm e^z
z^{p-q}
\left\{\frac {\Gamma(q)}{\Gamma(p)}\sum_{k=0}^N\frac {(q-p)_k(1-p)_k}{k!}\, z^{-k}+\mathcal O\left(\abs{z}^{-N-1}\right)\right\},
\qquad\text{\upshape for  $\abs{\Arg z}\le \frac {\pi}{2}-\epsilon$}.
\end{equation}
\end{theorem}
\begin{proof}
Assume
\begin{equation}\label{eqn:11}
\Re q>\Re p>1.
\end{equation}
Using  the binomial expansion and the identity
\begin{equation*}
(-1)^k\binom{p-1}{k}=\frac {(1-p)_k}{k!},\qquad\text{\upshape for all $k\in\mathbb Z_+$},
\end{equation*}
we obtain
\begin{multline}\label{eqn:16}
\int_0^1 (1-t)^{p-1} t^{q-p-1}\mathrm e^{-t z}\,\mathrm d t
=\sum_{k=0}^N \frac {(1-p)_k}{k!} \int_0^1  t^{k+q-p-1}\mathrm e^{-t z}\,\mathrm d t+
\\ +\frac {(1-p)_{N+1}}{N!}\int_0^1\left(\int_0^1(1-s)^N(1-st)^{p-N-2}\,\mathrm d s\right)
 t^{N+q-p} \mathrm e^{-t z} \,\mathrm d t.
\end{multline}
Now, if $N+2\ge \Re p>1$ we have
\begin{equation*}
\abs{\int_0^1(1-s)^N(1-st)^{p-N-2}\,\mathrm d s}\le \int_0^1(1-s)^{\Re p-2}\,\mathrm d s=\frac {1}{\Re p-1}.
\end{equation*}
Then we get
\begin{multline}
\label{eqn:72}
\abs{z^{q-p}\int_0^1\left(\int_0^1(1-s)^N(1-st)^{p-N-2}\,\mathrm d s\right)
 t^{N+q-p} \mathrm e^{-t z} \,\mathrm d t}\le \\ \le
\frac {\abs{z}^{\Re(q-p)}\mathrm e^{-\Im(p-q)\Arg z}}{\Re p-1}\int_0^1 t^{N+\Re(q-p)}\mathrm e^{-t\abs{z}\cos(\Arg z)}\,\mathrm d t \le \\ \le
\frac {\abs{z}^{\Re(q-p)}\mathrm e^{\frac{\pi}{2}\abs{\Im(p-q)}}}{\Re p-1}\int_0^1 t^{N+\Re(q-p)}\mathrm e^{-t(\sin\epsilon)\abs{z}}\,\mathrm d t
=\\=\frac {\abs z^{-N-1}\mathrm e^{\frac{\pi}{2}\abs{\Im(p-q)}}}{\Re p-1}\int_0^{\abs z} s^{N+\Re(q-p)}\mathrm e^{-(\sin\epsilon)s}\,\mathrm d s
\le\frac {\abs z^{-N-1}\mathrm e^{\frac{\pi}{2}\abs{\Im(p-q)}}}{\Re p-1}\int_0^{+\infty} s^{N+\Re(q-p)}\mathrm e^{-(\sin\epsilon)s}\,\mathrm d s,
\end{multline}
for $z\ne0$.

In conclusion, when $N\ge \Re p-2$, from   \eqref{eqn:16}, and \eqref{eqn:72} it follows that
\begin{multline}\label{eqn:74}
\int_0^1 (1-t)^{p-1} t^{q-p-1}\mathrm e^{-t z}\,\mathrm d t
=
\sum_{k=0}^N \frac {(1-p)_k}{k!} \int_0^1  t^{k+q-p-1}\mathrm e^{-t z}\,\mathrm d t+z^{p-q}\mathcal O\left(\abs{z}^{-N-1}\right),
\\ \text{\upshape for $\abs{z}\to\infty$, and $\abs{\Arg {z}}\le \frac \pi 2-\epsilon$}.
\end{multline}
On the other hand, by Lemma \ref{lem:8}, we have
\begin{multline}\label{eqn:73}
\int_0^1  t^{k+q-p-1}\mathrm e^{-t z}\,\mathrm d t=z^{p-q-k}\left\{\Gamma\left(k+q-p\right)+
\mathcal O\left(\abs{z}^{\Re(q-p)+k-1}\mathrm e^{-(\sin\epsilon)\abs{z}}\right)\right\}
=\\=
z^{p-q}\left\{\Gamma\left(k+q-p\right)z^{-k}+
\mathcal O\left(\abs{z}^{-N-1}\right)\right\},
\\\text{\upshape for $\abs{z}\to\infty$, and $\abs{\Arg {z}}\le \frac \pi 2-\epsilon$}.
\end{multline}
At last  \eqref{eqn:10} follows from \eqref{eqn:94}, \eqref{eqn:74}, and \eqref{eqn:73}, when $N\ge \Re p-2$. However this restriction can easily be eliminated, because, we have
\begin{align*}
\Phi(p,q;z)&=
\mathrm e^z z^{p-q}
\left\{\frac {\Gamma(q)}{\Gamma(p)}\sum_{k=0}^{N+M}\frac {(q-p)_k(1-p)_k}{k!}\, z^{-k}+\mathcal O\left(\abs{z}^{-N-M-1}\right)\right\}
\\&=\mathrm e^z z^{p-q}
\left\{\frac {\Gamma(q)}{\Gamma(p)}\sum_{k=0}^N\frac {(q-p)_k(1-p)_k}{k!}\, z^{-k}+
\frac {\Gamma(q)}{\Gamma(p)}\sum_{k=N+1}^{N+M}\frac {(q-p)_k(1-p)_k}{k!}\, z^{-k}
+\mathcal O\left(\abs{z}^{-N-M-1}\right)\right\}
\\&=\mathrm e^z z^{p-q}
\left\{\frac {\Gamma(q)}{\Gamma(p)}\sum_{k=0}^N\frac {(q-p)_k(1-p)_k}{k!}\, z^{-k}+
\mathcal O\left(\abs{z}^{-N-1}\right)\right\},
\end{align*}
where
\begin{equation*}
M=\min\{m\in\mathbb Z_+^\ast\mid m\ge \Re p-N-2\}.
\end{equation*}

It remains to eliminate the restriction $\Re q>\Re p>1$ and prove
 \eqref{eqn:10}  for all  $p\in\mathbb C$, and $q\in\mathbb C\setminus \mathbb Z_-$.

Rewrite the recurrence relation \cite[(9.9.11)]{Lebedev} as
\begin{equation}\label{eqn:17}			
\Phi(p,q;z)=\frac {q+z}{q}\,\Phi(p,q+1;z)-\frac {(q+1-p)z}{q(q+1)}\,\Phi(p,q+2;z).
\end{equation}
If $\Re q +1>\Re p>1$, from \eqref{eqn:17} and \eqref{eqn:10} we obtain that
\begin{equation}\label{eqn:18}\begin{aligned}[t]
&\Phi(p,q;z)=\mathrm e^z z^{p-q-1}\left\{\frac {q+z}q\,\frac {\Gamma(q+1)}{\Gamma(p)}\sum_{k=0}^N\frac {(q+1-p)_k(1-p)_k}{k!}\,z^{-k}+\mathcal O\left(\abs{z}^{-N-1}\right)\right\}
\\&\qquad\quad -\mathrm e^z z^{p-q-2}\left\{\frac {(q+1-p)z}{q(q+1)}\,\frac {\Gamma(q+2)}{\Gamma(p)}\sum_{k=0}^N\frac {(q+2-p)_k(1-p)_k}{k!}\,z^{-k}+\mathcal O\left(\abs{z}^{-N-1}\right)\right\}
\\&\quad=
\left\{\frac {\Gamma(q)}{\Gamma(p)}\sum_{k=0}^N\frac {(q+1-p)_k(1-p)_k}{k!}\,z^{-k}+\mathcal O\left(\abs{z}^{-N-1}\right)\right\}
z^{p-q}\mathrm e^z
\\&\qquad\qquad+
\left\{\frac {\Gamma(q)}{\Gamma(p)}\sum_{k=0}^N \frac {C_k}{k!}\,
z^{-k}+\mathcal O\left(\abs{z}^{-N-1}\right)\right\}
z^{p-q-1}\mathrm e^z,
\\ &\hspace{8cm} \text{\upshape for $\abs{z}\to\infty$, and $\abs{\Arg {z}}\le \frac \pi 2-\epsilon$},
\end{aligned}\end{equation}
with
\begin{equation}\label{eqn:19}\begin{aligned}[t]
C_k&=q(q+1-p)_k(1-p)_k-(q+1-p)(q+2-p)_k(1-p)_k
\\&=
(1-p)_k\left\{q(q+1-p)_k-(q+1-p)_k(q+1+k-p)\right\}=-(q+1-p)_k(1-p)_{k+1}.
\end{aligned}\end{equation}
Substituting \eqref{eqn:19} into \eqref{eqn:18} gives
\begin{equation*}\begin{aligned}[t]
&\Phi(p,q;z)=\mathrm e^z z^{p-q}
\left\{\frac {\Gamma(q)}{\Gamma(p)}\sum_{k=0}^N\frac {(q+1-p)_k(1-p)_k}{k!}\,z^{-k}+\mathcal O\left(\abs{z}^{-N-1}\right)\right\}
\\&\hspace{2cm}-\mathrm e^z z^{p-q-1}
\left\{\frac {\Gamma(q)}{\Gamma(p)}\sum_{k=0}^N \frac {(q+1-p)_k(1-p)_{k+1}}{k!}\,
z^{-k}+\mathcal O\left(\abs{z}^{-N-1}\right)\right\}
\\&\;=\mathrm e^z z^{p-q}
\left\{\frac {\Gamma(q)}{\Gamma(p)}+\frac {\Gamma(q)}{\Gamma(p)}\sum_{k=1}^N\left[\frac {(q+1-p)_k(1-p)_k}{k!}
-\frac {(q+1-p)_{k-1}(1-p)_{k}}{(k-1)!}\right]z^{-k}+\mathcal O\left(\abs{z}^{-N-1}\right)\right\}
\\&\;=\mathrm e^z z^{p-q}
\left\{\frac {\Gamma(q)}{\Gamma(p)}\sum_{k=0}^N\frac {(q-p)_k(1-p)_k}{k!}\,z^{-k}+\mathcal O\left(\abs{z}^{-N-1}\right)\right\},
\\&\hspace{8cm} \text{\upshape for $ \abs{z}\to\infty$, and $\abs{\Arg {z}}\le \frac \pi 2-\epsilon$}.
\end{aligned}\end{equation*}
This shows that \eqref{eqn:10}
 holds for $\Re q>\Re p-1$ and $\Re p>1$.
Iterating we get that \eqref{eqn:10} holds for all $q\in\mathbb C\setminus \mathbb Z_-$ and $\Re p>1$.

Now consider the recurrence relation \cite[(9.9.12)]{Lebedev}:
\begin{equation}\label{eqn:20}
\Phi(p,q;z)=\Phi(p+1,q;z)-\frac z q\,\Phi(p+1,q+1;z).
\end{equation}
Substituting \eqref{eqn:10} into \eqref{eqn:20} gives:
\begin{equation*}\begin{aligned}[t]
\Phi(p,q;z)&=
\mathrm e^{z} z^{p+1-q}
\left\{\frac{\Gamma(q)}{\Gamma(p+1)}\sum_{k=0}^{N+1}\frac {(q-p-1)_k(-p)_k}{k!}\, z^{-k}\,+\mathcal O\left(\abs{z}^{-N-2}\right)\right\}+
\\&\qquad-\mathrm e^{z} z^{p+1-q}
\left\{\frac 1{q}\,
\frac{\Gamma(q+1)}{\Gamma(p+1)}\sum_{k=0}^{N+1}\frac {(q-p)_k(-p)_k}{k!}\, z^{-k}\,+\mathcal O\left(\abs{z}^{-N-2}\right)\right\}
\\&=\mathrm e^{z} z^{p-q}
\left\{\frac{\Gamma(q)}{p\Gamma(p)}\sum_{k=1}^{N+1}\frac {(q-p-1)_k(-p)_k-(q-p)_k(-p)_k}{k!}\, z^{-k+1}\,+\mathcal O\left(\abs{z}^{-N-1}\right)\right\}
\\&=\mathrm e^{z} z^{p-q}
\left\{\frac{\Gamma(q)}{\Gamma(p)}\sum_{k=1}^{N+1}\frac {(q-p)_{k-1}(1-p)_{k-1}}{(k-1)!}\, z^{-k+1}\,+\mathcal O\left(\abs{z}^{-N-1}\right)\right\}
\\&=\mathrm e^{z} z^{p-q}
\left\{\frac{\Gamma(q)}{\Gamma(p)}\sum_{k=0}^{N}\frac {(q-p)_{k}(1-p)_{k}}{k!}\, z^{-k}\,+\mathcal O\left(\abs{z}^{-N-1}\right)\right\},
\\&\hspace{8cm} \text{\upshape for $\abs{z}\to\infty$, and $\abs{\Arg {z}}\le \frac \pi 2-\epsilon$}.
\end{aligned}\end{equation*}

This means that \eqref{eqn:10} holds for $\Re p>0$ and, by iteration, for all   $p\in\mathbb C$.
\end{proof}

\subsection{Asymptotic behavior of $\Theta$.}

For all $p\in\mathbb C$ set
\begin{equation}\label{eqn:620}
\Theta(p;z)=\sqrt \pi\left\{\frac {1}{\Gamma\!\left(p+\frac 12\right)}\Phi\!\left(p,\frac 12;z^2\right)-
\frac {2z}{\Gamma\!\left(p\right)}\Phi\!\left(p+\frac 12,\frac 32;z^2\right)\right\}.
\end{equation}
Observe that  $\Theta$ is an entire analytic function of $z$.
Moreover, since $\frac 1{\Gamma(-n)}=0$ for all $n\in\mathbb Z_+$,   $\Theta$ is  also  an    entire analytic function of $p$.

\begin{proposition}
Consider $p\in \mathbb C$ such that $\Re p>0$. For all  $\theta\in\left\{-\frac \pi 2,0,\frac \pi 2\right\}$, we have the integral representation
\begin{equation}\label{eqn:614}
\Theta(p;z)=\frac {\mathrm e^{\mathrm i p \theta}}{\Gamma(p)}\int_0^{\infty} t^{p-1}
\left(1+\mathrm e^{\mathrm i\theta} t\right )^{-\left(p+\frac 12\right)}
\mathrm e^{-\exp\left(\mathrm i \theta\right) z^2 t}\,\mathrm d t,
\qquad\text{\upshape for all $z\in\mathcal S_\theta$,}
\end{equation}
where
\begin{equation*}
\mathcal S_\theta=\left\{z\in\mathbb C^\ast\mid \abs{\Arg z+\frac \theta 2}<\frac \pi 4\right\}.
\end{equation*}
\end{proposition}
\begin{proof}
We have
\begin{equation}\label{eqn:783}
\abs{\Arg (z^2)+\theta}<\frac \pi 2,\qquad\text{\upshape for all $z\in \mathcal S_\theta$}.
\end{equation}
Then on $\mathcal S_\theta$ we have
\begin{equation*}
\Re \left(\mathrm e^{\mathrm i\theta}z^2\right)=\cos\left(\Arg(z^2)+\theta\right)\abs{z}^2>0
\end{equation*}
and the following integral is convergent:
\begin{equation}\label{eqn:781}
w(z)=
\mathrm e^{-\frac 12 z^2} \int_0^{\infty} t^{p-1}
\left(1+\mathrm e^{\mathrm i \theta} t\right )^{-\left(p+\frac 12\right)} \mathrm e^{-\exp\left(\mathrm i \theta\right) z^2 t}\,\mathrm d t.
\end{equation}
We have
\begin{align*}
&w''-(z^2+4p-1)w=
\\&\quad  =
\mathrm e^{-\frac 12 z^2} \int_0^{\infty} t^{p-1}
\left(1+\mathrm e^{\mathrm i \theta} t\right )^{-\left(p+\frac 12\right)} \mathrm e^{-\exp\left(\mathrm i \theta\right) z^2 t}
\Bigl\{\left(1+2\mathrm e^{\mathrm i \theta}t\right)^2 z^2-\left(1+2\mathrm e^{\mathrm i \theta}t\right)
-(z^2+4p-1)\Bigr\}\,\mathrm d t
\\&\quad  =
-4\mathrm e^{-\frac 12 z^2} \int_0^{\infty} \frac {\mathrm d}{\mathrm d t}
\Bigl\{t^p
\left(1+\mathrm e^{\mathrm i \theta} t\right )^{-\left(p-\frac 12\right)} \mathrm e^{-\exp\left(\mathrm i \theta\right) z^2 t}\Bigr\}\,\mathrm d t
\\&\quad  =
-4\mathrm e^{-\frac 12 z^2} \Bigl[t^p
\left(1+\mathrm e^{\mathrm i \theta} t\right )^{-\left(p-\frac 12\right)} \mathrm e^{-\exp\left(\mathrm i \theta\right) z^2 t}\Bigr]_{t=0}^{t=\infty}
=0.
\end{align*}
This means that \eqref{eqn:781}
is a solution to \eqref{eqn:502} with $\lambda=1-4p$.
By Proposition \ref{pro:56} there exist $c_1,c_2\in\mathbb C$ such that
\begin{equation}\label{eqn:618}
\int_0^{\infty} t^{p-1}
\left(1+\mathrm e^{\mathrm i \theta} t\right )^{-\left(p+\frac 12\right)} \mathrm e^{-\exp\left(\mathrm i \theta\right) z^2 t}\,\mathrm d t=
c_1\Phi\!\left(p,\frac 12;z^2\right)+c_2z\Phi\!\left(p+\frac 12,\frac 32;z^2\right),
\end{equation}
for all $z\in\mathcal S_\theta$.

Set $z=\mathrm e^{-\mathrm i\frac \theta 2}\abs{s}$, with $s\in\mathbb R^\ast$,
in \eqref{eqn:618}, and take the limit for  $s\to 0$. Since  $\mathrm e^{-\mathrm i\frac \theta 2}\abs{s}\in\mathcal S_\theta$,
thanks to Lemma \ref{lem:9} we get
\begin{equation}\label{eqn:607}
c_1=\int_0^{\infty} t^{p-1}
\left(1+\mathrm e^{\mathrm i \theta} t\right )^{-\left(p+\frac 12\right)}\,\mathrm d t=
\mathrm e^{-\mathrm i p\theta}\,\frac {\Gamma(p)\Gamma\!\left(\frac 12\right)}{\Gamma\!\left(p+\frac 12\right)}
=\sqrt\pi \mathrm e^{-\mathrm i p\theta}\,\frac {\Gamma(p)}{\Gamma\!\left(p+\frac 12\right)}.
\end{equation}
Now we compute $c_2$. Differentiate \eqref{eqn:618} with respect to $z$, set $z=\mathrm e^{-\mathrm i\frac \theta 2}\abs{s}$, with $s\in\mathbb R^\ast$, and take the limit for $s\to 0$.
We get
\begin{equation}\label{eqn:650}\begin{split}
c_2&=- 2\mathrm e^{\mathrm i\frac \theta 2}
\lim_{s\to 0} \abs{s}\int_0^{\infty} t^p\left(1+\mathrm e^{\mathrm i\theta}t\right)^{-\left(p+\frac 12\right)}
\mathrm e^{-s^2 t}\,\mathrm d t
\\&=- 2\mathrm e^{\mathrm i\frac \theta 2}
\lim_{s\to 0} \abs{s}\int_0^{\infty} \left(\frac {t}{s^2}\right)^p\left(1+\mathrm e^{\mathrm i\theta}\frac t{s^2}\right)^{-\left(p+\frac 12\right)}
\mathrm e^{-t}\,\frac {\mathrm d t}{s^2}
\\&=- 2\mathrm e^{\mathrm i\frac \theta 2}
\int_0^{\infty} t^{-\frac 12} \mathrm e^{-\mathrm i\left(p+\frac 12\right)\theta}
\mathrm e^{-t}\,\mathrm d t=- 2\mathrm e^{-\mathrm i p \theta}\Gamma\!\left(\tfrac 12\right)
=- 2\sqrt\pi \mathrm e^{-\mathrm i p \theta}.
\end{split}\end{equation}
From  \eqref{eqn:618}, \eqref{eqn:607}, and \eqref{eqn:650} we obtain
\begin{multline*}
\int_0^{\infty} t^{p-1}
\left(1+\mathrm e^{\mathrm i \theta} t\right )^{-\left(p+\frac 12\right)} \mathrm e^{-\exp(\mathrm i \theta) z^2 t}\,\mathrm d t=
\frac {\sqrt\pi}{\mathrm e^{\mathrm i p\theta}}\left\{\frac {\Gamma(p)}{\Gamma\!\left(p+\frac 12\right)}\, \Phi\!\left(p,\frac 12;z^2\right)
-2 z\Phi\!\left(p+\frac 12,\frac 32;z^2\right)\right\},
\\ \text{\upshape for all $z\in\mathcal S_\theta$},\end{multline*}
 which is equivalent to
\eqref{eqn:614}.\end{proof}

\begin{theorem}\label{thm:30}
Let $0<\epsilon<\frac \pi 2$.
For all $N\in\mathbb Z_+$ we have
\begin{multline}\label{eqn:611}
\Theta(p;z)=z^{-2p}\left\{\sum_{k=0}^N \frac{(-1)^k}{k!}\,(p)_k\left(p+\frac 12\right)_{\! k}
z^{-2k}+\mathcal O\left(\abs{z}^{-2(N+1)}\right)\right\},
\\
\text{\upshape for $\abs{z}\to\infty$,  and $\abs{\Arg z}\le \frac \pi 2 - \epsilon$.}
\end{multline}
\end{theorem}
\begin{remark}
Observe that when either $p$ or $p+\frac 12$ belong to $\mathbb Z_-$,
$\Theta(p;z)$ becomes a polynomial. So \eqref{eqn:611} holds on the whole complex plane.

Let $n\in\mathbb Z_+$. Then from \eqref{eqn:620}, \eqref{eqn:362}, \eqref{eqn:495}, and  \cite[(1.2.2)]{Lebedev}
we obtain
\begin{multline}\label{eqn:820}
\Theta(-n;z)=\frac {\sqrt \pi}{\Gamma\!\left(\frac 12-n\right)}\Phi\!\left(-n,\frac 12;z^2\right)=
\frac {\sqrt \pi}{\Gamma\!\left(\frac 12-n\right)}\sum_{k=0}^n\frac {(-n)_k}{k!\left(\frac 12\right)_k}\,z^{2k}
=\\=
z^{2n}\sum_{k=0}^n\frac {(-1)^k}{k!}\,(-n)_k\left(\frac 12-n\right)_{\!k} z^{-2k},
\end{multline}
and
\begin{multline}\label{eqn:821}
\Theta\left(-\frac 12-n;z\right)=
-\frac {2\sqrt \pi}{\Gamma\!\left(-\frac 12-n\right)}\,z\Phi\!\left(-n,\frac 32;z^2\right)=
-\frac {2\sqrt \pi}{\Gamma\!\left(-\frac 12-n\right)}\,z\sum_{k=0}^n\frac {(-n)_k}{k!\left(\frac 32\right)_k}\,z^{2k}=
\\=
z^{2n+1}\sum_{k=0}^n\frac {(-1)^k}{k!}\,(-n)_k\left(-\frac 12-n\right)_{\!k} z^{-2k}.
\end{multline}
 \end{remark}
\begin{proof}
\begin{trivlist}
\item (I) First we observe that it suffices to prove \eqref{eqn:611} for $\Re p>0$.

Using \eqref{eqn:362}, and \eqref{eqn:620}, a long, but straightforward computation shows that
\begin{equation}\label{eqn:613}
\Theta(p;z)=\left(2p+\frac 32+z^2\right)\Theta(p+1;z)-(p+1)\left(p+\frac 32\right)\Theta(p+2;z).
\end{equation}

Assume now $\Re p>-1$, and \eqref{eqn:611} true for $\Re p>0$. By \eqref{eqn:613} we obtain
\begin{align*}
&\Theta(p;z)=\left(2p+\frac 32+z^2\right)z^{-2(p+1)}
\left\{\sum_{k=0}^N \frac{(-1)^k(p+1)_k\left(p+\frac 32\right)_k}{k!}\,z^{-2k}+\mathcal O\left(\abs{z}^{-2(N+1)}\right)\right\}
\\&\qquad-(p+1)\left(p+\frac 32\right)z^{-2(p+2)}
\left\{\sum_{k=0}^N \frac{(-1)^k(p+2)_k\left(p+\frac 52\right)_k}{k!}\,z^{-2k}+\mathcal O\left(\abs{z}^{-2(N+1)}\right)\right\}=
\\&\quad
\begin{aligned}&=z^{-2p}\left(2p+\frac 32\right)
\left\{\sum_{k=1}^{N+1} \frac{(-1)^{k-1}(p+1)_{k-1}\left(p+\frac 32\right)_{k-1}}{(k-1)!}\,z^{-2k}+\mathcal O\left(\abs{z}^{-2(N+2)}\right)\right\}
\\&\qquad+z^{-2p}\left\{\sum_{k=0}^N \frac{(-1)^k(p+1)_k\left(p+\frac 32\right)_k}{k!}\,z^{-2k}+\mathcal O\left(\abs{z}^{-2(N+1)}\right)\right\}
\\&\qquad-z^{-2p}(p+1)\left(p+\frac 32\right)
\left\{\sum_{k=2}^{N+2} \frac{(-1)^{k-2}(p+2)_{k-2}\left(p+\frac 52\right)_{k-2}}{(k-2)!}\,z^{-2k}+\mathcal O\left(\abs{z}^{-2(N+3)}\right)\right\}=
\end{aligned}
\\&\quad
\begin{aligned}&=z^{-2p}+z^{-2p}\left\{2p+\frac 32-(p+1)\left(p+\frac 32\right)\right\}z^{-2}
\\&\qquad+z^{-2p}\left\{\left(-\frac {\left(2p+\frac 32\right)k}{p\left(p+\frac 12\right)}+\frac {(p+k)\left(p+\frac 12+k\right)}{p\left(p+\frac 12\right)}
-\frac {k(k-1)}{p\left(p+\frac 12\right)}\right)
\sum_{k=2}^N\frac{(-1)^k(p)_k\left(p+\frac 12\right)_k}{k!}\,z^{-2k}\right\}
\\&\qquad+z^{-2p}\mathcal O\left(\abs{z}^{-2(N+1)}\right)=
\end{aligned}
\\&\quad=z^{-2p}\left\{\sum_{k=0}^N \frac{(-1)^k(p)_k\left(p+\frac 12\right)_k}{k!}\,z^{-2k}+\mathcal O\left(\abs{z}^{-2(N+1)}\right)\right\}.
\end{align*}
This shows that \eqref{eqn:611} is true for $\Re p>-1$. By iteration we get that \eqref{eqn:611} is true for all $p\in\mathbb C$.
\item(II) Since $\left(-\frac \pi 2,\frac \pi 2\right)\subset \mathcal S_{-\frac \pi2} \cup \mathcal S_{0} \cup \mathcal S_{\frac \pi2}$ it suffices to prove \eqref{eqn:611} for
\begin{equation}\label{eqn:794}
\abs{\Arg z+\frac \theta 2}\le \frac \pi 4-\frac \epsilon 2,
\end{equation}
for all $\theta\in\left\{-\frac \pi 4,0,\frac \pi 4\right\}$.

According to (I), we may assume  $\Re p>0$.
Integrating term by term the binomial expansion
\begin{multline*}
\left(1+\mathrm e^{\mathrm i \theta} t\right)^{-\left(p+\frac 12\right)}=\sum_{k=0}^N \frac {(-1)^k\left(p+\frac 12\right)_k}{k!}\, \mathrm e^{\mathrm i k \theta}t^k+
\\+\frac {(-1)^{N+1}\left(p+\frac 12\right)_{N+1}}{N!}\,\mathrm e^{\mathrm i(N+1)\theta}t^{N+1}\int_0^1(1-s)^N\left(1+\mathrm e^{\mathrm i\theta} s t\right)^{-\left(p+N+\frac 32\right)}\,\mathrm d s,
\end{multline*}
thanks to \eqref{eqn:614}  we obtain
\begin{multline}\label{eqn:775}
\Theta(p;z)=\sum_{k=0}^N \frac {(-1)^k\left(p+\frac 12\right)_k}{\Gamma(p) k!}\,\mathrm e^{\mathrm i(p+k)\theta}
\int_0^{\infty}t^{p+k-1}\mathrm e^{-\exp(\mathrm i\theta) z^2 t}\,\mathrm d t+ \\
+\frac {(-1)^{N+1}\left(p+\frac 12\right)_{N+1}}{\Gamma(p)N!}\,\mathrm e^{\mathrm i (p+N+1)\theta}
\int_0^{\infty}\left(\int_0^1(1-s)^N\left(1+\mathrm e^{\mathrm i\theta} s t\right)^{-\left(p+N+\frac 32\right)}\,\mathrm d s\right)
t^{p+N}\mathrm e^{-\exp(\mathrm i \theta) z^2 t}\,\mathrm d t.
\end{multline}

Thanks to  Lemma \ref{lem:10},
we have
\begin{equation}\label{eqn:776}
\frac {\mathrm e^{\mathrm i(p+k)\theta}}{\Gamma(p)}\,
\int_0^{\infty}t^{p+k-1}\mathrm e^{-\exp(\mathrm i\theta) z^2 t}\,\mathrm d t
=\frac {\Gamma(p+k)}{\Gamma(p)}\,(z^2)^{-(p+k)}=(p)_k z^{-2(p+k)}=(p)_k z^{-2p} z^{-2k}.
\end{equation}
Moreover
\begin{equation*}
\abs{1+\mathrm e^{\mathrm i\theta} st}^2=1+2(\cos\theta)st+s^2t^2\ge 1,\qquad \text{\upshape for $\abs{\theta}\le \frac \pi 2$}.
\end{equation*}
Then we have
\begin{equation}\label{eqn:615}
\abs{\int_0^1(1-s)^N\left(1+\mathrm e^{\mathrm i\theta} s t\right)^{-\left(p+N+\frac 32\right)}\,\mathrm d s}\le
\int_0^1 (1-s)^N \mathrm e^{(\Im p)\Arg\left(1+\mathrm e^{\mathrm i\theta} st\right)}\,\mathrm d s\le \frac {\mathrm e^{\abs{\Im p}\pi}}{N+1}.
\end{equation}
On the other hand from \eqref{eqn:794} we obtain
\begin{equation*}
\Re\left(\mathrm e^{\mathrm i \theta}z^2\right)=\cos\left(\Arg (z^2)+ \theta\right)\abs{z^2}\ge \cos\left(\frac \pi 2-\epsilon\right)\abs{z^2}= (\sin\epsilon)\abs{z}^2.
\end{equation*}
Then \eqref{eqn:615} implies that
\begin{multline}\label{eqn:609}\begin{aligned}
&\abs{\int_0^{\infty}\left(\int_0^1(1-s)^N\left(1+\mathrm e^{\mathrm i\theta} s t\right)^{-\left(p+N+\frac 32\right)}\,\mathrm d s\right)
t^{p+N}\mathrm e^{-\exp(\mathrm i \theta) z^2 t}\,\mathrm d t}\le
\\ &\qquad \le
\frac {\mathrm e^{\abs{\Im p}\pi}}{N+1}
\abs{\int_0^{\infty}
t^{\Re p+N}\mathrm e^{-\Re\left(\exp(\mathrm i \theta z^2\right) t}\,\mathrm d t}
\le \frac {\mathrm e^{\abs{\Im p}\pi}}{N+1}
\abs{\int_0^{\infty}
t^{\Re p+N}\mathrm e^{-(\sin\epsilon) \abs{z}^2 t}\,\mathrm d t}
\\&\qquad
=\frac {\mathrm e^{\abs{\Im p}\pi}}{N+1}\,\bigl((\sin\epsilon) \abs{z}^2\bigr)^{-(\Re p+N+1)}
\abs{\int_0^{\infty}
s^{\Re p+N}\mathrm e^{-s}\,\mathrm d s}
\\&\qquad\le\frac {\mathrm e^{\abs{ p}\pi}\Gamma(\Re p+N+1)}{(N+1)(\sin\epsilon)^{\Re p+N+1}}\,\abs{z}^{-2(\Re p+N+1)},
\end{aligned}
\\
\text{\upshape for $\abs{\Arg z+\frac \theta 2}\le \frac \pi 4-\frac \epsilon 2$, and $\theta\in\left\{-\frac \pi 4,0,\frac \pi 4\right\}$}.
\end{multline}
In conclusion,  the expansion \eqref{eqn:611} follows from   \eqref{eqn:775}, \eqref{eqn:776}, and \eqref{eqn:609}. \qedhere
\end{trivlist}\end{proof}

\section{Asymptotic expansions of the general solution to   Hermite-Weber equation.}\label{sec:6}

Let $w_1$, and $w_2$ be the solutions to equation \eqref{eqn:502} given by \eqref{eqn:520}.
\begin{proposition}\label{pro:57}
We have the following identities (recall that $\frac 1\Gamma$ extends to an entire function):
\begin{align}\label{eqn:803}
&\frac {w_1(z)}{\Gamma\left(\frac{3-\lambda}4\right)}\mp \frac {2w_2(z)}{\Gamma\left(\frac{1-\lambda}4\right)}=
\frac {\mathrm e^{-\frac 12 z^2}}{\sqrt \pi}\Theta\left(\frac {1-\lambda}4;\pm z\right), \\
\label{eqn:804}
&\frac {w_1(z)}{\Gamma\left(\frac{3+\lambda}4\right)}\pm \frac {2\mathrm i w_2(z)}{\Gamma\left(\frac{1+\lambda}4\right)}=
\frac {\mathrm e^{\frac 12 z^2}}{\sqrt \pi}\Theta\left(\frac {1+\lambda}4;\mp \mathrm i  z\right).
\end{align}
and
\begin{align}\label{eqn:788}
&w_1(z)=\sqrt \pi\, \mathrm e^{-\mathrm i \frac {1+\lambda}4 \pi}\left\{\frac {\mathrm i \mathrm e^{-\frac 12 z^2}}{\Gamma\!\left(\frac {1+\lambda}4\right)}\,
\Theta\!\left(\frac {1-\lambda}4;\pm z\right)+
\frac {\mathrm e^{\frac 12 z^2}}{\Gamma\!\left(\frac {1-\lambda}4\right)}\,\Theta\!\left(\frac {1+\lambda}4;\mp\mathrm i z\right)\right\},
\\\label{eqn:789}
&w_2(z)=\mp\frac {\sqrt \pi} {2}\, \mathrm e^{-\mathrm i \frac {1+\lambda}4 \pi}
\left\{\frac {\mathrm e^{-\frac 12 z^2}}{\Gamma\!\left(\frac {3+\lambda}4\right)}\,
\Theta\!\left(\frac {1-\lambda}4;\pm z\right)-
\frac {\mathrm e^{\frac 12 z^2}}{\Gamma\!\left(\frac {3-\lambda}4\right)}\,\Theta\!\left(\frac {1+\lambda}4;\mp\mathrm i z\right)\right\}.
\end{align}
\end{proposition}
\begin{proof}
From \eqref{eqn:620}, and Proposition \ref{pro:58} we have
\begin{equation*}\begin{split}
\Theta(p;\mp \mathrm i z)&=\sqrt \pi\left\{\frac {1}{\Gamma\!\left(\frac 12+p\right)}\,\Phi\!\left(p,\frac 12;-z^2\right)\pm
\frac {2\mathrm i z}{\Gamma\!\left(p\right)}\,\Phi\!\left(\frac 12+p,\frac 32;-z^2\right)\right\}
\\
&=\sqrt \pi\left\{\frac {\mathrm e^{- z^2}}{\Gamma\!\left(\frac 12+p\right)}\,\Phi\!\left(\frac 12-p,\frac 12;z^2\right)\pm
\frac {2\mathrm i\mathrm e^{- z^2}}{\Gamma(p)}\,z\Phi\!\left(1-p,\frac 32;z^2\right)\right\}.
\end{split}\end{equation*}
This identity can be rewritten as
\begin{equation}\label{eqn:786}
\Theta\!\left(\frac 12- p;\mp\mathrm i z\right)
=\sqrt \pi\left\{\frac {\mathrm e^{- z^2}}{\Gamma(1-p)}\,\Phi\!\left(p,\frac 12;z^2\right)\pm
\frac {2\mathrm i z\mathrm e^{- z^2}}{\Gamma\!\left(\frac 12-p\right)}\,\Phi\!\left(\frac 12+p,\frac 32;z^2\right)\right\}.
\end{equation}
Then from \eqref{eqn:620}, and \eqref{eqn:786},  we obtain
\begin{align}\label{eqn:787}
&\frac {1}{\Gamma\left(\frac 12+p\right)}\,\Phi\!\left(p,\frac 12;z^2\right)\mp
\frac {2}{\Gamma\left(p\right)}\,z\Phi\!\left(\frac 12+p,\frac 32;z^2\right)=
\frac {1}{\sqrt \pi}\,\Theta(p; \pm z),
\\
\label{eqn:805}
&\frac {1}{\Gamma(1-p)}\,\Phi\!\left(p,\frac 12;z^2\right)\pm
\frac {2\mathrm i}{\Gamma\left(\frac 12-p\right)}\,z\Phi\!\left(\frac 12+p,\frac 32;z^2\right)=
\frac {\mathrm e^{z^2}}{\sqrt \pi}\,\Theta\left(\frac 12-p;\mp\mathrm i z\right).\end{align}
Letting $p=\frac {1-\lambda}4$ in \eqref{eqn:787} and \eqref{eqn:805}, and using \eqref{eqn:520}, we obtain \eqref{eqn:803}, and \eqref{eqn:804}.

From \cite[(1.2.2)]{Lebedev} we get
\begin{equation}\label{eqn:822}
\frac {1}{\Gamma\!\left(\frac {1-\lambda} 4\right)\Gamma\!\left(\frac {3+\lambda}4\right)}
\pm\frac {\mathrm i}{\Gamma\!\left(\frac {1+\lambda}4\right)\Gamma\!\left(\frac {3-\lambda}4\right)}=
\frac {\mathrm e^{\pm\mathrm i \frac {1+\lambda}4 \pi}}{\pi}.
\end{equation}
Using this identity, we can solve the
system given by \eqref{eqn:803}, and \eqref{eqn:804}, obtaining
\eqref{eqn:788}, and \eqref{eqn:789}.
 \end{proof}

\begin{proposition}\label{pro:59}
Let $0<\epsilon<\frac \pi 4$.  For all $c_1,\,c_2\in\mathbb C$, and $N\in\mathbb Z_+$.
\begin{Mylist}
\item\label{itm:23} If
$\frac {c_1}{\Gamma\left(\frac {1-\lambda}4\right)}\pm \frac {c_2}{2\Gamma\left(\frac {3-\lambda}4\right)}\ne 0$, we have
\begin{align*}
&c_1w_1(z)+ c_2w_2(z)=
\\ &\qquad=\sqrt \pi \mathrm e^{\frac 12 z^2} z^{-\frac {1+\lambda}2}
\left\{\left(\frac {c_1}{\Gamma\!\left(\frac {1-\lambda}4\right)}+ \frac {c_2}{2\Gamma\!\left(\frac {3-\lambda}4\right)}\right)
\sum_{k=0}^N\frac 1{k!}\left(\frac {1+\lambda}4\right)_{\!k}\left(\frac {3+\lambda}4\right)_{\!k} z^{-2k}+
\mathcal O\left(\abs{z}^{-2(N+1)}\right)\right\},
\\&\hspace{8.5cm}\text{\upshape for $\abs{z}\to\infty$, and $\abs{\Arg  z}\le \frac \pi 4-\epsilon$},
\intertext{\upshape and}
&c_1w_1(z)+ c_2w_2(z)=
\\ &\qquad=\sqrt \pi \mathrm e^{\frac 12 z^2}(- z)^{-\frac {1+\lambda}2}
\left\{\left(\frac {c_1}{\Gamma\!\left(\frac {1-\lambda}4\right)}- \frac {c_2}{2\Gamma\!\left(\frac {3-\lambda}4\right)}\right)
\sum_{k=0}^N\frac 1{k!}\left(\frac {1+\lambda}4\right)_{\!k}\left(\frac {3+\lambda}4\right)_{\!k} z^{-2k}+
\mathcal O\left(\abs{z}^{-2(N+1)}\right)\right\},
\\&\hspace{8.5cm}\text{\upshape for $\abs{z}\to\infty$, and $\abs{\Arg (- z)}\le \frac \pi 4-\epsilon$}.
\end{align*}
\item\label{itm:24} If
\begin{equation}\label{eqn:823}
c_1=\frac {c}{\Gamma\!\left(\frac {3-\lambda}4\right)},\qquad
c_2=-\frac {2c}{\Gamma\!\left(\frac {1-\lambda}4\right)},
\end{equation}
with $c\ne 0$, and $\lambda\notin\{1+2n:n\in\mathbb Z_+\}$, we have
\begin{align*}
&c_1w_1(z)+ c_2w_2(z)=
\\ &\qquad=
\frac {c}{\sqrt \pi}\,\mathrm e^{-\frac 12 z^2}  z^{-\frac {1-\lambda}2}
\left\{\sum_{k=0}^N\frac {(-1)^k}{k!}\left(\frac {1-\lambda}4\right)_{\! k}\left(\frac {3-\lambda}4\right)_{\! k} z^{-2k}+
\mathcal O\left(\abs{z}^{-2(N+1)}\right)\right\},
\\&\hspace{8.5cm}\text{\upshape for $\abs{z}\to\infty$, and $\abs{\Arg  z}\le \frac \pi 4-\epsilon$},
\\
&c_1w_1(z)+ c_2w_2(z)=
\\ &\qquad=
\sqrt \pi \mathrm e^{\frac 12 z^2}(- z)^{-\frac {1+\lambda}2}
\left\{\frac {2c}{\Gamma\!\left(\frac {1-\lambda}4\right)\Gamma\!\left(\frac {3-\lambda}4\right)}
\sum_{k=0}^N\frac 1{k!}\left(\frac {1+\lambda}4\right)_{\!k}\left(\frac {3+\lambda}4\right)_{\!k} z^{-2k}+
\mathcal O\left(\abs{z}^{-2(N+1)}\right)\right\},
\\&\hspace{8.5cm}\text{\upshape for $\abs{z}\to\infty$, and $\abs{\Arg (- z)}\le \frac \pi 4-\epsilon$},
\end{align*}
\item\label{itm:36} If
\begin{equation}\label{eqn:824}
c_1=\frac {c}{\Gamma\!\left(\frac {3-\lambda}4\right)},\qquad
c_2=\frac {2c}{\Gamma\!\left(\frac {1-\lambda}4\right)},
\end{equation}
with $c\ne 0$, and
$\lambda\notin\{1+2n:n\in\mathbb Z_+\}$, we have
\begin{align*}
&c_1w_1(z)+ c_2w_2(z)=
\\ &\qquad=\sqrt \pi \mathrm e^{\frac 12 z^2} z^{-\frac {1+\lambda}2}
\left\{\frac {2c}{\Gamma\!\left(\frac {1-\lambda}4\right)\Gamma\!\left(\frac {3-\lambda}4\right)}
\sum_{k=0}^N\frac 1{k!}\left(\frac {1+\lambda}4\right)_{\!k}\left(\frac {3+\lambda}4\right)_{\!k} z^{-2k}+
\mathcal O\left(\abs{z}^{-2(N+1)}\right)\right\},
\\&\hspace{8.5cm}\text{\upshape for $\abs{z}\to\infty$, and $\abs{\Arg  z}\le \frac \pi 4-\epsilon$},
\\
&c_1w_1(z)+ c_2w_2(z)=
\\ &\qquad=\frac {c}{\sqrt \pi}\,\mathrm e^{-\frac 12 z^2} (- z)^{-\frac {1-\lambda}2}
\left\{\sum_{k=0}^N\frac {(-1)^k}{k!}\left(\frac {1-\lambda}4\right)_{\! k}\left(\frac {3-\lambda}4\right)_{\! k} z^{-2k}+
\mathcal O\left(\abs{z}^{-2(N+1)}\right)\right\}.
\\&\hspace{8.5cm}\text{\upshape for $\abs{z}\to\infty$, and $\abs{\Arg (- z)}\le \frac \pi 4-\epsilon$}.
\end{align*}
\item\label{itm:37} If
\begin{equation*}
c_1=\frac {c}{\Gamma\!\left(\frac {3-\lambda}4\right)},\qquad
c_2=\mp\frac {2c}{\Gamma\!\left(\frac {1-\lambda}4\right)},
\end{equation*}
with $c\ne 0$, and $\lambda=1+4n$, with $n\in\mathbb Z_+$,
we have
\begin{equation*}
c_1w_1(z)+ c_2w_2(z)=
\frac {c}{\sqrt\pi}\,\mathrm e^{-\frac 12\,z^2}z^{2n}\sum_{k=0}^n\frac {(-1)^k}{k!}\,(-n)_k\left(\frac 12-n\right)_{\!k} z^{-2k},
\qquad\text{\upshape for all $z$}.
\end{equation*}
\item\label{itm:38} If \begin{equation*}
c_1=\frac {c}{\Gamma\!\left(\frac {3-\lambda}4\right)},\qquad
c_2=\mp\frac {2c}{\Gamma\!\left(\frac {1-\lambda}4\right)},
\end{equation*}
with $c\ne 0$, and $\lambda=3+4n$, with $n\in\mathbb Z_+$,
we have
\begin{equation*}
c_1w_1(z)+ c_2w_2(z)=\pm\frac {c}{\sqrt\pi}
\mathrm e^{-\frac 12\,z^2}z^{2n+1}\sum_{k=0}^n\frac {(-1)^k}{k!}\,(-n)_k\left(-\frac 12-n\right)_{\!k}\,z^{-2k}\qquad\text{\upshape for all $z$}.
\end{equation*}
\end{Mylist}
\end{proposition}
\begin{proof}
\eqref{itm:23}  follows from \eqref{eqn:520}, and Theorem \ref{thm:31}, with $p=\frac {1-\lambda}4$.
Observe that
\begin{equation*}
\abs{\Arg z}\le \frac \pi 4- \epsilon  \implies \abs{\Arg (z^2)}\le \frac \pi 2-2\epsilon,
\end{equation*}
and that
\begin{equation*}
(z^2)^p=\begin{cases} z^{2p},&\text{\upshape if $-\frac \pi 4<\Arg z\le \frac \pi 4$}, \\
(-z)^{2p},&\text{\upshape if $-\frac \pi 4<\Arg (-z)\le \frac \pi 4$}.\end{cases}
\end{equation*}

From \eqref{eqn:823}, \eqref{eqn:824}, and  \eqref{eqn:803} we have
\begin{equation*}c_1w_1(z)+ c_2w_2(z)=
\frac {c}{\Gamma\!\left(\frac {3-\lambda}4\right)} w_1(z)\mp\frac {2c}{\Gamma\!\left(\frac {1-\lambda}4\right)}w_2(z)=
\frac {c}{\sqrt \pi}\,\mathrm e^{-\frac 12 z^2}\Theta\!\left(\frac {1-\lambda}4;\pm z\right).
\end{equation*}
Then  \eqref{itm:24} and  \eqref{itm:36}  follow from  Theorem \ref{thm:30}  with $p=\frac {1-\lambda}4$; while
 \eqref{itm:37} and  \eqref{itm:38}  follow from \eqref{eqn:820}, and \eqref{eqn:821}.
\end{proof}

\begin{proposition}\label{pro:60}
Let $0<\epsilon<\frac \pi 8$.  For all $c_1,\,c_2\in\mathbb C$, and $N\in\mathbb Z_+$.
\begin{Mylist}
\item\label{itm:25}  If $\left(\frac {\mathrm i c_1}{\Gamma\left(\frac {1+\lambda}4\right)}\mp \frac {c_2}{2\Gamma\left(\frac {3+\lambda}4\right)}\right)
\left(\frac {c_1}{\Gamma\left(\frac {1-\lambda}4\right)}\pm \frac {c_2}{2\Gamma\left(\frac {3-\lambda}4\right)}\right)\ne 0$, we have
\begin{align*}
&c_1w_1(z)+ c_2w_2(z)=
\sqrt \pi \mathrm e ^{-\mathrm i \frac {1+\lambda}{4} \pi}\cdot
\\ &\quad\cdot \begin{aligned}[t]
&\left\{\mathrm e^{-\frac 12 z^2} z^{-\frac {1-\lambda}2}
\left[\left(\frac {\mathrm i c_1}{\Gamma\!\left(\frac {1+\lambda}4\right)}- \frac {c_2}{2\Gamma\!\left(\frac {3+\lambda}4\right)}\right)
\sum_{k=0}^N\frac {(-1)^k}{k!}\left(\frac {1-\lambda}4\right)_{\!k}\left(\frac {3-\lambda}4\right)_{\!k} z^{-2k}+
\mathcal O\left(\abs{z}^{-2(N+1)}\right)\right]+\right.
\\&\quad
\left.+\mathrm e^{\frac 12 z^2}(- \mathrm i z)^{-\frac {1+\lambda}2}
\left[\left(\frac {c_1}{\Gamma\!\left(\frac {1-\lambda}4\right)}+ \frac {c_2}{2\Gamma\!\left(\frac {3-\lambda}4\right)}\right)
\sum_{k=0}^N\frac 1{k!}\left(\frac {1+\lambda}4\right)_{\!k}\left(\frac {3+\lambda}4\right)_{\!k} z^{-2k}+
\mathcal O\left(\abs{z}^{-2(N+1)}\right)\right]\right\},
\end{aligned}
\\ &\hspace{8cm}\text{\upshape for $\abs{z}\to\infty$, and $\abs{\Arg (z)-\frac \pi 4}\le \epsilon$},
\\ &c_1w_1(z)+ c_2w_2(z)=
\sqrt \pi \mathrm e ^{-\mathrm i \frac {1+\lambda}{4} \pi}\cdot
\\ &\quad\cdot \begin{aligned}[t]
&\left\{\mathrm e^{-\frac 12 z^2}(- z)^{-\frac {1-\lambda}2}
\left[\left(\frac {\mathrm i c_1}{\Gamma\!\left(\frac {1+\lambda}4\right)}+ \frac {c_2}{2\Gamma\!\left(\frac {3+\lambda}4\right)}\right)
\sum_{k=0}^N\frac {(-1)^k}{k!}\left(\frac {1-\lambda}4\right)_{\!k}\left(\frac {3-\lambda}4\right)_{\!k} z^{-2k}+
\mathcal O\left(\abs{z}^{-2(N+1)}\right)\right]+\right.
\\&\quad
\left.+\mathrm e^{\frac 12 z^2}(\mathrm i z)^{-\frac {1+\lambda}2}
\left[\left(\frac {c_1}{\Gamma\!\left(\frac {1-\lambda}4\right)}- \frac {c_2}{2\Gamma\!\left(\frac {3-\lambda}4\right)}\right)
\sum_{k=0}^N\frac 1{k!}\left(\frac {1+\lambda}4\right)_{\!k}\left(\frac {3+\lambda}4\right)_{\!k} z^{-2k}+
\mathcal O\left(\abs{z}^{-2(N+1)}\right)\right]\right\},
\end{aligned}
\\ &\hspace{8cm}\text{\upshape for $\abs{z}\to\infty$, and $\abs{\Arg(- z)-\frac \pi 4}\le \epsilon$},
\end{align*}
\item\label{itm:26} If
\begin{equation}\label{eqn:825}
c_1=\frac {c}{\Gamma\!\left(\frac {3+\lambda}4\right)},\qquad c_2=\frac {2\mathrm i c}{\Gamma\!\left(\frac {1+\lambda}4\right)},
\end{equation}
with $c\ne 0$, and $\lambda\notin\{-(1+2n):n\in\mathbb Z_+\}$, we have
\begin{align*}
&c_1w_1(z)+ c_2w_2(z)=
\frac {c}{\sqrt \pi}\,
\mathrm e^{\frac 12 z^2}(-\mathrm i z)^{-\frac {1+\lambda}2}
\left\{\sum_{k=0}^N\frac 1{k!}\left(\frac {1+\lambda}4\right)_{\!k}\left(\frac {3+\lambda}4\right)_{\!k} z^{-2k}+
\mathcal O\left(\abs{z}^{-2(N+1)}\right)\right\}.
\\ &\hspace{8cm}\text{\upshape for $\abs{z}\to\infty$, and $\abs{\Arg (z)-\frac \pi 4}\le \epsilon$},
\\ &c_1w_1(z)+ c_2w_2(z)=
\sqrt \pi \mathrm e ^{-\mathrm i \frac {1+\lambda}{4} \pi}\cdot
\\ &\quad\cdot \begin{aligned}[t]
&\left\{\mathrm e^{-\frac 12 z^2}(- z)^{-\frac {1-\lambda}2}
\left[\frac {\mathrm 2 \mathrm i c}{\Gamma\!\left(\frac {1+\lambda}4\right)\Gamma\!\left(\frac {3+\lambda}4\right)}
\sum_{k=0}^N\frac {(-1)^k}{k!}\left(\frac {1-\lambda}4\right)_{\!k}\left(\frac {3-\lambda}4\right)_{\!k} z^{-2k}+
\mathcal O\left(\abs{z}^{-2(N+1)}\right)\right]+\right.
\\&\quad
\left.+\mathrm e^{\frac 12 z^2}(\mathrm i z)^{-\frac {1+\lambda}2}
\left[\frac {c\,\mathrm e^{-\mathrm i \frac {1+\lambda}{4}\pi}}{\pi}
\sum_{k=0}^N\frac 1{k!}\left(\frac {1+\lambda}4\right)_{\!k}\left(\frac {3+\lambda}4\right)_{\!k} z^{-2k}+
\mathcal O\left(\abs{z}^{-2(N+1)}\right)\right]\right\},
\end{aligned}
\\ &\hspace{8cm}\text{\upshape for $\abs{z}\to\infty$, and $\abs{\Arg(- z)-\frac \pi 4}\le \epsilon$}.
\end{align*}
\item\label{itm:39} If
\begin{equation}\label{eqn:826}
c_1=\frac {c}{\Gamma\!\left(\frac {3+\lambda}4\right)},\qquad c_2=-\frac {2\mathrm i c}{\Gamma\!\left(\frac {1+\lambda}4\right)},
\end{equation}
with $c\ne 0$,  and $\lambda\notin\{-(1+2n):n\in\mathbb Z_+\}$,
we have
\begin{align*}
&c_1w_1(z)+ c_2w_2(z)=
\sqrt \pi \mathrm e ^{-\mathrm i \frac {1+\lambda}{4} \pi}\cdot
\\ &\quad\cdot \begin{aligned}[t]
&\left\{\mathrm e^{-\frac 12 z^2} z^{-\frac {1-\lambda}2}
\left[\frac {2\mathrm i c}{\Gamma\!\left(\frac {1+\lambda}4\right)\Gamma\!\left(\frac {3+\lambda}4\right)}
\sum_{k=0}^N\frac {(-1)^k}{k!}\left(\frac {1-\lambda}4\right)_{\!k}\left(\frac {3-\lambda}4\right)_{\!k} z^{-2k}+
\mathcal O\left(\abs{z}^{-2(N+1)}\right)\right]+\right.
\\&\quad
\left.+\mathrm e^{\frac 12 z^2}(- \mathrm i z)^{-\frac {1+\lambda}2}
\left[\frac {c\,\mathrm e^{-\mathrm i\frac {1+\lambda}{4}\pi}}\pi
\sum_{k=0}^N\frac 1{k!}\left(\frac {1+\lambda}4\right)_{\!k}\left(\frac {3+\lambda}4\right)_{\!k} z^{-2k}+
\mathcal O\left(\abs{z}^{-2(N+1)}\right)\right]\right\},
\end{aligned}
\\ &\hspace{8cm}\text{\upshape for $\abs{z}\to\infty$, and $\abs{\Arg (z)-\frac \pi 4}\le \epsilon$},
\\ &c_1w_1(z)+ c_2w_2(z)=
\frac {c}{\sqrt \pi}\,
\mathrm e^{\frac 12 z^2}(\mathrm i z)^{-\frac {1+\lambda}2}
\left\{\sum_{k=0}^N\frac 1{k!}\left(\frac {1+\lambda}4\right)_{\!k}\left(\frac {3+\lambda}4\right)_{\!k} z^{-2k}+
\mathcal O\left(\abs{z}^{-2(N+1)}\right)\right\},
\\ &\hspace{8cm}\text{\upshape for $\abs{z}\to\infty$, and $\abs{\Arg(- z)-\frac \pi 4}\le \epsilon$}.
\end{align*}
\item\label{itm:41}  If
\begin{equation*}
c_1=\frac {c}{\Gamma\!\left(\frac {3+\lambda}4\right)},\qquad c_2=\pm\frac {2\mathrm i c}{\Gamma\!\left(\frac {1+\lambda}4\right)},
\end{equation*}
with $c\ne 0$,  and
$\lambda=-(1+4n)$, with $n\in\mathbb Z_+$,
  we have
\begin{equation*}
c_1w_1(z)+c_2w_2(z)=
\frac {(-1)^nc}{\sqrt \pi}\,\mathrm e^{\frac 12 z^2}
 z^{2n}\sum_{k=0}^n\frac {1}{k!}\,(-n)_k\left(\frac 12-n\right)_{\!k} z^{-2k},\qquad\text{\upshape for  all $z$}.
\end{equation*}
\item\label{itm:42}  If
\begin{equation*}
c_1=\frac {c}{\Gamma\!\left(\frac {3+\lambda}4\right)},\qquad c_2=\pm\frac {2\mathrm i c}{\Gamma\!\left(\frac {1+\lambda}4\right)},
\end{equation*}
with $c\ne 0$,  and
$\lambda=-(3+4n)$, with $n\in\mathbb Z_+$,
  we have
\begin{equation*}
c_1w_1(z)+c_2w_2(z)=
\mp\mathrm i \,\frac {(-1)^n c}{\sqrt \pi}\,\mathrm e^{\frac 12 z^2}
 z^{2n+1}\sum_{k=0}^n\frac {1}{k!}\,(-n)_k\left(-\frac 12-n\right)_{\!k} z^{-2k},\qquad\text{\upshape for  all $z$}.
\end{equation*}
\item\label{itm:27}  If  \begin{equation}\label{eqn:827}
c_1=\frac {c}{\Gamma\!\left(\frac {3-\lambda}4\right)},\qquad
c_2=-\frac {2c}{\Gamma\!\left(\frac {1-\lambda}4\right)},
\end{equation}
with $c\ne 0$, and
$\lambda\notin\{1+2n:n\in\mathbb Z_+\}$,
  we have
\begin{align*}
&c_1w_1(z)+ c_2w_2(z)=
\frac {c}{\sqrt \pi}\,
\mathrm e^{-\frac 12 z^2} z^{-\frac {1-\lambda}2}
\left\{\sum_{k=0}^N\frac {(-1)^k}{k!}\left(\frac {1-\lambda}4\right)_{\!k}\left(\frac {3-\lambda}4\right)_{\!k} z^{-2k}+
\mathcal O\left(\abs{z}^{-2(N+1)}\right)\right\},
\\ &\hspace{8cm}\text{\upshape for $\abs{z}\to\infty$, and $\abs{\Arg (z)-\frac \pi 4}\le \epsilon$},
\\ &c_1w_1(z)+ c_2w_2(z)=
\sqrt \pi \mathrm e ^{-\mathrm i \frac {1+\lambda}{4} \pi}\cdot
\\ &\quad\cdot \begin{aligned}[t]
&\left\{\mathrm e^{-\frac 12 z^2}(- z)^{-\frac {1-\lambda}2}
\left[-\frac {c\, \mathrm e^{-\mathrm i\frac {1+\lambda}{4}\pi}}\pi
\sum_{k=0}^N\frac {(-1)^k}{k!}\left(\frac {1-\lambda}4\right)_{\!k}\left(\frac {3-\lambda}4\right)_{\!k} z^{-2k}+
\mathcal O\left(\abs{z}^{-2(N+1)}\right)\right]+\right.
\\&\quad
\left.+\mathrm e^{\frac 12 z^2}(\mathrm i z)^{-\frac {1+\lambda}2}
\left[\frac {2c}{\Gamma\!\left(\frac {1-\lambda}4\right)\Gamma\!\left(\frac {3-\lambda}4\right)}
\sum_{k=0}^N\frac 1{k!}\left(\frac {1+\lambda}4\right)_{\!k}\left(\frac {3+\lambda}4\right)_{\!k} z^{-2k}+
\mathcal O\left(\abs{z}^{-2(N+1)}\right)\right]\right\},
\end{aligned}
\\ &\hspace{8cm}\text{\upshape for $\abs{z}\to\infty$, and $\abs{\Arg(- z)-\frac \pi 4}\le \epsilon$},
\end{align*}
\item\label{itm:40}  If
\begin{equation}\label{eqn:828}
c_1=\frac {c}{\Gamma\!\left(\frac {3-\lambda}4\right)},\qquad
c_2=\frac {2c}{\Gamma\!\left(\frac {1-\lambda}4\right)},
\end{equation}
with $c\ne 0$, and
$\lambda\notin\{1+2n:n\in\mathbb Z_+\}$,
  we have
\begin{align*}
&c_1w_1(z)+ c_2w_2(z)=
\sqrt \pi \mathrm e ^{-\mathrm i \frac {1+\lambda}{4} \pi}\cdot
\\ &\quad\cdot \begin{aligned}[t]
&\left\{\mathrm e^{-\frac 12 z^2} z^{-\frac {1-\lambda}2}
\left[-\frac {c\,\mathrm e^{-\mathrm i\frac {1+\lambda}{4}\pi}}{\pi}
\sum_{k=0}^N\frac {(-1)^k}{k!}\left(\frac {1-\lambda}4\right)_{\!k}\left(\frac {3-\lambda}4\right)_{\!k} z^{-2k}+
\mathcal O\left(\abs{z}^{-2(N+1)}\right)\right]+\right.
\\&\quad
\left.+\mathrm e^{\frac 12 z^2}(- \mathrm i z)^{-\frac {1+\lambda}2}
\left[\frac {2c}{\Gamma\!\left(\frac {1-\lambda}4\right)\Gamma\!\left(\frac {3-\lambda}4\right)}
\sum_{k=0}^N\frac 1{k!}\left(\frac {1+\lambda}4\right)_{\!k}\left(\frac {3+\lambda}4\right)_{\!k} z^{-2k}+
\mathcal O\left(\abs{z}^{-2(N+1)}\right)\right]\right\},
\end{aligned}
\\ &\hspace{8cm}\text{\upshape for $\abs{z}\to\infty$, and $\abs{\Arg (z)-\frac \pi 4}\le \epsilon$},
\\ &c_1w_1(z)+ c_2w_2(z)=
\frac {c}{\sqrt \pi}\,
\mathrm e^{-\frac 12 z^2}(- z)^{-\frac {1-\lambda}2}
\left\{\sum_{k=0}^N\frac {(-1)^k}{k!}\left(\frac {1-\lambda}4\right)_{\!k}\left(\frac {3-\lambda}4\right)_{\!k} z^{-2k}+
\mathcal O\left(\abs{z}^{-2(N+1)}\right)\right\},
\\ &\hspace{8cm}\text{\upshape for $\abs{z}\to\infty$, and $\abs{\Arg(- z)-\frac \pi 4}\le \epsilon$},
\end{align*}
\item\label{itm:43} If
\begin{equation*}
c_1=\frac {c}{\Gamma\!\left(\frac {3-\lambda}4\right)},\qquad
c_2=\mp\frac {2c}{\Gamma\!\left(\frac {1-\lambda}4\right)},
\end{equation*}
with $c\ne 0$, and
$\lambda=1+4n$, with $n\in\mathbb Z_+$,
we have
\begin{equation*}
c_1w_1(z)+ c_2w_2(z)=\frac {c}{\sqrt\pi}\,\mathrm e^{-\frac 12\,z^2}z^{2n}\sum_{k=0}^n\frac {(-1)^k}{k!}\,(-n)_k\left(\frac 12-n\right)_{\!k} z^{-2k},
\qquad\text{\upshape for all $z$}.
\end{equation*}
\item\label{itm:44} If \begin{equation*}
c_1=\frac {c}{\Gamma\!\left(\frac {3-\lambda}4\right)},\qquad
c_2=\mp\frac {2c}{\Gamma\!\left(\frac {1-\lambda}4\right)},
\end{equation*}
with $c\ne 0$, and
$\lambda=3+4n$, with $n\in\mathbb Z_+$,
we have
\begin{equation*}
c_1w_1(z)+ c_2w_2(z)=\pm\frac {c}{\sqrt\pi}
\mathrm e^{-\frac 12\,z^2}z^{2n+1}\sum_{k=0}^n\frac {(-1)^k}{k!}\,(-n)_k\left(-\frac 12-n\right)_{\!k}\,z^{-2k}\qquad\text{\upshape for all $z$}.
\end{equation*}
\end{Mylist}
\end{proposition}
\begin{proof}
In the computations we make use of identity \eqref{eqn:822}.

 \eqref{itm:25}  follows from \eqref{eqn:788}, and \eqref{eqn:789}, and  Theorem \ref{thm:30}, with $p=\frac {1\mp\lambda}4$.
Observe that
\begin{equation*}
\abs{\Arg (\pm z)- \frac \pi 4}\le \epsilon \iff
  \abs{\Arg (\mp\mathrm i z)+\frac \pi 4}\le \epsilon,
\end{equation*}
and
\begin{equation*}
\abs{\Arg (\pm z)- \frac \pi 4}\le \epsilon\implies \abs{\Arg (\pm z)}\le \frac \pi 2-\epsilon.
\end{equation*}

From \eqref{eqn:825}, \eqref{eqn:826}, and \eqref{eqn:804} we have
\begin{equation*}
c_1w_1(z)+c_2w_2(z)=c\left\{\frac {w_1(z)}{\Gamma\left(\frac{3+\lambda}4\right)}\pm \frac {2\mathrm i w_2(z)}{\Gamma\left(\frac{1+\lambda}4\right)}\right\}=
\frac {c\mathrm e^{\frac 12 z^2}}{\sqrt \pi}\Theta\left(\frac {1+\lambda}4;\mp \mathrm i  z\right).
\end{equation*}
Then  \eqref{itm:26} and  \eqref{itm:39}  follow from  Theorem \ref{thm:30}  with
$p=\frac {1+\lambda}4$; while
 \eqref{itm:41}, and  \eqref{itm:42}  follow from \eqref{eqn:820}, and \eqref{eqn:821}.

From \eqref{eqn:827}, \eqref{eqn:828}, and \eqref{eqn:803} we have
\begin{equation*}
c_1w_1(z)+c_2w_2(z)=c\left\{\frac {w_1(z)}{\Gamma\left(\frac{3-\lambda}4\right)}\mp \frac {2w_2(z)}{\Gamma\left(\frac{1-\lambda}4\right)}\right\}=
\frac {c\mathrm e^{-\frac 12 z^2}}{\sqrt \pi}\Theta\left(\frac {1-\lambda}4;\pm z\right).
\end{equation*}
Then  \eqref{itm:27} and  \eqref{itm:40}  follow from  Theorem \ref{thm:30}  with
$p=\frac {1-\lambda}4$; while
 \eqref{itm:43}, and  \eqref{itm:44}  follow from \eqref{eqn:820}, and \eqref{eqn:821}.
\end{proof}

\end{document}